\documentclass[review,hidelinks,onefignum,onetabnum]{siamart251104}
\usepackage{ulem}
\usepackage{graphicx,epstopdf} 
\usepackage[caption=false]{subfig} 
\newtheorem{assumption}{Assumption}
\usepackage{mathrsfs}
\usepackage{multirow}
\usepackage{amsfonts} 
\usepackage[utf8]{inputenc}

\usepackage{lipsum}
\usepackage{amsfonts}
\usepackage{graphicx}
\usepackage{epstopdf}
\usepackage{algorithmic}
\ifpdf
  \DeclareGraphicsExtensions{.eps,.pdf,.png,.jpg}
\else
  \DeclareGraphicsExtensions{.eps}
\fi

\newsiamremark{remark}{Remark}
\newsiamremark{hypothesis}{Hypothesis}
\crefname{hypothesis}{Hypothesis}{Hypotheses}
\newsiamthm{claim}{Claim}
\newsiamremark{fact}{Fact}
\crefname{fact}{Fact}{Facts}

\headers{$\ell_{1/2}$-Regularized QP with Assignment Constraints}{Lijun Xie, Ran Gu, Xin Liu}

\title{Convergence-Guaranteed Algorithms for  $\ell_{1/2}$-Regularized 
Quadratic Programs with Assignment Constraints\thanks{Submitted to the editors DATE.
\funding{
 The work of Ran Gu was supported in part by the National Key R\&D Program of China (2022YFA1003800),
  the National Natural Science Foundation of China (12201318), 
  the Natural Science Foundation of Tianjin (25JCJQJC00300), 
  and the Tianjin Science and Technology Program (24ZXZSSS00320). 
  The work of Xin Liu was supported in part by the National Natural Science Foundation of China (12125108, 12288201)
   and the Research Grants Council grant JLFS/P-501/24 for the CAS AMSS–PolyU Joint Laboratory in Applied Mathematics.
}}}

\author{Lijun Xie\thanks{NITFID, School of Statistics and Data Sciences, LPMC, KLMDASR and LEBPS, Nankai University
, Tianjin 300071, China
  (\email{2120230133@mail.nankai.edu.cn}).}
\and Ran Gu\thanks{NITFID, School of Statistics and Data Sciences, LPMC, KLMDASR and LEBPS, Nankai University
, Tianjin 300071, China
  (\email{rgu@nankai.edu.cn}).}
\and Xin Liu\thanks{State Key Laboratory of Mathematical Sciences, Academy of Mathematics and Systems Science, 
Chinese Academy of Sciences, Beijing 100190, China, 
and School of Mathematical Sciences, University of Chinese Academy of Sciences, 
Beijing 100049, China (\email{liuxin@lsec.cc.ac.cn}).}
}

\usepackage{amsopn}

\nolinenumbers

\ifpdf
\hypersetup{
  pdftitle={ Convergence-Guaranteed Algorithms for $\ell_{1/2}$-Regularized Quadratic Programs with Assignment Constraints},
  pdfauthor={Lijun Xie, Ran Gu, Xin Liu}
}
\fi 

\begin{document}

\maketitle
\begin{abstract}
This paper addresses a quadratic problem with assignment constraints, an NP-hard combinatorial optimization problem arisen from facility location, multiple-input multiple-output detection, and maximum mean discrepancy calculation et al. The discrete nature of the constraints precludes the use of continuous optimization algorithms. Therefore, we begin by relaxing the binary constraints into continuous box constraints and incorporate an $\ell_{1/2}$ regularization term to drive the relaxed variables toward binary values.
We prove that when the regularization parameter is larger than a threshold, the regularized problem is equivalent to the original problem: they share identical local and global minima, and all Karush-Kuhn-Tucker points of the regularized problem are feasible assignment matrices. To solve the regularized problem approximately and efficiently, we adopt the variable splitting technique, and solve it using the alternating direction method of multipliers (ADMM) framework, in which all subproblems admit closed-form solutions. Detailed theoretical analysis confirms the algorithm's convergence, and finite-step termination under certain conditions. Finally, the algorithm is validated on various numerical tests.
\end{abstract}

\begin{keywords}
 \(\ell_{1/2}\) quasi-norm regularization; ADMM; nonconvex optimization; maximum mean
discrepancy
\end{keywords}

\begin{MSCcodes}
90C26, 90C30, 65K05
\end{MSCcodes}

\section{Introduction}
In this paper, we consider a quadratic problem with generalized assignment matrices constraints formulated as:
\begin{equation}
\label{eq:ori-pro}
\min_{X \in \mathcal{F}_1 } \quad  \frac{1}{2} \langle A,\, XX^{\top} \rangle + \langle G,\, X \rangle,
\end{equation}
where $A\in \mathbb{R}^{n\times n}$ is a positive semi-definite matrix and $\mathcal{F}_1$ denotes the set of assignment matrices:
\begin{equation*}
    \mathcal{F}_1 = \{ X \in \mathbb{R}^{n\times m} \mid \mathbf{1}_n^{\top} X = b ~ \mathbf{1}_m^{\top}, X \mathbf{1}_m = \mathbf{1}_n, X_{ij} \in \{0,1\} \},
\end{equation*}
where $b = n/m$. Here, $m$ is assumed to be a divisor of $n$, ensuring $b$ is a positive integer, and $\mathbf{1}_k \in \mathbb{R}^k$ denotes the all-ones vector of dimension $k$. 

This formulation arises naturally in scenarios involving two sets of entities: a source set with \( n \) elements and a target set with \( m \) elements. An assignment matrix \( X \in \mathcal{F}_1 \) encodes a deterministic mapping between these sets, where the binary entry \( X_{ij} = 1 \) indicates that the \( i \)-th source entity is assigned to the \( j \)-th target entity. The constraints enforce two key structural properties. Each source entity is assigned to exactly one target entity (enforced by \( X \mathbf{1}_m = \mathbf{1}_n \)), ensuring a total partition of the source set.  Each target entity receives exactly \( b \) source entities (enforced by \( \mathbf{1}_n^{\top} X = b~\mathbf{1}_m^{\top} \)), reflecting a balanced distribution of the source entities across the target set.

Notably, when \( n = m \) and \( b = 1 \), this problem becomes a special case of the quadratic assignment problem, which traditionally restricts \( X \) to be a permutation matrix. By generalizing the permutation constraint to allow \( n \neq m \) and balanced multi-assignment (\( b \geq 1 \)), our formulation captures a broader class of practical assignment scenarios where entities can be grouped or distributed. Problem \cref{eq:ori-pro} finds widespread applications in real-world domains, including facility location planning \cite{cubukcuoglu2021hospitallayout}, memory layout optimization in signal processors \cite{Memory_Layout}, scheduling \cite{geoffrion1976scheduling}, multiple-input multiple-output detection \cite{zhao2021MIMO_detection}, and maximum mean discrepancy (MMD) \cite{2021-AAAI-MMD}, to name but a few. Furthermore, many discrete combinatorial optimization problems, such as the traveling salesman problem, the maximum clique problem, and graph partitioning, can be reduced to instances of problem \cref{eq:ori-pro} \cite{surveyforQAP}. 

A fundamental relaxation strategy involves softening the binary constraints \( X_{ij} \in \{0,1\} \) to continuous bounds \( 0 \leq X \leq 1 \), where \( \leq \) denotes componentwise ordering, resulting in the following problem.
\begin{equation}
\label{eq:relaxed_pro}
\min_{X \in \mathcal{F}_2 } \quad  \frac{1}{2} \langle A,\, XX^{\top} \rangle + \langle G,\, X \rangle,
\end{equation}
where  $\mathcal{F}_2$ is the transportation polytope defined as: $$   \mathcal{F}_2 = \left\{ X \in \mathbb{R}^{n \times m} \mid \mathbf{1}_n^{\top} X = b ~ \mathbf{1}_m^{\top}, \, X \mathbf{1}_m = \mathbf{1}_n, \, 0 \leq X \leq 1 \right\}.$$

By the Hoffman-Kruskal theorem \cite{schrijver2002combinatorial}, the convex hull of matrices in $\mathcal{F}_1$ coincides with the set of $\mathcal{F}_2$, thereby ensuring that the relaxation is tight for linear assignment problem. However, the convex quadratic structure of the objective function in problem \cref{eq:relaxed_pro} allows the optimal solution to lie in points that are not vertices of $\mathcal{F}_2$, thereby violating the binary constraints in $\mathcal{F}_1$. Thus, problem \cref{eq:relaxed_pro} is not equivalent to the original binary constrained problem.

It can be observed that the solution to problem \cref{eq:ori-pro} exhibits strong sparsity due to the assignment constraints.
To leverage this sparsity and enforce binary solutions for problem \cref{eq:relaxed_pro}, we add an sparse-inducing regularization term to the objective function. 
The \(\ell_0\) quasi-norm is well known for its ability to promote sparsity, however, it is discontinuous in \(X\), and thus difficult to optimize. The \(\ell_1\)-norm is often used as a good approximation of the \(\ell_0\) quasi-norm under certain conditions. Nevertheless, it fails to promote sparsity for problem \cref{eq:relaxed_pro}, since for any \(X \in \mathcal{F}_2\), \(\|X\|_1\) remains a constant \(n\). To better induce sparsity, we employ the $\ell_p$ quasi-norm (defined as \(\| X \|_p^p = \sum_{i=1}^n \sum_{j=1}^m | X_{ij} |^p\), \(0 <p <1 \)), which has demonstrated strong capability in recovering sparse solutions \cite{Chartrand_2008_lp, jiang2009notelpnphard, xu2012l_}.
Specifically, $\ell_{1/2}\text{ quasi-norm}$ has been well studied, and it admits closed-form solutions in subproblems of some proximal gradient algorithms \cite{lpjiangbo,xu2012l_}, which facilitates efficient computation. Thus we take $p =\frac{1}{2}$ in this work, leading to the problem \cref{eq:l1/2-pro}:  
\begin{equation}
\label{eq:l1/2-pro}
\min_{X \in \mathcal{F}_2} \quad  \frac{1}{2} \langle A,\, XX^{\top} \rangle + \langle G,\, X \rangle + \eta \|X\|_{1/2}^{1/2}.
\end{equation}  

In \cref{sec: model}, we investigate the theoretical properties of problem~\cref{eq:l1/2-pro}. When the sparsity penalty parameter satisfies $\eta > 4 \lambda_{\text{max}}(A)$, 
the overall objective function becomes strongly concave, and the \(\ell_{1/2}\)-regularization term will drive the entries of 
$X$ toward binary values under the row and column constraints. Moreover, when \(\eta > 4 \|A\|_{\infty}\), the regularized problem \cref{eq:l1/2-pro} is equivalent to the original binary problem \cref{eq:ori-pro}, and they possess identical global and local minimizers. 
 
Building upon the above theoretical results, we next develop a novel algorithmic framework to solve problem \cref{eq:l1/2-pro}. Although problem~\cref{eq:l1/2-pro} can be solved by first-order schemes such as subgradient projection or proximal point methods, the corresponding subproblems require joint treatment of projection onto a simplex and $\ell_{1/2}$ regularization, resulting in the absence of closed-form solutions. Therefore, we reformulate the problem using a variable-splitting technique, as shown in~\cref{eq:l1/2-proxy}:
\begin{equation}
\label{eq:l1/2-proxy}
\min_{(X,Y) \in \mathcal{F}_3} \quad  \frac{1}{2} \langle A,\, XY^{\top} \rangle + \langle G,\, Y \rangle + \eta \|X\|_{1/2}^{1/2},
\end{equation}  
where the feasible set \( \mathcal{F}_3 \) is defined as:  
\[
\mathcal{F}_3 = \left\{ X,Y \in \mathbb{R}^{n \times m} \mid \mathbf{1}_n^{\top} Y = b ~ \mathbf{1}_m^{\top}, ~ Y \mathbf{1}_m = \mathbf{1}_n, ~ 0 \leq X \leq 1, ~ Y - X = 0 \right\}.
\]  
Note that by imposing $X=Y$, problem \cref{eq:l1/2-proxy} becomes exactly equivalent to problem~\cref{eq:l1/2-pro}. They share identical global optima and KKT points, as established in \cref{pro:Equivalence_of_KKT_points}.
In this reformulation, the quadratic term is transformed into a bilinear one, while the constraint set $\mathcal{F}_2$ is decomposed into separate affine constraints in $Y$
 and simple box constraints in 
$X$. When solved via the Alternating Direction Method of Multipliers (ADMM), both the $X$-subproblem and the $Y$-subproblem admit exact closed-form solutions, thereby enhancing the effectiveness and robustness of the ADMM algorithm.

Our theoretical contributions focus on the convergence guarantees of the ADMM algorithm. While the convergence of ADMM in convex settings has been thoroughly established \cite{chen2019extended,chen2025convergence,han2018linear,wenzaiwen,mihic2021managing,li2016majorized}, we focus our discussion on the nonconvex setting.
Hong et al.~\cite{hong2016convergence} analyzed problems with linear constraints and coupled variables, which may include convex nonsmooth or nonconvex smooth terms. 
Jiang et al.~\cite{ma2019structured} explored nonsmooth nonconvex and variable coupling problems. Their framework allows the possibly nonsmooth nonconvex part to be Lipschitz continuous or lower semicontinuous. However, the last block of variables must be unconstrained. Wang et al.~\cite{yin2019global} also studied the variable coupling case, but their analysis requires the objective function to be smooth.
Han \cite{han2022survey} proposed a semi-proximal ADMM variant with convergence guarantees to solve three-block nonconvex optimization problems under linear constraints. However, this method requires the three variable blocks to be separable in the objective function.
Other works~\cite{boct2020proximal,cohen2022dynamic,hien2024inertial,kong2024global,themelis2020douglas,yashtini2022convergence} also explore ADMM for nonsmooth nonconvex problems, but a very small fraction of the existing frameworks explore more complicated structure problems such as problem \cref{eq:l1/2-proxy}. The structure of problem \cref{eq:l1/2-proxy} features a nonsmooth nonconvex penalty term, a coupling term, and constraints on all variables.
Then we further demonstrate that the algorithm converges to a binary solution exactly within a finite number of iterations, avoiding infinite loops.

Our overall contributions are summarized as follows:  
\begin{enumerate}  
    \item We relax the binary constraints into box constraints and employ an \(\ell_{1/2}\) quasi-norm regularization. The formulation is provably equivalent to the original problem under certain conditions. 
    
    \item We reformulate problem~\cref{eq:l1/2-pro} into the equivalent problem~\cref{eq:l1/2-proxy}, which shares identical KKT points with it. We adopt the ADMM algorithm to solve it, where both the $X$- and $Y$-subproblems admit closed-form solutions. Experimental results on representative examples demonstrate that our algorithm performs effectively and robustly in practical scenarios.    
    \item Theoretically, we conduct a comprehensive analysis encompassing convergence guarantees and the finite-step termination property of our algorithm. The theoretical understanding of ADMM for problems involving nonconvexity, nonsmoothness, coupling term and full constraints on all variables is still evolving. By rigorously analyzing a specific instance within this challenging domain, our work contributes to establishing a theoretical foundation for future studies of ADMM in similarly structured complex optimization problems.
\end{enumerate}

The paper is organized as follows. Mathematical preliminaries are presented in \cref{sec: prelimi}.  The link between the \(\ell_{1/2}\)-regularized problem \cref{eq:l1/2-pro} and the original problem \cref{eq:ori-pro} is proved in \cref{sec: model}.  The proposed algorithm is detailed in \cref{sec: algo}. The theoretical analysis of ADMM is provided in \cref{Theoretical}. The experimental results are shown in \cref{sec:experiments}. Finally, conclusions are drawn in \cref{sec:conclusions}.

\section{Mathematical Preliminaries}
\label{sec: prelimi}
In this section, we present some essential preliminaries, which lay the foundation for the convergence analysis of the proposed ADMM algorithm in subsequent chapters. To begin with, we review essential concepts from nonconvex subdifferential calculus, with comprehensive treatments in Rockafellar and Wets~\cite{rockafellar2009variational}.

For a proper lower semicontinuous function $g: \mathbb{R}^d \to (-\infty, +\infty]$, its domain is defined as:
$
\operatorname{dom} g := \{ x \in \mathbb{R}^d : g(x) < +\infty \}.
$
Then the subdifferential of $g$ is defined as follows.
\begin{definition}[Subdifferentials]
    Let $g : \mathbb{R}^d \to (-\infty, +\infty]$ be a proper and lower semicontinuous function.

\begin{itemize}
    \item[(i)] For a given $x \in \mathrm{dom} \, g$, the Fréchet subdifferential of $g$ at $x$, written $\hat{\partial} g(x)$, is the set of all vectors $u \in \mathbb{R}^d$ that satisfy:
    \[
    \liminf_{y \to x,\, y \neq x} \frac{ g(y) - g(x) - \langle u, y - x \rangle }{ \| y - x \| } \geq 0.
    \]
    If $x \notin \mathrm{dom} \, g$, we set $\hat{\partial} g(x) = \emptyset$.

    \item[(ii)] The limiting subdifferential (or simply, the subdifferential) of $g$ at $x$, written $\partial g(x)$, is defined through the closure process:
    \[
    \partial g(x) := \left\{ u \in \mathbb{R}^d : \exists x^k \to x,\, g(x^k) \to g(x),\, \text{and } u^k \in \hat{\partial} g(x^k) \to u \text{ as } k \to \infty \right\}.
    \]

    \item[(iii)] The horizon subdifferential of $g$ at $x \in \mathrm{dom} \, g$, written $\partial^\infty g(x)$, is the set of all vectors $u \in \mathbb{R}^d$ for which there exist $x^k \in \mathbb{R}^d$, $u^k \in \partial g(x^k)$, and $\mathbb{R} \ni t_k \downarrow 0$ such that:
    \[
    (x^k, g(x^k), t_k u^k) \to (x, g(x), u) \quad \text{as } k \to \infty.
    \]
\end{itemize}
\end{definition}

Subdifferentials play a central role in nonsmooth nonconvex optimization by providing generalized gradients to establish optimality conditions. 
The next two theorems provide a fundamental connection between variational geometry and optimality conditions more generally. Firstly, if $g$ is differentiable, the first-order conditions for optimality are as follows~\cite[Theorem 6.12]{rockafellar2009variational}.
\begin{theorem}[Basic first-order conditions for optimality]
\label{first order for differentiable}
Consider the problem of minimizing a differentiable function \( g \) over a set \( C \subset \mathbb{R}^d \). A necessary condition for \( \bar{x} \) to be locally optimal is:
\begin{equation*}
    \langle \nabla g(\bar{x}), w \rangle \geq 0 \quad \text{for all } w \in T_C(\bar{x}),
\end{equation*}
which is the same as \( -\nabla g(\bar{x}) \in \hat{N}_C(\bar{x}) \) and implies:
\begin{equation*}
\label{eq: Nc condition}
    -\nabla g(\bar{x}) \in N_C(\bar{x}), \quad \text{or} \quad 0 \in \nabla g(\bar{x}) + N_C(\bar{x}),
\end{equation*}
where $T_C(\bar{x}), \hat{N}_C(\bar{x}),  N_C(\bar{x})$ are the tangent cone, normal cone, regular normal cone at $\bar{x}$, respectively.
\end{theorem}

\Cref{first order for differentiable} will be used in the convergence analysis in \cref{Theoretical}. It provides the first-order necessary condition for minimizing a differentiable function over a constraint set. Furthermore, when $g$ is a proper, lower semi-continuous function, the corresponding first-order condition takes the following form~\cite[Theorem 8.15]{rockafellar2009variational}.

\begin{theorem}
\label{optimality relative to a set}
Consider the problem of minimizing a proper, lower semi-continuous function \( g: \mathbb{R}^d \to \overline{\mathbb{R}} \) over a closed set \( C \subset \mathbb{R}^d \). Let \( \bar{x} \) be a point of \( C \) at which the following constraint qualification is satisfied: the set \( \partial^\infty g(\bar{x}) \) contains no nonzero vector $v$ such that \( -v \in N_C(\bar{x}) \). Then, for \( \bar{x} \) to be locally optimal, it is necessary that
\[
\partial g(\bar{x}) + N_C(\bar{x}) \ni 0,
\]
where $N_C(\bar{x})$ is the normal cone of the set $C$ at $\bar{x}$, capturing the non-descent directions from $\bar{x}$ relative to $C$.
\end{theorem}

Based on \cref{optimality relative to a set}, we propose the definition of a KKT point for the problem~\cref{eq: general form}.

\begin{equation}
\label{eq: general form}
\begin{aligned}
\underset{X,Y\in C}{\min} \quad & \psi(X,Y),\\
\text{s.t.} \quad & AX + BY = 0.
\end{aligned}
\end{equation}

Here $\psi(X,Y)$ is a proper and lower semi-continuous function, $C$ is a closed set.

\begin{definition}
\label{def: kkt}
For problem \cref{eq: general form}, its augment Lagrangian $L_{\beta} = \psi(X,Y)+\langle AX+BY,\Lambda \rangle + \frac{\beta}{2}\|AX+BY\|_F^2 $, where $\Lambda$ is a Lagrangian multiplier. Point $(X^*,Y^*,\Lambda^*) $ is a KKT point if:
\begin{equation}
\label{eq: kkt condition}
\begin{aligned}
    & AX^* + BY^* = 0, \\
    &0 \in \partial_X \psi(X^*,Y^*) + A^T\Lambda^* + N_C(X^*),\\
    &0 \in \partial_Y \psi(X^*,Y^*) + B^T\Lambda^* + N_C(Y^*).\\
\end{aligned}
\end{equation}
\end{definition}

The $\ell_{1/2}$-norm in problem \cref{eq:l1/2-pro} is nonsmooth and nonconvex. To support the convergence analysis of our algorithm in \cref{Theoretical}, we examine its restricted prox-regularity property, first introduced by Wang et al.~\cite{yin2019global}. We begin by recalling the definition of this concept.
\begin{definition}
(Restricted prox-regularity) For a lower semi-continuous function $f$, let $E >1$, $f: \mathbb{R}^d \to \mathbb{R} \cup \{\infty\}$, and define the exclusion set:
\begin{equation*}
    S_E := \left\{ x \in \operatorname{dom}(f) : \| v \| > E, \; \forall \; v \in \partial f(x) \right\},
\end{equation*}

f is called restricted prox-regular if, for any $E >1$ and bounded set $T \subset \operatorname{dom}(f)$, there exists $\gamma >0$, such that: 
\begin{equation}
\label{eq: res-prox-regu}
        f(y) + \frac{\gamma}{2} \| x - y \|^2 \geq f(x) + \langle v, y - x \rangle, \quad 
    \forall \,  x \in T \setminus S_E, \; y \in T, \; v \in \partial f(x), \; \| v \| \leq E.
\end{equation}
\end{definition}
While~\cite{yin2019global} establishes restricted prox-regularity of $\ell_q(x) = \|x\|_q^q$ for $q \in (0,1)$, our convergence analysis in \cref{Theoretical} requires this property for the constrained function:
\[
\theta(x) := \eta \|x\|_q^q + \mathbb{I}_{[0,1]^d}(x)
\]
where $\mathbb{I}_{[0,1]^d}(x)$ denotes the indicator function of the box constraint. We now prove that $\theta$ is restricted prox-regularity,  building upon the framework established in \cite{yin2019global}.
\begin{proposition}
\label{pro:reg}
The function \( \theta(x) :=\eta \|x\|_{q}^{q} + \mathbb{I}_{[0,1]^d}(x) \), with \(x\in \mathbb{R}^d, q \in (0,1) \), is restricted prox-regular.
\end{proposition}
\begin{proof}
The limiting subdifferential of $\theta$ at $x$ is given by:
\[
\partial \theta(x) = \left\{ v = (v_1, \dots, v_d)^\top \middle|
\begin{aligned}
& v_i = \eta q x_i^{q-1} && \text{if } 0 < x_i < 1 \\
& v_i \in \mathbb{R} && \text{if } x_i = 0 \\
& v_i \in [\eta q, +\infty) && \text{if } x_i = 1
\end{aligned}
\right\}.
\]

This characterization follows from:
\begin{enumerate}
    \item The subdifferential decomposition $\partial(\eta \|x\|_q^q + \mathbb{I}_{[0,1]^d}(x)) = \prod_{i=1}^d \partial (\eta x_i^q + \mathbb{I}_{[0,1]}(x_i))$
    \item Pointwise application of the sum rule for limiting subdifferentials
    \item The subdifferential of $\mathbb{I}_{[0,1]}(x_i)$ is:
    \[
    \partial \mathbb{I}_{[0,1]}(x_i) = 
    \begin{cases} 
    (-\infty, 0] & x_i = 0,\\
    \{0\} & 0 < x_i < 1, \\
    [0, +\infty) & x_i = 1.
    \end{cases}
    \]

    The subdifferential of $x_i^q$ is: 
    \[
    \partial x_i^q = 
    \begin{cases} 
    qx_i^{q-1} & x_i > 0,\\
    R &  x_i =0. \\
    \end{cases}
    \]
\end{enumerate}

For any $E > 1$, define $\gamma = \max \left\{ \frac{4d(1+E)}{c^2}+2d\eta q ,\ q(1-q)c^{q-2}+2d\eta q \right\}$ with $c = \frac{1}{3} \left( \frac{\eta q}{E} \right)^{\frac{1}{1-q}}$. 
The exclusion set $S_E \supset \left\{ x : \min_{\substack{x_i \neq 0}} |x_i| \leq 3c \right\}$. 
For bounded set $T \subset \operatorname{dom} \theta$: 
If $T$ has no point with entries in $\{0,1\}$, \cref{eq: res-prox-regu} holds automatically. 
Else, for $z \in T \setminus S_E$, $v\in \partial\theta(z)$ and $y \in T$: 
Partition $z = (z^1, z^2)$ where $z^2$ contains all entries $z_i = 1$. 
Then partition $y = (y^1, y^2)$ and $v = (v^1, v^2)$ correspondingly. 

Assume $\|z^1 - y^1\| \leq c$. 
If $z^1 \neq 0$, then: 
\begin{equation*} 
\operatorname{supp}(z^1) \subseteq \operatorname{supp}(y^1), \quad
\|z^1\|_0 \leq \|y^1\|_0 ,
\end{equation*} 
where $\operatorname{supp}(\cdot)$ is the support set and $\|\cdot\|_0$ its cardinality. 
Define the projection: 
\[ 
y_i' = \begin{cases} y_i^1 & i \in \operatorname{supp}(z^1), \\ 0 & \text{otherwise}. \end{cases} 
\]

Then for any \( v^1 \in \partial \theta(z^1) \), the following holds:
\begin{equation}
\label{eq:y1-z1_1}  
\begin{aligned}
\|y^1\|_q^q - \|z^1\|_q^q - \langle v^1, y^1 - z^1 \rangle 
&\overset{(a)}{\geq} \|y'\|_q^q - \|z^1\|_q^q - \langle v^1, y' - z^1 \rangle \\
&\overset{(b)}{\geq} -\frac{q(1 - q)}{2} c^{q-2} \|z^1 - y'\|^2  \\
&\overset{(c)}{\geq} -\frac{q(1 - q)}{2} c^{q-2} \|z^1 - y^1\|^2,
\end{aligned}
\end{equation}
where (a) holds for \( \|y^1\|_q^q = \|y'\|_q^q + \|y^1 - y'\|_q^q \) by the definition of \( y' \),
(b) holds because \( \theta(x) \) is twice differentiable along the line segment connecting \( z^1 \) and \( y' \), and the second order derivative is no bigger than \( q(1 - q)c^{q-2} \), and (c) holds because
$\|z^1 - y^1\| \geq \|z^1 - y'\|.$ If $z^1=0$, then $y^1=0$ as well, and \cref{eq:y1-z1_1} is satisfied naturally. 

If \( \|z^1 - y^1\| > c \), then for any \( v^1 \in \partial \theta(z^1) \), we have
\begin{equation}
\label{eq:y1-z1_2}
\|y^1\|_q^q - \|z^1\|_q^q - \langle v^1, y^1 - z^1 \rangle \geq -(2n + 2nE) \geq -\frac{2d + 2dE}{c^2} \|y^1 - z^1\|^2.
\end{equation}

By the definition of $v^2$, we have: 
\begin{equation}
\label{eq:y2-z2-all}
   \theta(y^2) - \theta(z^2) \geq -\frac{d\eta q}{2} \langle v^2,y^2-z^2 \rangle,
\end{equation}
for $\theta$ is prox-regular at $x = 1$. Combing \cref{eq:y1-z1_1} , \cref{eq:y1-z1_2}  and \cref{eq:y2-z2-all} yields the result.
\end{proof}

We define the Hausdorff distance for bounded sets~\cite{bourkhissi2025convergence}, which is essential for the convergence analysis.
\begin{definition}[Hausdorff distance]
For bounded sets \( B, C \subset \mathbb{R}^d \), the Hausdorff distance is:
\begin{equation*}
\mathrm{dist}_H(B, C) := \max \left\{ 
\sup_{b \in B} \mathrm{dist}(b, C),\ 
\sup_{c \in C} \mathrm{dist}(c, B) 
\right\},
\end{equation*}
where $\mathrm{dist}(b, C) = \inf_{c \in C} \|b - c\|$ and $\mathrm{dist}(c, B) = \inf_{b \in B} \|c - b\|$.
\end{definition}

\begin{proposition}
\label{distH} 
Let set $C = \{x|Ax=b \}$.
For any \( \tau> 0 \), define the truncated normal cone mapping:
\begin{equation}
\label{nc mapping1}  
\bar{N}_{C}(x) = N_{C}(x) \cap \mathbb{B}_{\tau}, \quad \forall x \in C,
\end{equation}
where \( \mathbb{B}_{\tau} \subset \mathbb{R}^d \) is the closed ball centered at the origin with radius \( \tau \). Then
there exists \(\kappa > 0\) satisfying:
\[
\operatorname{dist}_H\left(\bar{N}_C(x), \bar{N}_C(x')\right) \leq \kappa \|x - x'\|, \quad \forall x, x' \in C.
\]
\end{proposition}

The proof follows directly from~\cite{bourkhissi2025convergence}.

\section{\(\ell_{1/2}\)-regularized model}
\label{sec: model}
This section establishes theoretical conditions under which the proposed regularization remains equivalent to the original formulation and investigates the binary property of the KKT points of problem~\cref{eq:l1/2-pro}. Our analysis is partly inspired by the approach in~\cite{lpjiangbo}, where the authors study a smoothed $\ell_{1/2}$-regularized problem with a similar structure.  Define $f_1(X) = \frac{1}{2} \langle A,\, XX^{\top} \rangle + \langle G,\, X \rangle$, $f_2(X) = \sum_{i=1}^n\sum_{j=1}^m|X_{ij}|^{1/2}$, $F(X) = f_1(X) + \eta f_2(X)$.  
We begin with the following key propositions:
\begin{proposition}
\label{lem: equivalent1}
Let $X^*$ be feasible for the problem \cref{eq:l1/2-pro}. If $X^*$ has non-integer entries, then there exists a cycle $\mathcal{C} = \{ X_{i_1j_1}^*, X_{i_1j_2}^*, X_{i_2j_2}^*, \dots, X_{i_sj_s}^*, X_{i_sj_1}^* \}$ with $2s \leq 2m$ entries in $(0,1)$.
\end{proposition}
\begin{proof}
Assume $X_{i_1j_1}^* \in (0,1)$. Since $\sum_j X_{i_1j}^* =1$, there exists $ j_2 \neq j_1$ with $X_{i_1j_2}^* \in (0,1)$.
Similarly, $\sum_i X_{ij_2}^* =b$ implies $\exists i_2 \neq i_1$ with $X_{i_2j_2}^* \in (0,1)$.
If $X_{i_2j_1}^* \in (0,1)$, we obtain the cycle $\mathcal{C} = (X_{i_1j_1}^*, X_{i_1j_2}^*, X_{i_2j_2}^*, X_{i_2j_1}^*)$.

Otherwise, assume that a fractional cycle $\mathcal{C}'_s = \{ X_{i_1j_1}^*, X_{i_1j_2}^*, \dots, X_{i_sj_s}^* \}$ is constructed. If $\exists t \in \{1,\dots,s-1\}$ with $X_{i_sj_t}^* \in (0,1)$, then the closed cycle $\mathcal{C} = \{ X_{i_tj_t}^*, X_{i_tj_{t+1}}^*, \dots, X_{i_sj_s}^*, X_{i_sj_t}^* \}$ satisfies \cref{lem: equivalent1}. 
Else, the row-sum constraint $\sum_j X_{i_sj}^* =1$ implies $\exists j_{s+1} \neq j_s$ with $X_{i_sj_{s+1}}^* \in (0,1)$, extending $\mathcal{C}_s'$ to include $X_{i_sj_{s+1}}$.
If $\exists t \in \{1,\dots,s-1\}$ with $X_{i_tj_{s+1}}^* \in (0,1)$, the cycle $\mathcal{C} = \{ X_{i_tj_{s+1}}^*, X_{i_tj_{t+1}}^*\dots, X_{i_sj_{s+1}}^*\}$ satisfies \cref{lem: equivalent1}.
Else, extend $\mathcal{C}_{s}'$ to $\mathcal{C}_{s+1}'$.

When extended to $\mathcal{C}_m = \{ X_{i_1j_1}^*, \dots, X_{i_mj_m}^* \}$, the column indices $\{j_1,\dots,j_m\}$ form a permutation of $\{1,\dots,m\}$. By row $i_m$'s constraint $\sum_j X_{i_mj}^* =1$, $\exists t \in \{1,\dots,m-1\}$ with $X_{i_mj_t}^* \in (0,1)$. Thus, the cycle $\mathcal{C} = \{ X_{i_tj_t}^*, X_{i_tj_{t+1}}^*, \dots, X_{i_mj_m}^*, X_{i_mj_t}^* \}$ satisfies \cref{lem: equivalent1}. 
\end{proof}

\begin{proposition}
\label{pro:concave}
   Let $\nu_{f_1} = \lambda_{\text{max}}(A)$, $\nu_{f_2} = \frac{1}{4}$, where $\lambda_{\text{max}}(A)$ is the maximum eigenvalue of $A$.  Then when $\eta > 4\lambda_{\text{max}}(A)$, the function $F(X)$ is strongly concave with positive parameter $\eta \nu_{f_2} - \nu_{f_1}$ over $\mathcal{F}_2$.
\end{proposition}
\begin{proof}
For function \( f_1(X) \), by the second-order Taylor expansion and the definition of \( \nu_{f_1} \), we have:
    \begin{equation}
    \label{eq:taylor1}
        f_1(Y) - f_1(X) - \langle \nabla f_1(X), Y - X \rangle \leq \frac{\nu_{f_1}}{2} \|Y - X\|_{\text{F}}^2, \quad \forall X, Y \in \mathcal{F}_2.
    \end{equation}
    
    Using the convex combination \( Z = tX + (1-t)Y \) for \( t \in [0,1] \) and substituting into \cref{eq:taylor1}, we derive the inequality for the convex combination of \( f_1 \):
    \begin{equation}
    \label{eq:taylor2}
        f_1(tX + (1-t)Y) \geq t f_1(X) + (1-t) f_1(Y) - \frac{\nu_{f_1}}{2} t(1-t) \|Y - X\|_{\text{F}}^2, ~\forall X, Y \in \mathcal{F}_2, \, \forall t \in [0,1].
    \end{equation}
    
    For function \( f_2(X) \), it is known to be strongly concave with the parameter \( \nu_{f_2} \), which implies:
    \begin{equation}
    \label{eq:taylor3}
        f_2(tX + (1-t)Y) \geq t f_2(X) + (1-t) f_2(Y) + \frac{\nu_{f_2}}{2} t(1-t) \|Y - X\|_{\text{F}}^2,~ \forall X, Y \in \mathcal{F}_2, \, \forall t \in [0,1].
    \end{equation}
    
    Multiplying \cref{eq:taylor3} by \( \eta \) yields that for $\forall X, Y \in \mathcal{F}_2, \, \forall t \in [0,1]$,
    \begin{equation}
    \label{eq:taylor4}
        \eta f_2(tX + (1-t)Y) \geq \eta \left( t f_2(X) + (1-t) f_2(Y) + \frac{\nu_{f_2}}{2} t(1-t) \|Y - X\|_{\text{F}}^2 \right).
    \end{equation}
    
    Summing \cref{eq:taylor2} and \cref{eq:taylor4} and using the definition \( F(X) = f_1(X) + \eta f_2(X) \), we obtain:
    \begin{align}
    \label{eq:final}
    F(tX + (1-t)Y) &\geq t F(X) + (1-t) F(Y) + \frac{t(1-t)}{2} \left( \eta \nu_{f_2} - \nu_{f_1} \right) \|Y - X\|_{\text{F}}^2,
    \end{align}
for $\forall X, Y \in \mathcal{F}_2, \, \forall t \in [0,1]$.
     Thus, \( F(X) \) is strongly concave over \( \mathcal{F}_2 \) with positive parameter \( \eta \nu_{f_2} - \nu_{f_1} \).
\end{proof}
\begin{theorem}[Equivalence between Problem \cref{eq:ori-pro} and Problem \cref{eq:l1/2-pro}]
\label{thm:penalty} 
If $\eta > 4\|A\|_\infty$, then:
\begin{enumerate}
    \item Any local minimizer of problem \cref{eq:l1/2-pro} is feasible and locally optimal for problem \cref{eq:ori-pro}. On the other hand, each feasible point of problem \cref{eq:ori-pro} is a local minimizer of problem \cref{eq:l1/2-pro}.
    \item Problem \cref{eq:ori-pro} and problem \cref{eq:l1/2-pro} share identical globally optimal solutions.
\end{enumerate}
\end{theorem}
\begin{proof}
Let $X^*$ be a local minimizer of problem \cref{eq:l1/2-pro}.
By \cref{lem: equivalent1}, if $X^*$ has fractional entries, there exists a cycle $\mathcal{C} = \{ X_{i_1j_1}^*, \dots, X_{i_sj_s}^*, X_{i_sj_1}^* \}$ with each entry in $(0,1)$. Fixing entries outside $\mathcal{C}$, consider the reduced problem. The direction $v = (1,-1,1,-1,\dots)^\top$ lies in the tangent cone. The pseudo-Hessian matrix is:\\
{\scriptsize
\begin{equation*}
H = 
     \begin{pmatrix}
 A_{i_1i_1} &  &   &  &  &  &  A_{i_1i_s}\\
  & A_{i_1i_1} & A_{i_1i_2}  &   & &  & \\
  & A_{i_2i_1} & A_{i_2i_2} &  &  & &\\
  &  &  &  \ddots &  &   &   \\
  &  &  &  & A_{i_{s-1}i_{s-1}} & A_{i_{s-1}i_s} & \\
    &  &  &  & A_{i_{s}i_{s-1}} & A_{i_{s}i_s} & \\
    A_{i_si_1} & & & & & &  A_{i_si_s}
  
     \end{pmatrix}  - \frac{\eta}{4} 
     \begin{pmatrix}
         X_{i_1i_1}^{-\frac{3}{2}}&  &  &  &  \\
         & X_{i_1j_2}^{-\frac{3}{2}} &  &  &\\
         &  &  \ddots &  &\\
         &  &  &  X_{i_si_s}^{-\frac{3}{2}} \\
         &  &  &  &  X_{i_sj_1}^{-\frac{3}{2}}
     \end{pmatrix}.
 \end{equation*}
}
Thus, when \( \eta \geq 4 \| A \|_{\infty} \), the following holds:  
\begin{equation*}
    v^T H v \leq 2\sum_{t=1}^s |A_{i_ti_t}| + 2\sum_{t=1}^s |A_{i_ti_{t+1}}| - \frac{\eta}{4} \left(\sum_{t=1}^s X_{i_ti_t}^{-\frac{3}{2}} + \sum_{t=1}^sX_{i_ti_{t+1}}^{-\frac{3}{2}}\right)
    < 2s \| A \|_{\infty} - \frac{s}{2} \eta \leq 0,
\end{equation*}  
contradicting the second-order necessary condition. Thus $X^*$ doesn't have fractional entries. Since \cref{eq:ori-pro} has a discrete domain, $X^*$ is locally optimal for the problem \cref{eq:ori-pro}.

Let $\bar{X}$ be a feasible point of the problem \cref{eq:ori-pro} and $X^l,~ l = 1,2, \dots, \frac{n!}{(b!)^m } -1$ denote all remaining feasible points of the problem \cref{eq:ori-pro}.
For any fixed $l$, consider the feasible direction $D^l = X^l -\bar{X}$. It follows that $\bar{X} + t D^l \in \mathcal{F}_2$ for all $t \in [0,1]$. Define:
\begin{equation}
    \begin{aligned}
        & U_1^l = \{(i,j) \mid \bar{X}_{ij} = 0 , X_{ij}^l =0\}, \quad U_2^l = \{(i,j) \mid \bar{X}_{ij} = 0 , X_{ij}^l =1 \},  \\
        & U_3^l = \{(i,j) \mid \bar{X}_{ij} = 1 , X_{ij}^l =0\}, \quad U_4^l = \{(i,j) \mid \bar{X}_{ij} = 1 , X_{ij}^l =1\}. \\
    \end{aligned}
\end{equation}

Clearly, we can see that $|U_1^l| \geq 0,~|U_4^l| \geq 0$ and $|U_2^l| = |U_3^l| \geq 2, ~ \|D^l\|_F^2 = 2|U_2^l| \geq 4$. By the mean-value theorem, for any $t\in (0,1]$, there exists $\xi \in (0,1)$, such that:
\begin{equation}
\label{eq:deltaf1}
\begin{aligned}
    f_1(\bar{X} + t D^l) - f_1(\bar{X}) & =  t \langle \nabla f_1 (\bar{X} + \xi t D^l) , D^l \rangle \\
    & \geq -t \| D^l\|_F \| \nabla f_1 (\bar{X} + \xi t D^l) \|_F \\
    & \geq -t \| D^l\|_F \left(\nu_{f_1} \| \bar{X} + \xi t D^l \|_F + \|G\|_F \right)\\
    & \geq - t \| D^l\|_F  \left( \nu_{f_1}\left(\|\bar{X}\|_F + \xi t \| D^l\|_F \right) + \|G\|_F  \right)\\
    & \geq - t \| D^l\|_F  \left( \nu_{f_1}\left(\sqrt{n} +  \| D^l\|_F\right) + \|G\|_F  \right),\\
\end{aligned}
\end{equation}
where the second inequality follows from the Lipschitz continuity of $\nabla f_1(X)$ with parameter $\nu_{f_1}$, and the last inequality uses $|\bar{X}|_F = \sqrt{n}$.

For any $t \in (0, \frac{1}{2}]$, we have:  
\begin{equation}
\label{eq:deltaf2}
\begin{aligned}
    f_2(\bar{X} + t D^l) - f_2(\bar{X}) &= \sum_{(i,j) \in \bigcup_{s=1}^4 U_s^l} \left( (\bar{X}_{ij} + t D_{ij}^l)^{1/2} - \bar{X}_{ij}^{1/2} \right) \\
    &= |U_2^l| t^{1/2} + |U_3^l| \left( (1-t)^{1/2} - 1 \right) \\
    &\geq \frac{t}{2} \|D^l\|_F^2 \left( t^{-1/2} - \sqrt{2} \right),
\end{aligned}
\end{equation}
where the last inequality is due to $\|D^l\|_F^2 = 2|U_2^l|$ and the Lagrange mean-value theorem.

Define 
\[
\tilde{t}^l = \left( \frac{\eta \|D^l\|_F}{2 \left( \nu_{f_1} \left( \sqrt{n} + \|D^l\|_F \right) + \|G\|_F \right) + \sqrt{2}\eta \|D^l\|_F} \right)^2 < \frac{1}{2},
\]
and let $\tilde{t} = \min\left\{ \tilde{t}^l \mid l = 1, 2, \dots, \frac{n!}{(b!)^m} - 1 \right\}$.

Then, for all $t \in (0, \tilde{t})$ and all $l = \frac{n!}{(b!)^m } -1$, we have:
\begin{equation}
    F(\bar{X} + t D^l ) - F(\bar{X}) > 0.
\end{equation}

This implies that $D^l$ for $l = 1, 2, \dots, \frac{n!}{(b!)^m} - 1$ are strictly increasing feasible directions. Let $\text{Conv}(\bar{X}, \tilde{t})$ denote the convex hull of the points $\bar{X}$ and $\bar{X} + \tilde{t} D^l$,  $l = 1, 2, \dots, \frac{n!}{(b!)^m} - 1$. Since $\eta > 4\|A\|_{\infty} \geq 4\lambda_{\text{max}}(A)$, the function $F(X)$ is strongly concave, as proved in \cref{pro:concave}. Thus, for any $X \in \text{Conv}(\bar{X}, \tilde{t})$, we have $F(X) > F(\bar{X})$. Moreover, one can always choose a sufficiently small fixed $t > 0$ such that $B(\bar{X}, t) \cap \mathcal{F}_2 \subset \text{Conv}(\bar{X}, \tilde{t})$, where $B(\bar{X}, t) = \{ X \in \mathbb{R}^{n \times m} \mid \|X - \bar{X}\|_F^2 \leq t \}$. Consequently, $\bar{X}$ is a local minimizer of the
problem \cref{eq:l1/2-pro}.

Since all local minimizers of the problem \cref{eq:l1/2-pro} are in $\mathcal{F}_1$,  problem \cref{eq:l1/2-pro} reduces to:
\begin{equation}
\label{eq: local-feasible11}
\underset{X \in \mathcal{F}_1}{\min} \quad  \frac{1}{2} \left\langle A, X X^T \right\rangle + \left\langle G, X \right\rangle + \eta \|X\|_{1/2}^{1/2}.
\end{equation}

For any feasible point in $\mathcal{F}_1$, $\|X\|_p^p=n$ remains the same. Thus, problem \cref{eq: local-feasible11} degrades to problem \cref{eq:ori-pro}, which implies that problem \cref{eq:ori-pro} and \cref{eq:l1/2-pro} are equivalent and share the same global optimal solutions.
\end{proof} 

To address computational intractability of problem \cref{eq:ori-pro}, we relax binary constraints to continuous constraints.
However, continuous relaxation may yield fractional solutions incompatible with assignments constraints. 
Motivated by the sparsity pattern of the variable, we augment the objective with an $\ell_{1/2}$ quasi-norm penalty.
\Cref{thm:penalty} provides theoretical guarantees for this exact penalty formulation. However, problem \cref{eq:l1/2-pro} is still NP-hard, it is difficult to find it's global optimum. Thus, we only compute it's KKT point. Following \cref{01kkt} shows that calculating problem \cref{eq:l1/2-pro} is sufficient.

\begin{theorem}[Binary Nature of KKT Points]
\label{01kkt}
    When $\eta > \frac{2}{\sqrt{2}-1} \sqrt{n}(\|G\|_F + \sqrt{n}\|A\|_F )$, any KKT point of the problem \cref{eq:l1/2-pro} is a feasible point of the problem \cref{eq:ori-pro}. 
\end{theorem}
\begin{proof}
    Suppose $X^*$ is a KKT point of the problem \cref{eq:l1/2-pro}. Define $U'(X^*) = \{(i,j) \mid X^*_{ij} \in (0,1)\}$, $U^1(X^*) = \{(i,j) \mid X^*_{ij} =1\}$,
$U'(X^*_{\cdot j}) = \{ i \mid X^*_{ij} \in (0,1)\}$, $U'(X^*_{i\cdot}) = \{ j \mid X^*_{ij} \in (0,1)\}$. Then we let:
    \begin{equation}
 \sum_{i\in U'(X^*_{\cdot j})} X^*_{ij} = b_j , ~\quad \forall j \in U'(X^*_{i\cdot}) , \sum_{j\in U'(X^*_{i\cdot})} X^*_{ij} = 1 , ~\quad \forall i \in U'(X^*_{\cdot j}).
    \end{equation}

By the KKT condition \cref{kkt_of_pen_pro}, we have:
\begin{equation}
\label{eq:xijkkt}
    A_i^TX^*_j + G_{ij} + \frac{1}{2}\eta (X_{ij}^*)^{-1/2} = \nu_i + \mu_j, \quad \forall (i,j) \in U'(X^*).
\end{equation}

Multiplying by $X^*_{ij}$ on both sides of \cref{eq:xijkkt} and then summing them over $(i,j) \in U'(X^*)$, we have:
\begin{equation}
\label{eq:mulx1}
\begin{aligned}
    &\sum_{(i,j) \in U'(X^*)} \left( (A_i^TX^*_j + G_{ij} )X^*_{ij}  +  \frac{1}{2}\eta (X_{ij}^*)^{1/2}   \right) \\
    &  = \sum_{(i,j) \in U'(X^*)} \left( \nu_i + \mu_j   \right) X^*_{ij}  \\
    &= \sum_{i\in U'(X^*_{\cdot j})} \left( \nu_i\left( \sum_{j\in U'(X^*_{i\cdot})} X^*_{ij} \right) \right) +  \sum_{i\in U'(X^*_{i\cdot})} \left( \mu_i\left( \sum_{j\in U'(X^*_{\cdot j})} X^*_{ij} \right) \right) \\
    &= \sum_{i\in U'(X^*_{\cdot j})} \nu_i  + \sum_{i\in U'(X^*_{i\cdot})} \mu_j b_j.\\
\end{aligned}    
\end{equation}

Define $\mathcal{F}_1' = \{X \mid X \in \mathcal{F}_1, ~ U'(X) \subseteq U'(X^*), X_{ij} =1, \forall (i,j) \in U^1(X^*) \}$.
Consider any point $X \in \mathcal{F}_1'$, we have:
\begin{equation}
\label{eq:mulx2}
\begin{aligned}
    &\sum_{(i,j) \in U'(X^*)} \left( (A_i^TX^*_j + G_{ij} )X_{ij}  +  \frac{1}{2}\eta (X_{ij}^*)^{-1/2}  X_{ij} \right) \\
    &  = \sum_{(i,j) \in U'(X^*)} \left( \nu_i + \mu_j   \right) X_{ij}  \\
    &= \sum_{i\in U'(X^*_{\cdot j})} \left( \nu_i\left( \sum_{j\in U'(X^*_{i\cdot})} X_{ij} \right) \right) +  \sum_{i\in U'(X^*_{i\cdot})} \left( \mu_i\left( \sum_{i\in U'(X_{\cdot j})} X_{ij} \right) \right) \\
    &= \sum_{i\in U'(X^*_{\cdot j})} \nu_i  + \sum_{i\in U'(X^*_{i\cdot})} \mu_j b_j.\\
\end{aligned}    
\end{equation}

Combined \cref{eq:mulx1} with \cref{eq:mulx2} we can obtain:
\begin{equation}
\label{eq:c1}
    \sum_{(i,j) \in U'(X^*)} \frac{1}{2}\eta (X^*_{ij})^{-1/2} X_{ij} = \sum_{(i,j) \in U'(X^*)} (A_i^TX^*_j + G_{ij} ) (X_{ij} - X^*_{ij}) + \sum_{(i,j) \in U'(X^*)}  \frac{1}{2}\eta (X^*_{ij})^{1/2}.
\end{equation}

Given any $(i_0,j_0)\in U'(X^*)$, we can choose some special $\tilde{X} \in \mathcal{F}_1'$ such that $\tilde{X}_{i_0j_0} = 1$, then:
\begin{equation}
\label{eq:c2}
    \sum_{(i,j) \in U'(X^*)} (X^*_{ij})^{-1/2} \tilde{X}_{ij} = (X^*_{i_0j_0})^{-1/2} + \sum_{\substack{(i,j) \in U'(X^*)\\ (i,j) \neq (i_0, j_0)}} (X^*_{ij})^{-1/2} \tilde{X}_{ij} \geq (X^*_{i_0j_0})^{-1/2}+ n-1.
\end{equation}

By the fact that $\|X^* - \tilde{X}\|_F \leq  \sqrt{n} $, we have:
\begin{equation}
\label{eq:c3}
    \sum_{(i,j) \in U'(X^*)}  (A_i^TX^*_j + G_{ij} ) (X^*_{ij} - \tilde{X}_{ij}) \leq \|A_i^TX^*_j + G_{ij}\|_F \|X^* - \tilde{X}\|_F \leq \sqrt{n}(\|G\|_F + \sqrt{n}\|A\|_F ).
\end{equation}

Combining \cref{eq:c1}, \cref{eq:c2}, \cref{eq:c3} with $\sum_{(i,j) \in U'(X^*)} (X^*_{ij})^{1/2} \leq n$, we conculde that:
\begin{equation}
    X^*_{i_0j_0} \geq \left( \frac{2}{\eta} \sqrt{n}(\|G\|_F + \sqrt{n}\|A\|_F ) +1\right)^{-2},~\forall (i_0,j_0)\in U'(X^*)
\end{equation}

When  $\eta > \frac{2}{\sqrt{2}-1} \sqrt{n}(\|G\|_F + \sqrt{n}\|A\|_F )$, we see that any $(i_0,j_0)\in U'(X^*)$, $ X^*_{i_0j_0} > \frac{1}{2}$, which makes a contradiction. Thus, $U'(X^*) =\emptyset$, $X^*$ is a feasible point for the problem \cref{eq:ori-pro}. 
\end{proof}
\section{ADMM Algorithm}
\label{sec: algo}
In \cref{sec: model}, we reformulate problem \cref{eq:ori-pro} into a more solvable form, specifically problem \cref{eq:l1/2-pro}.
However, due to the presence of the nonsmooth term and the constraint set 
$\mathcal{F}_2$, many first-order methods, such as subgradient projection, proximal point projection, do not admit closed-form solutions in their corresponding subproblems. Therefore, we use the variable splitting method and reformulate problem \cref{eq:l1/2-pro} into problem \cref{eq:l1/2-proxy}. 
Despite this transformation, problem \cref{eq:l1/2-proxy} remains difficult to solve globally due to its nonconvex and nonsmooth nature. 
Therefore, we aim to compute a high quality KKT point of the problem rather than its exact global optimum. It is resonable because any KKT point of problem \cref{eq:l1/2-proxy} is a local minimum of problem \cref{eq:ori-pro} by \cref{pro:Equivalence_of_KKT_points} and \cref{01kkt}.  
To this end, we adopt the ADMM framework, as the resulting subproblems admit closed-form solutions. Moreover, ADMM offers an effective approach for obtaining a KKT point in problems that exhibit nonconvexity and nonsmoothness.
 We also provide a detailed discussion on the optimality conditions and the termination criterion. The convergence guarantee and the finite-step termination property are established in the subsequent section.
\subsection{Variable Splitting}
We adopt the variable splitting technique and reformulate problem \cref{eq:l1/2-pro} into problem \cref{eq:l1/2-proxy}. To compute a KKT point of the problem \cref{eq:l1/2-proxy}, we employ ADMM owing to its efficiency in handling structured nonconvex problems. ADMM decomposes the problem into simpler subproblems that are solved easier. We begin by constructing the augmented Lagrangian function:
\begin{equation*}
L_\beta(X, Y, \Lambda) = \frac{1}{2} \left\langle A, X Y^{\top} \right\rangle + \langle G, Y \rangle + \eta  \|X\|^{1/2}_{1/2} + \langle \Lambda, Y-X \rangle + \frac{\beta}{2} \| Y - X \|_F^2,
\end{equation*}
where $\Lambda \in \mathbb{R}^{n \times m}$ is the dual variable (or Lagrangian multiplier) associated with the constraint $Y = X$, and $\beta > 0$ is the penalty parameter that controls the weight of the quadratic penalty for constraint violation. 

The algorithm alternately updates the three variables: \( X \), \( Y \), and \( \Lambda \). The subsequent section details the update process for each variable.
\subsection{Update of Variables}
\label{subsection:Update of Variables}
We now detail the \(k\textsuperscript{th}\) iteration of the ADMM algorithm. First, with \( \Lambda^{k-1} \) and \( Y^{k-1} \) fixed, the subproblem for updating \( X \) is formulated as follows:
\begin{equation}
\label{eq: update Y}
\begin{aligned}
X^{k} 
&= \underset{X \in \mathbb{R}^{n\times m}}{\arg\min} \quad L_\beta(X, Y^{k-1}, \Lambda^{k-1}) \\
&= \underset{X \in \mathbb{R}^{n\times m}}{\arg\min} \quad 
   \frac{\beta}{2} \left\| X - Y^{k-1} - \frac{1}{\beta} \left( \Lambda^{k-1} - \frac{1}{2} A Y^{k-1} \right) \right\|_F^2 
   + \eta \|X\|^{1/2}_{1/2}, \\
\text{s.t.} \quad 
&0 \leq X \leq 1.
\end{aligned}
\end{equation}

Defining \( R^k = Y^{k-1} + \frac{1}{\beta} \left( \Lambda^{k-1} - \frac{1}{2} A Y^{k-1} \right) \), problem \cref{eq: update Y} decouples into the following independent subproblem for each entry \( x \) of \( X \):
\begin{equation} 
\label{eq: Update_Y_1}
\begin{aligned}  
\min_{x} & \quad \frac{\beta}{2} \left( x - r \right)^2 + \eta x^{1/2}, \\
\text{s.t.} & \quad 0 \leq x \leq 1,
\end{aligned}
\end{equation}
where \( r \) denotes the corresponding entry of \( R^k \).

The problem \cref{eq: Update_Y_1} admits a closed-form solution. Specifically, for \( 0 < x < 1 \), the first-order optimality condition yields the following equation:
\begin{equation}
\label{x*}
\beta (x - r) + \frac{1}{2} \eta x^{-1/2} = 0, \quad 0 < x < 1.
\end{equation}

To solve it, we apply the substitution \( t = x^{-1/2} \), which transforms it into a cubic equation:
\[
\frac{1}{2}\eta t^3 - r t^2 + \beta = 0, \quad t > 1.
\]

This cubic equation is then solved using the method described in ~\cite{xu2012l_}.
If the cubic equation has at least one real root \( t > 1 \), we define the candidate solution \( x^* = \{(\max\, \{ t >1 |  \, \frac{1}{2}\eta t^3 - r t^2 + \beta = 0 \} )^{-2} \)\}. Otherwise, we set \( x^* = \emptyset \). The global minimizer for \cref{eq: Update_Y_1} is found by comparing the objective function values at the boundary points \( x = 0 \) and \( x = 1 \), and at the stationary point candidate \( x^* \):
\begin{equation}
\label{eq: Update_Y_2}
x^k = \underset{x \in \{0,1\} \cup x^*}{\arg\min} \ \frac{\beta}{2}(x - r)^2 + \eta x^{1/2}.
\end{equation}

This process yields the updated matrix \( X^{k} \) in an entry-wise manner.

Next, with \( \Lambda^{k-1} \) and \( X^{k} \) fixed, the subproblem for updating \( Y \) is formulated as follows:
\begin{equation}
\label{eq:update x}
\begin{aligned}
Y^{k} &= \underset{Y \in \mathbb{R}^{n\times m}}{\arg\min} \quad  L_\beta(X^{k}, Y, \Lambda^{k-1}) \\
&= \underset{Y \in \mathbb{R}^{n\times m}}{\arg\min} \quad \frac{\beta}{2} \left\| Y - X^{k} + \frac{1}{\beta} \left( \Lambda^{k-1} + \frac{1}{2} A X^{k} + G \right) \right\|_F^2, \\
\text{s.t.} \quad & \mathbf{1}_n^T Y = b ~ \mathbf{1}_m^T, \\
               & Y \mathbf{1}_m = \mathbf{1}_n.
\end{aligned}
\end{equation}

Defining $B^k = X^{k} - \frac{1}{\beta}\left(\Lambda^{k-1} + \frac{1}{2} A X^{k} + G\right)$, the problem simplifies to:
\begin{equation}
\label{eq:update x 1}
\begin{aligned}
\underset{Y \in \mathbb{R}^{n\times m}}{\min} \quad  & \frac{\beta}{2} \left\| Y - B^k \right\|_F^2, \\
\text{s.t.} \quad & \mathbf{1}_n^T Y = b ~ \mathbf{1}_m^T, \\
                & Y \mathbf{1}_m = \mathbf{1}_n.
\end{aligned}
\end{equation}

The problem \cref{eq:update x 1} is a projection onto an affine constraint set and admits the closed-form solution:
\begin{equation}
\label{eq: up_y}
    Y^{k} = B^k-\frac{1}{n}\mathbf{1}_n\mathbf{1}_n^{\top}B^k + \frac{1}{m}(\mathbf{1}_n - B^k\mathbf{1}_m + \frac{1}{n}\mathbf{1}_n\mathbf{1}_n^TB^k\mathbf{1}_m)\mathbf{1}_m^{\top}.
\end{equation}

Finally, the dual variable (multiplier) \( \Lambda \) is updated via:
\begin{equation}
\label{eqn:lambda}
    \Lambda^{k}  = \Lambda^{k-1} + \beta(Y^{k} - X^{k}).
\end{equation}

The three variables \( X \), \( Y \), and \( \Lambda \) are updated alternately. The availability of closed-form solutions for both the \( X \) subproblem and the \( Y \) subproblem enhances the numerical stability and robustness of the algorithm.

\subsection{Optimality Conditions and Termination Criterion}
Building upon the described algorithm, we now analyze its optimality conditions in light of \cref{def: kkt} and establish a corresponding termination criterion.

We define the feasible sets as $\mathcal{X}= \{X \in \mathbb{R}^{n \times m} \mid 0\leq X \leq 1 \}$ and $\mathcal{Y}= \{Y \in \mathbb{R}^{n \times m} \mid \mathbf{1}_n^T Y = b ~ \mathbf{1}_m^T,\, Y \mathbf{1}_m = \mathbf{1}_n \}$. The normal cone to the box constraint set $\mathcal{X}$ at a point $X$ is characterized entry-wise. A matrix $M$ belongs to $N_{\mathcal{X}}(X)$ if and only if each entry $M_{ij}$ satisfies:
\begin{equation*}
    M_{ij} \in
    \begin{cases}
        \{0\}, & \text{if } X_{ij} \in (0,1), \\
        [0, \infty), & \text{if } X_{ij} = 1, \\
        (-\infty, 0], & \text{if } X_{ij} = 0.
    \end{cases}
\end{equation*}

The gradients of these constraint functions are given by $\nabla (Y \mathbf{1}_m - \mathbf{1}_n) = I \otimes \mathbf{1}_m^{\top}$ and $\nabla (\mathbf{1}_n^{\top} Y - b  \mathbf{1}_m^{\top}) = \mathbf{1}_n \otimes I$. Consequently, the normal cone $N_{\mathcal{Y}}(Y)$ is the linear space spanned by the adjoints of these gradient operators:
\begin{equation*}
N_{\mathcal{Y}}(Y) = \left\{ Z \in \mathbb{R}^{n \times m} \;\middle|\;
Z = \nu \mathbf{1}_m^{\top} + \mathbf{1}_n \mu^{\top},\;
 \nu \in \mathbb{R}^n,\; \mu \in \mathbb{R}^m
\right\}.
\end{equation*}

Based on \cref{def: kkt}, a point $(X^*,M^*,Z^*)$ is a KKT point of the problem \cref{eq:l1/2-pro} if it satisfies:
\begin{equation}
\label{kkt_of_pen_pro} 
0 \in A_i^{\top} X_j + G_{ij} + \eta \, \partial |X_{ij}|^{1/2} + M_{ij}^* + Z_{ij}^*.
\end{equation}

Similarly, $(X^*,Y^*,\Lambda^*,M^*,Z^*)$ is a KKT point of problem \cref{eq:l1/2-proxy} if the following conditions hold:
\begin{subequations}
\label{overall_kkt}
\begin{align}
    0 &= Y^* - X^*, \label{overall_kkt_a} \\
    0 &\in -\Lambda_{ij}^* +  \frac{1}{2} A^{\top}_i Y^*_j  + \eta \partial(|X_{ij}^*|)^{1/2} + M_{ij}^*, \quad \ M_{ij}^* \in N_{\mathcal{X}}(X^*_{ij}) \quad \forall i,j, \label{overall_kkt_d} \\
    0 &= \Lambda_{ij}^* +  \frac{1}{2} A^{\top}_i X^*_j  + G_{ij} + Z_{ij}^*, \quad Z^* \in N_{\mathcal{Y}}(Y^*_{ij}) \quad \forall i,j. \label{overall_kkt_e}
\end{align}
\end{subequations}

\begin{proposition}[Equivalence of KKT points]
\label{pro:Equivalence_of_KKT_points}
The KKT points of problem~\cref{eq:l1/2-pro} and problem~\cref{eq:l1/2-proxy} are equivalent; that is, \(X^*\) is a KKT point of one problem if and only if it is a KKT point of the other.
\end{proposition}

\begin{proof}
Let \(X^*\) be a KKT point of problem~\cref{eq:l1/2-pro}. Then there exist multipliers \(M^*\) and \(Z^*\) satisfying the KKT conditions in~\cref{kkt_of_pen_pro}. 
Consequently, there exist \(\Lambda^*\) and \(Y^*\) such that
\begin{equation}
\label{eq:Lambda^*}
\begin{aligned}
&0 = Y^* - X^*,\\
&\Lambda^*_{ij} = \tfrac{1}{2}A_i^\top Y_j^* + \eta\, \partial |X_{ij}^*|^{1/2} + M_{ij}^*, \quad \forall i,j,
\end{aligned}
\end{equation}
which correspond precisely to conditions~\cref{overall_kkt_a} and~\cref{overall_kkt_d}. 
By comparing~\cref{kkt_of_pen_pro} and~\cref{eq:Lambda^*}, we obtain condition~\cref{overall_kkt_e}. 

Conversely, if \(X^*\) is a KKT point of problem~\cref{eq:l1/2-proxy}, then by summing~\cref{overall_kkt_d} and~\cref{overall_kkt_e}, it follows that \(X^*\) satisfies the KKT conditions of problem~\cref{eq:l1/2-pro}. 
Hence, the two problems share identical KKT points.
\end{proof}

However, since ADMM algorithm generates only sequences of iterations $(X^k, Y^k, \Lambda^k)$, we design a termination criterion to check how well these iterates satisfy the KKT conditions in \cref{overall_kkt}.

Consider the subproblem of updating \( X \). Since \( X^{k} \) is computed as an exact solution to its subproblem, it must satisfy the following optimality conditions:
\begin{equation}
\label{kkt_x} 
   0 \in \beta(X^{k}_{ij} - Y^{k-1}_{ij}) - \Lambda^{k-1}_{ij} + \frac{1}{2} A^{\top}_i Y^{k-1}_j  + \eta \partial (|X_{ij}^k|)^{1/2} + M_{ij}^k, \quad M_{ij}^k \in N_{\mathcal{X}}(X^k_{ij}) \quad \forall i,j. 
\end{equation}

Comparing the conditions in \cref{kkt_x} with the target KKT conditions \cref{overall_kkt_d}, it therefore suffices to ensure that the quantity \( h^{k} \) is sufficiently small, where \( h^{k} \) is defined as:
\begin{equation*}
h^{k} = \left\| \frac{1}{2} A (Y^{k} - Y^{k-1}) + \beta (Y^{k-1} - Y^{k}) \right\|_F = \left\| \left( \frac{1}{2} A - \beta I \right) (Y^{k} - Y^{k-1}) \right\|_F.
\end{equation*}

Similarly, the subproblem for \( Y \) \cref{eq:update x 1} is solved exactly, ensuring that the updated \( Y^k \) satisfies:
\begin{equation}
\label{kkt y}
0 = \Lambda^{k}_{ij} + \frac{1}{2} A_{i}^{\top} X^{k}_{j} + G_{ij} + Z^k_{ij}, \quad  Z^k \in N_{\mathcal{Y}}(Y^k)\quad \forall i,j.
\end{equation}

Thus, condition \cref{kkt y} ensures that \cref{overall_kkt_e} is satisfied by the iterates. However, the primal feasibility condition \cref{overall_kkt_a}, \( Y^* = X^* \), is generally not fulfilled until convergence. In summary, the residuals of the algorithm can be determined as \cref{eq:kkt_res}: 
\begin{equation}
\label{eq:kkt_res}
    \begin{aligned}
    & h^{k} = \| \frac{1}{2} A (Y^{k} - Y^{k-1}) + \beta (Y^{k-1} - Y^{k}) \|_F, \\
    & p^{k} = \beta \| Y^{k} - X^{k} \|_F .
\end{aligned}
\end{equation}

Lastly, the overall framework of the ADMM algorithm is outlined in \cref{alg:algorithm framework}.
\begin{algorithm}
\caption{ADMM Algorithm for Solving Problem \cref{eq:l1/2-proxy}}
\label{alg:algorithm framework}
\begin{algorithmic}[1]
\REQUIRE Penalty parameters \( \beta > 0 \), \( \eta > 0 \); 
\STATE \textbf{Initialize:} \( X^0, Y^0 \); \( \Lambda^0 = \mathbf{0} \)

\FOR{k=1,2,3,\dots}
    \STATE \( R^k \gets Y^{k-1} + \frac{1}{\beta} \left( \Lambda^{k-1} - \frac{1}{2} A Y^{k-1} \right) \) \COMMENT{Prepare for \( X \) update}
    \STATE \( X^{k} \gets \text{Solve subproblem \cref{eq: Update_Y_2} entry-wise using } R^k \) 
    \STATE \( B^k \gets X^{k} - \frac{1}{\beta} \left( \Lambda^{k-1} + \frac{1}{2} A X^{k} + G \right) \) \COMMENT{Prepare for \( Y \) update}
    \STATE \( Y^{k} \gets B^k - \frac{1}{n} \mathbf{1}_n \mathbf{1}_n^{\top} B^k + \frac{1}{m} \left( \mathbf{1}_n - B^k \mathbf{1}_m + \frac{1}{n} \mathbf{1}_n \mathbf{1}_n^{\top} B^k \mathbf{1}_m \right) \mathbf{1}_m^{\top} \) \COMMENT{Closed-form solution \cref{eq: up_y}}
    \STATE \( \Lambda^{k} \gets \Lambda^{k-1} + \beta (Y^{k} - X^{k}) \) \COMMENT{Dual variable update}
\ENDFOR
\end{algorithmic}
\end{algorithm}

\section{Theoretical Analysis}
\label{Theoretical} 
This section presents a comprehensive theoretical analysis of \cref{alg:algorithm framework}, culminating in a detailed convergence theorem and a proof of its finite-step termination property.

\subsection{Convergence Proof}
We begin by reformulating \cref{eq:l1/2-proxy} into the following equivalent constrained problem using indicator functions:
\begin{equation*}
\min_{X, Y} \quad 
\frac{1}{2} \left\langle A, X Y^\top \right\rangle 
+ \left\langle G, Y \right\rangle 
+ \eta \|X\|^{1/2}_{1/2} 
 + \mathbb{I}_{\mathcal{Y}}(Y) + \mathbb{I}_{\mathcal{X}}(X), 
\quad \text{s.t.} \quad 
Y - X = 0.
\end{equation*}

The augmented Lagrangian function for this reformulated problem is:
\begin{equation}
\label{eq:augmented-lagrangian} 
\begin{aligned}
    L_{\beta}(X, Y, \Lambda) = & 
    \frac{1}{2} \left\langle A, X Y^\top \right\rangle 
    + \left\langle G, Y \right\rangle 
    + \eta \|X\|^{1/2}_{1/2}  + \mathbb{I}_{\mathcal{Y}}(Y) +\mathbb{I}_{\mathcal{X}}(X) \\
    +& \left\langle \Lambda, Y - X \right\rangle 
    + \frac{\beta}{2} \|Y - X\|_F^2. \\
\end{aligned}
\end{equation}

\begin{assumption}
\label{assu 1}
    The sequence of dual variables $\{\Lambda^k\}$ in \cref{alg:algorithm framework} is bounded.
\end{assumption}
This is a common assumption in the convergence analysis of nonconvex optimization algorithms \cite{bolte2018nonconvex,bourkhissi2025convergence,cohen2022dynamic}. While the boundedness of primal iterates $\{(X^k, Y^k)\}$ can often be guaranteed under certain conditions (e.g., a compact feasible set or a coercive augmented Lagrangian), ensuring the boundedness of the dual sequence $\{\Lambda^k\}$ in nonconvex settings is considerably more challenging. Traditional coercivity-based analysis techniques do not extend straightforwardly to this setting. As noted in~\cite{hallak2023adaptive}, definitive theoretical guarantees on the boundedness of dual variables in general nonconvex ADMM remain an open challenge. In  \cref{sec:experiments}, we show that for practical scenarios, this assumption is valid.

Under \cref{assu 1}, we now establish the boundedness of the iterative sequence generated by \cref{alg:algorithm framework}.

 \begin{proposition}[Boundedness of Iterates]
\label{pro: sq bound}
    Under \cref{assu 1}, the sequence of iterates $\{ (X^k, Y^k, \Lambda^k) \}$ generated by \cref{alg:algorithm framework} is bounded.
\end{proposition}
\begin{proof}
    The boundedness of $\{X^k\}$ follows directly from the box constraint $0 \leq X \leq 1$. From the dual variable update rule $ \Lambda^{k} = \Lambda^{k-1} + \beta(Y^{k} - X^{k}) $, we can express $Y^{k}$ as:
\[
Y^{k} = X^{k} + \frac{1}{\beta}(\Lambda^{k} - \Lambda^{k-1}).
\]

Under \cref{assu 1}, it follows that the sequence $\{Y^k\}$ is also bounded. This completes the proof.
\end{proof}

Leveraging \cref{assu 1} and \cref{pro: sq bound}, we now derive two key properties:
\begin{enumerate}
    \item The sequence of augmented Lagrangian values \( \{L_\beta ( X^k, Y^k, \Lambda^k )\} \) is bounded below.
    \item The norm of updates of dual variables \( \|\Lambda^{k+1} - \Lambda^k\|_F \) is bounded by the norms of the updates of the primary variables \( \|X^{k+1} - X^k\|_F \) and \( \|Y^{k+1} - Y^k\|_F \).  
\end{enumerate}

These properties are formally stated in the following propositions.
\begin{proposition}[Boundedness Below of Augmented Lagrangian]
    Under \cref{assu 1}, the sequence of augmented Lagrangian values $ \{ L_\beta ( X^k, Y^k, \Lambda^k ) \} $ is bounded below.
\end{proposition}
\begin{proof}
We begin by rewriting the augmented Lagrangian in the following form:
\[
\begin{aligned}
L_{\beta}(X^k, Y^k, \Lambda^k) 
&= \frac{1}{2}\left\langle A, X^k (Y^k)^\top \right\rangle + \left\langle G, Y^k \right\rangle + \eta \|X^k\|^{1/2}_{1/2}  + \left\langle \Lambda^k, Y^k - X^k \right\rangle \\
&\quad + \mathbb{I}_{\mathcal{Y}}(Y^k)  + \mathbb{I}_{\mathcal{X}}(X^k) + \frac{\beta}{2} \| Y^k - X^k \|_F^2 \\
&= \frac{1}{2} \left\langle A, X^k (Y^k)^\top \right\rangle + \left\langle G, Y^k \right\rangle + \eta \|X^k\|^{1/2}_{1/2} - \frac{1}{2\beta} \| \Lambda^k \|_F^2 \\
&\quad + \frac{\beta}{2} \left\| Y^k - X^k + \frac{1}{\beta} \Lambda^k \right\|_F^2  \quad + \mathbb{I}_{\mathcal{Y}}(Y^k)  + \mathbb{I}_{\mathcal{X}}(X^k).
\end{aligned}
\]
The indicator functions \( \mathbb{I}_{\mathcal{Y}}(Y^k) \) and \( \mathbb{I}_{\mathcal{X}}(X^k) \) are nonnegative. The remaining terms are continuous functions evaluated over the bounded sequence $\{ (X^k, Y^k, \Lambda^k) \}$ (guaranteed by \cref{assu 1} and \cref{pro: sq bound}). Therefore, the entire expression is bounded below.
\end{proof}

\begin{proposition}[Bound on Dual Variable Updates Norm]
\label{pro: lambda delta}
    Under \cref{assu 1}, there exists a constant $\kappa > 0$ such that the update norm of the dual variables satisfies
    \[
    \|\Lambda^{k+1} - \Lambda^{k}\|_F \leq \frac{1}{2}\|A\|_F\|X^{k+1} - X^{k}\|_F + \kappa \| Y^{k+1} - Y^{k} \|_F.
    \]
\end{proposition}
\begin{proof}
We begin with the optimality condition for the $Y$ subproblem from \cref{first order for differentiable}:
\[
\beta(Y^{k} - X^{k}) + \Lambda^{k-1} + \frac{1}{2}AX^{k} + G \in -N_{\mathcal{Y}}(Y^{k}),
\]
where $\mathcal{Y} = \{Y \mid \mathbf{1}_n^T Y = b~\mathbf{1}_m^T,\ Y \mathbf{1}_m = \mathbf{1}_n\}$.
Combining this with the update rule for the dual variable, we obtain:
$$
\Lambda^{k} + \frac{1}{2}AX^{k} + G\in - N_{\mathcal{Y}}(Y^{k}).
$$

Replacing $k$ by $k+1$ gives:$$
\Lambda^{k+1} + \frac{1}{2}AX^{k+1}+ G \in - N_{\mathcal{Y}}(Y^{k+1}).
$$

By \cref{pro: sq bound}, the sequence is bounded, hence there exists a constant $\rho > 0$ such that:
\[
\left\|\Lambda^{k} + \frac{1}{2}AX^{k} + G\right\|_F \leq \rho \quad \text{for all } k.
\]

Define the truncated normal cone $\bar{N}_{\mathcal{Y}}(Y) = N_{\mathcal{Y}}(Y) \cap \mathbb{B}_{\rho}$ using $\rho$ as the radius.
Consequently, for all $k \geq 1$ we have :
\[
\Lambda^{k} + \frac{1}{2}AX^{k} + G \in -\bar{N}_{\mathcal{Y}}(Y^{k}).
\]
Since $\mathcal{Y}$ satisfies the Hausdorff Lipschitz condition in \cref{distH}, there exists $\kappa > 0$ such that:
\[
\operatorname{dist}_H\left(\bar{N}_{\mathcal{Y}}(Y^{k+1}), \bar{N}_{\mathcal{Y}}(Y^{k})\right) \leq \kappa \|Y^{k+1} - Y^{k}\|_F.
\]

Therefore, for any $k \geq 1$:
\[
\left\|\Lambda^{k+1} - \Lambda^{k} + \frac{1}{2}A(X^{k+1} - X^{k})\right\|_F \leq \kappa \|Y^{k+1} - Y^{k}\|_F.
\]

Applying the triangle inequality yields for any $k \geq 1$:
\[
\|\Lambda^{k+1} - \Lambda^{k}\|_F \leq \frac{1}{2}\|A\|_F\|X^{k+1} - X^{k}\|_F + \kappa \|Y^{k+1} - Y^{k}\|_F.
\]

This completes the proof.
\end{proof}

Leveraging the property of restricted prox-regularity stated in \cref{pro:reg}, we establish the following proposition.
\begin{proposition}[Restricted Prox-regularity Inequality]
\label{pro: prox-re}
    Define:
    $$f_2(X) = \eta \|X\|^{1/2}_{1/2}  + \mathbb{I}_{\{0 \leq X \leq 1\}}(X),$$
    there exists a constant $\gamma > 0$ such that:
    \begin{equation}
    f_{2}(X^k) - f_2(X^{k+1}) - \left\langle \Lambda^{k} + \beta(Y^k - X^{k+1}) - \frac{1}{2} A Y^{k+1}, X^{k} - X^{k+1} \right\rangle \geq -\frac{\gamma}{2} \| X^k - X^{k+1} \|_F^2.
    \end{equation}
\end{proposition}
\begin{proof}
From the optimality condition of the $X$ subproblem, we define
\[
v_1^{k+1} = \Lambda^{k} + \beta(Y^k - X^{k+1}) - \frac{1}{2} A Y^{k+1} \in \partial f_2(X^{k+1}).
\]
By the boundedness of the iterative sequence (guaranteed by \cref{pro: sq bound}), there exists a constant \( E > 0 \) such that:
\[
\|v_1^{k+1}\|_F = \left\| \Lambda^{k} + \beta(Y^k - X^{k+1}) - \frac{1}{2} A Y^{k+1} \right\|_F \leq E.
\]

Since $f_{2}$ is restricted prox-regular and $\|v_1^{k+1}\|_F \leq E$, there exists a constant $\gamma > 0$ such that:
\[
f_{2}(X^k) - f_2(X^{k+1}) - \left\langle \Lambda^{k} + \beta(Y^k - X^{k+1}) - \frac{1}{2} A Y^{k+1}, X^{k} - X^{k+1} \right\rangle \geq -\frac{\gamma}{2} \| X^k - X^{k+1} \|_F^2.
\]
\end{proof}

\cref{pro: decrease-y } analyzes the decrease in the augmented Lagrangian during the $X$-update step.

\begin{proposition}[Decrease in Augmented Lagrangian during X Update]
\label{pro: decrease-y }
    Let $f_1(X, Y) = \frac{1}{2} \left\langle A, X Y^\top \right\rangle + \left\langle G, Y \right\rangle$. Then the decrease in the augmented Lagrangian with respect to $X$ satisfies:
    \begin{equation*}
        L_\beta(X^{k}, Y^k, \Lambda^k) - L_\beta(X^{k+1}, Y^{k}, \Lambda^{k}) \geq \frac{\beta - \gamma} {2}  \| X^{k+1} - X^{k} \|_F^2.
    \end{equation*}
\end{proposition}

\begin{proof}
    By an easy calculation, we have the following equality:
\begin{equation}
\label{eq: lp-g}
    f_1(X^{k}, Y^k) - f_1(X^{k+1}, Y^{k}) - \left\langle \frac{1}{2} A Y^{k}, X^{k} - X^{k+1} \right\rangle = 0.
\end{equation}

The decrease in the augmented Lagrangian with respect to $X$ can be expressed as:
\begin{equation*}
    \begin{aligned}
    & L_\beta(X^{k}, Y^k, \Lambda^k) - L_\beta(X^{k+1}, Y^{k}, \Lambda^{k}) \\
    = & \frac{1}{2} \left\langle A, X^{k} (Y^{k})^\top \right\rangle + \eta \|X^k\|^{1/2}_{1/2}  - \frac{1}{2} \left\langle A, X^{k+1} (Y^{k})^\top \right\rangle - \eta \|X^{k+1}\|^{1/2}_{1/2} \\
    & - \left\langle \Lambda^{k}, X^{k} - X^{k+1} \right\rangle + \frac{\beta}{2} \| Y^{k} - X^{k} \|_F^2 - \frac{\beta}{2} \| Y^{k} - X^{k+1} \|_F^2 \\
    = & \frac{1}{2} \left\langle A, X^{k} (Y^{k})^\top \right\rangle + \eta \|X^k\|^{1/2}_{1/2} - \frac{1}{2} \left\langle A, X^{k+1} (Y^{k})^\top \right\rangle - \eta \|X^{k+1}\|^{1/2}_{1/2}\\
    & - \left\langle \Lambda^{k}, X^{k} - X^{k+1} \right\rangle + \frac{\beta}{2} \| X^{k} - X^{k+1} \|_F^2 + \beta \left\langle X^{k+1} - Y^{k}, X^{k} - X^{k+1} \right\rangle \\
    = & \frac{\beta}{2} \| X^{k} - X^{k+1} \|_F^2 + \frac{1}{2} \left\langle A, X^{k} (Y^{k})^\top \right\rangle - \frac{1}{2} \left\langle A, X^{k+1} (Y^{k})^\top \right\rangle - \left\langle \frac{1}{2} A Y^{k}, X^{k} - X^{k+1} \right\rangle \\
    & + \eta \|X^k\|^{1/2}_{1/2} - \eta \|X^{k+1}\|^{1/2}_{1/2}+ \left\langle -\Lambda^{k} + \beta(X^{k+1} - Y^{k}) + \frac{1}{2} A Y^{k}, X^{k} - X^{k+1} \right\rangle \\
    \geq & \frac{\beta - \gamma }{2} \| X^{k+1} - X^{k} \|_F^2.
    \end{aligned}
\end{equation*}

The second equality utilizes the identity $\|b + c\|^2 - \|a + c\|^2 = \|b - a\|^2 + 2 \langle a + c, b - a \rangle$ (cosine rule), and the final inequality follows from \cref{eq: lp-g} and \cref{pro: prox-re}.
\end{proof}

\begin{proposition}[Convergence of Iterate Differences]
\label{pro: to 0}
    Under \cref{assu 1} and the condition that
    \begin{equation}
    \label{assump 2}
       \beta \geq \max\left\{ 2\sqrt{2}\kappa,\ \frac{\gamma  + \sqrt{(\gamma  + 4\|A\|_F^2}}{2} \right\},
    \end{equation}
    we have $\sum_{k=1}^{\infty} \left\| X^{k+1} - X^{k} \right\|_F^2 < \infty$ and $\sum_{k=1}^{\infty} \left\| Y^{k+1} - Y^{k} \right\|_F^2 < \infty$.
\end{proposition}
\begin{proof}
From \cref{pro: decrease-y }, the decrease in the augmented Lagrangian during the $X$ update satisfies:
\begin{equation*}
    L_\beta(X^{k}, Y^k, \Lambda^k) - L_\beta(X^{k+1}, Y^{k}, \Lambda^{k}) \geq \frac{\beta - \gamma }{2} \| X^{k+1} - X^{k} \|_F^2.
\end{equation*}

Due to the strong convexity of the $Y$ subproblem, we have:
\begin{equation*}
    L_\beta(X^{k+1}, Y^k, \Lambda^k) - L_\beta(X^{k+1}, Y^{k+1}, \Lambda^{k}) \geq \frac{\beta}{4} \| Y^{k+1} - Y^{k} \|_F^2.
\end{equation*}

From the dual variable update rule, we obtain:
\begin{equation*}
    L_\beta(X^{k+1}, Y^{k+1}, \Lambda^k) - L_\beta(X^{k+1}, Y^{k+1}, \Lambda^{k+1}) = \frac{1}{\beta} \| \Lambda^{k+1} - \Lambda^{k} \|_F^2.
\end{equation*}

Define $L_a = \beta - \gamma$. Then:
\begin{equation}
\label{eq: temp_delta_l}
\begin{aligned}
& L_\beta\left(X^{k+1}, Y^{k+1}, \Lambda^{k+1}\right) - L_\beta\left(X^{k}, Y^{k}, \Lambda^{k}\right) \\
\leq & -\frac{\beta}{4} \left\| Y^{k+1} - Y^{k} \right\|_F^2 - \frac{L_a}{2} \left\| X^{k+1} - X^{k} \right\|_F^2 + \frac{1}{\beta} \left\| \Lambda^{k+1} - \Lambda^{k} \right\|_F^2.
\end{aligned}
\end{equation}

From \cref{pro: lambda delta}, we obtain the bound:
\begin{equation*}
  \| \Lambda^{k+1} - \Lambda^{k} \|_F^2 \leq \frac{1}{2} \|A\|_F^2 \| X^{k+1} - X^{k} \|_F^2 + 2\kappa^2 \| Y^{k+1} - Y^{k} \|_F^2.
\end{equation*}

Substituting this bound into \cref{eq: temp_delta_l} yields:
\begin{equation*}
\begin{aligned}
& L_\beta\left(X^{k+1}, Y^{k+1}, \Lambda^{k+1}\right) - L_\beta\left(X^{k}, Y^{k}, \Lambda^{k}\right) \\
\leq & -\left(\frac{\beta}{4} - \frac{2\kappa^2}{\beta}\right) \left\| Y^{k+1} - Y^{k} \right\|_F^2 - \left( \frac{L_a}{2} - \frac{\|A\|_F^2}{2\beta} \right) \| X^{k+1} - X^{k} \|_F^2.
\end{aligned}
\end{equation*}

Under \cref{assump 2}, we have $\frac{\beta}{4} - \frac{2\kappa^2}{\beta} \geq 0$ and $\frac{L_a}{2} - \frac{\|A\|_F^2}{2\beta} \geq 0$. Therefore:
\begin{equation*}
\begin{aligned}
& \left( \frac{\beta}{4} - \frac{2\kappa^2}{\beta} \right) \left\| Y^{k+1} - Y^{k} \right\|_F^2 + \left( \frac{L_a}{2} - \frac{\|A\|_F^2}{2\beta} \right) \| X^{k+1} - X^{k} \|_F^2 \\
& \leq L_\beta\left(X^{k}, Y^{k}, \Lambda^{k}\right) - L_\beta \left(X^{k+1}, Y^{k+1}, \Lambda^{k+1}\right).
\end{aligned}
\end{equation*}

Summing this inequality from $k = 1$ to $N$ yields:
\begin{equation*}
\begin{aligned}
& \left( \frac{\beta}{4} - \frac{2\kappa^2}{\beta} \right) \sum_{k=1}^N \left\| Y^{k+1} - Y^{k} \right\|_F^2 + \left( \frac{L_a}{2} - \frac{\|A\|_F^2}{2\beta} \right) \sum_{k=1}^N \| X^{k+1} - X^{k} \|_F^2 \\
& \leq L_\beta \left(X^{1}, Y^{1}, \Lambda^{1}\right) - L_\beta \left(X^{N+1}, Y^{N+1}, \Lambda^{N+1}\right).
\end{aligned}
\end{equation*}

Taking the limit as $N \to \infty$, we obtain:
\begin{equation*}
  \left( \frac{\beta}{4} - \frac{2\kappa^2}{\beta} \right) \sum_{k=1}^{\infty} \left\| Y^{k+1} - Y^{k} \right\|_F^2 + \left( \frac{L_a}{2} - \frac{\|A\|_F^2}{2\beta} \right) \sum_{k=1}^{\infty} \| X^{k+1} - X^{k} \|_F^2 < \infty.
\end{equation*}

This completes the proof of the proposition.
\end{proof}

\begin{theorem}[Global Convergence]
\label{last conv theo}
    Let $X^*$ be a cluster point of the sequence $\{X^k\}$. 
    If \cref{assu 1} and \cref{assump 2} hold, and if $\eta >  \frac{2}{\sqrt{2}-1} \sqrt{n}(\|G\|_F + \sqrt{n}\|A\|_F )$, then $X^*$ is the limit point of sequence $\{X^k\}$, and is a local minimizer of the original problem \cref{eq:ori-pro}.
\end{theorem}
\begin{proof}
    Let $\{ X^*, Y^*, \Lambda^* \}$ be a cluster point of the bounded sequence $\{ X^k, Y^k, \Lambda^k \}$. Then there exists an index set $\mathbb{K} \subseteq \mathbb{N}$ such that:
\[
\lim_{k \in \mathbb{K}, k \to \infty} (X^k, Y^k, \Lambda^k) = (X^*, Y^*, \Lambda^*).
\]

By \cref{pro: to 0}, we have:
\[
\lim_{k \in \mathbb{K}, k \to \infty} (X^{k+1} - X^{k}) = 0 \quad \text{and} \quad \lim_{k \in \mathbb{K}, k \to \infty} (Y^{k+1} - Y^{k}) = 0.
\]

By \cref{pro: lambda delta}, we also have $\lim_{k \in \mathbb{K}, k \to \infty} (\Lambda^{k+1} - \Lambda^{k}) = 0$. Consequently, it follows that:
\[
\lim_{k \in \mathbb{K}, k \to \infty} (Y^{k} - X^{k}) = Y^* - X^* = 0.
\]

Since $Y^*$ satisfies $\mathbf{1}_n^\top Y^* = b~ \mathbf{1}_m^\top$ and $Y^* \mathbf{1}_m = \mathbf{1}_n$, and $Y^* = X^*$, it follows that $\mathbf{1}_n^\top X^* = b~ \mathbf{1}_m^\top$, $X^* \mathbf{1}_m = \mathbf{1}_n$, and $0 \leq X^* \leq 1$.
Therefore, $X^*$ is a feasible point for \cref{eq:l1/2-proxy}.

Since $X^*$ and $Y^*$ are KKT points for the $X$ and $Y$ subproblems, respectively, the following optimality conditions hold according to \cref{kkt_x} and \cref{kkt y}:
\begin{subequations}
\label{eq:kkt-conditions}
\begin{align}
    & 0 \in -\Lambda^*_{ij} + \frac{1}{2}A_i^\top X_j^* + \eta \partial |X_{ij}^{*}|^{1/2} + M_{ij}^*, \quad M_{ij}^* \in N_{\mathcal{X}}(X^{*}_{ij}),\quad \forall i,j,\label{eq:kkt-condition1} \\
    & 0 = \Lambda_{ij}^* + \frac{1}{2}A_i^\top X_j^* + G_{ij} + Z_{ij}^*, \quad Z_{ij}^* \in N_{\mathcal{Y}}(X^{*}_{ij}), \quad \forall i,j.\label{eq:kkt-condition2} 
\end{align}
\end{subequations}

Summing these two inclusions yields:
\[
0 \in A_i^\top X_j^* + \eta \partial |X_{ij}^{*}|^{1/2} + G_{ij} + M_{ij}^* + Z_{ij}^*, \quad M_{ij}^* \in N_{\mathcal{X}}(X^{*}_{ij}), \quad Z_{ij}^* \in N_{\mathcal{Y}}(X^{*}_{ij}),~ \forall i,j.
\]

This shows that the cluster point $X^*$ satisfies the KKT conditions for the penalty problem \cref{eq:l1/2-pro} as specified in \cref{kkt_of_pen_pro}, establishing that $X^*$ is a KKT point for \cref{eq:l1/2-pro}. By \cref{01kkt},  $X^*$ is also a feasible point and local minimizer of the problem \cref{eq:ori-pro}. 
Since the set of cluster points of a bounded sequence
is connected, and the feasible set of the binary problem \cref{eq:ori-pro} is discrete, thus the iterative
sequence $\{ X^k\}$ can have only one cluster point. This implies that the entire sequence
$\{ X^k\}$ converges to $X^*$. 
\end{proof}

\subsection{Finite-step termination property}
In this subsection, we demonstrate the finite-step termination property of \cref{alg:algorithm framework}. It highlights an appealing feature of \cref{alg:algorithm framework}.
\begin{theorem}
\label{thm:finite}
Assume that for the limit point \( X^* \) of \( \{X^k\} \), strict complementarity holds such that \( M_{ij}^* > 0 \) whenever \( X_{ij}^* = 1 \). With \cref{assu 1}, \cref{assump 2}, and \(  \eta > \frac{2}{\sqrt{2}-1} \sqrt{n}(\|G\|_F + \sqrt{n}\|A\|_F ), \beta > \frac{1}{4} \eta \) holds,  \cref{alg:algorithm framework} terminates in finite iterations.
\end{theorem}
\begin{proof}
Recall that $R^{k} = Y^{k-1}+\frac{1}{\beta}\left(\Lambda^{k-1}-\frac{1}{2} A Y^{k-1}\right) $, by the boundness of the $\{X^k, Y^k, \Lambda^k\}$ sequence, there exists $N>0$, such that $|R_{ij}^{k}| <N$. By \cref{last conv theo}, we have: $\lim_{k \to \infty }X^k = X^* $.  Thus $\forall \zeta \in \left(0, \min\{ (\frac{\eta}{\beta N})^2) ,1-(\frac{\eta}{4\beta})^{2/3} \}\right)$, there exists $K_1>0$, such that $\forall k >K_1, |X_{ij}^k -X_{ij}^*| < \zeta$. Then $X_{ij}^k \in [0,\zeta)~ \text{or} ~ X_{ij}^k \in (1-\zeta,1], ~\forall k > K_1 $, depending on the corresponding $X^*_{ij}$ being 0 or 1. 

We first consider the entries of  
 $X^*$ where the value is 0. Suppose that there still exists $X_{st}^k \in (0, \zeta)$ in $X^k, ~k> K_1$. Consider the subproblem for $X$, define the objective function in \cref{eq: Update_Y_1} as $obj(\cdot)$, then:
\begin{equation*}
\begin{aligned}
    obj(X_{st}^k)  - obj(0) =& \frac{\beta}{2} (X_{st}^k)^2 - \beta R_{st}^k X_{st}^k + \eta (X_{st}^k)^{1/2} \\
    =& \beta X_{st}^k \left(\frac{1}{2} X_{st}^k +  \frac{\eta}{\beta}(X_{st}^k)^{-1/2} -  R_{st}^k\right) \\
    >& \beta X_{st}^k \left( -N +  \frac{\eta}{\beta} \zeta^{-1/2} \right) \\
    >& 0 \quad \text{for} \quad \zeta \in \left(0, \min \left\{ \left(\frac{\eta}{\beta N}\right)^2, 1-\left(\frac{\eta}{4\beta}\right)^{2/3} \right\} \right),
\end{aligned}
\end{equation*}
which makes a contradiction. Thus there does not contain any $X_{st}^k  \in (0,\zeta)$ in $X^k, ~ k>K_1$. 

Next, if there exists $X_{st}^k \in (1-\zeta,1), ~ k > K_1$, then by the first-order optimal condition \cref{x*}, we have: $R_{st}^k=X_{st}^k+\frac{\eta}{2\beta}(X_{st}^k)^{-1/2}$. Define the function $l(x)= x + \frac{\eta}{2\beta}x^{-1/2}$, then $l(x)$ is monotonically increasing in $(1-\zeta,1)$. Thus, we have: $R_{st}^k < 1+\frac{\eta}{2\beta}$.

For $R^*_{st} = X^*_{ij} + \frac{1}{\beta} (\Lambda_{st}^* - \frac{1}{2}A_i^{\top}X_j^*)$, by the optimality condition \cref{eq:kkt-condition1}, we have: $R_{st}^* \in X_{st}^* + \frac{\eta}{2\beta} (X_{st}^*)^{-1/2} + \frac{1}{\beta}M_{ij}^*=1+\frac{\eta}{2\beta} +\frac{1}{\beta}M_{ij}^* > 1 + \frac{\eta}{2\beta}$.

We now leverage the convergence of \(\{R_{st}^k\}\) to \(R_{st}^*\). For any \(\delta > 0\), there exists \(K_2 \in \mathbb{N}\) such that for all \(k > K_2\):
$|R_{st}^k - R_{st}^*| < \delta$.

Choosing \(\delta = \frac{1}{2}\left(R_{st}^* - \left(1 + \frac{\eta}{2\beta}\right)\right)\) gives:
\[
R_{st}^k > R_{st}^* - \frac{1}{2}\left(R_{st}^* - 1 - \frac{\eta}{2\beta}\right) = \frac{R_{st}^*}{2} + \frac{1}{2} + \frac{\eta}{4\beta} > \frac{1}{2}\left(1 + \frac{\eta}{2\beta}\right) + \frac{1}{2} + \frac{\eta}{4\beta} = 1 + \frac{\eta}{2\beta}.
\]

Thus, for all \(k > K_2\), \(R_{st}^k > 1 + \frac{\eta}{2\beta}\). Let \(K = \max\{K_1, K_2\}\). For \(k > K\), we derive a contradiction. This contradiction implies no such \(k > K\) can exist, meaning \(X_{st}^k \notin (1 - \zeta, 1)\) for all sufficiently large \(k\). By the algorithm's update rule, this forces \(X_{st}^k = 1\) within finite steps, establishing finite termination.
\end{proof}

In the above theorem, a strict complementarity  assumption is imposed. The strict complementarity condition strengthens the standard KKT complementarity by excluding degenerate cases in which both a constraint and its multiplier vanish simultaneously. 
In semidefinite and quadratic programming, it is often assumed as a key regularity condition ensuring well-behaved optimality, active-set identification, and convergence of interior-point, augmented Lagrangian, and ADMM-type methods \cite{basu2024complexityadpsc3,chua2008analyticitysdpsc4,de2019strictsdpsc5,SDPsc1,SDPSC2,sc,Boleysc,han2013localsc,raghunathan2014admmsc}

\section{Experimental results}
\label{sec:experiments}
This section presents experimental validations of the proposed algorithm for solving problem \cref{eq:ori-pro}, focusing on two key aspects: its application in MMD-based mini-batch selection (a practical scenario for deep learning training) and analyses of experimental test cases. 

\subsection{Problem Setup}
MMD is a popular statistical metric for quantifying the discrepancy between the marginal distributions of two datasets. It operates by mapping data into a high-dimensional feature space using kernel functions, then comparing the mean embeddings of the distributions. This kernel-based approach enables MMD to capture complex nonlinear distributional differences. This metric ranges from zero to positive values, where smaller MMD values indicate greater distributional similarity.
Consider two sets of samples $\mathbb{A} = \{a_1, a_2, \dots, a_{n_1}\}$ and $\mathbb{C} = \{c_1, c_2, \dots, c_{n_2}  \}$ drawn from different distributions. The empirical MMD between their distributions is given by~\cite{gretton2006kernel}:
\begin{equation*}
\left\|\frac{1}{n_1}\sum_{i=1}^{n_1} \phi(a_i) - \frac{1}{n_2} \sum_{i=1}^{n_2} \phi(c_i) \right\|_{\mathcal{H}}^2,
\end{equation*}
where $\phi: \mathcal{X} \to \mathcal{H}$ maps samples from input space $\mathcal{X}$ to a universal Reproducing Kernel Hilbert Space (RKHS) $\mathcal{H}$.

The training of deep learning models relies on the use of mini-batches. Typically, a training set $\mathbb{D}$ comprising $n$ samples is randomly partitioned into $m = \left\lceil n/b\right\rceil$ mini-batches, denoted as ${\mathbb{B}_1, \mathbb{B}_2, \dots, \mathbb{B}_m}$ before each epoch, where $b$ is a predefined hyperparameter known as batch size. These mini-batches are then sequentially fed into the model for training.
While the mini-batch gradient serves as an unbiased estimator of the true gradient over the entire training set, it can suffer from high variance due to the inherent variability introduced by random sampling~\cite{2010self-paced-learning,2015-pmlrzhaoa15-IS-pSMD-p-SDCA}. Increasing the batch size can reduce this variance but may suffer slower convergence speed~\cite{sample-size-matters-2012}. Carefully designed mini-batch selection methods can enhance the quality of gradient estimators without enlarging the batch size, thus balancing this trade-off.

Banerjee et al.~\cite{2021-AAAI-MMD} posited that the distribution of the selected samples should closely align with that of the unselected samples to ensure an accurate estimation of the entire training set. For the detail formulation, we refer readers to see \cite{2021-AAAI-MMD}. In contrast, we minimize the distributional discrepancy between all batches $\{\mathbb{B}_i\}_{i=1}^m$ and $\mathbb{D}$: 
\begin{equation}
\label{eq:frame1} \underset{\mathbb{B}_i,i=1,\dots,m}{\min} \quad  \frac{1}{m}\sum_{i=1}^{m}\left\| \frac{1}{b} \sum_{x_j \in \mathbb{B}_i}  \phi(x_j) - \frac{1}{n} \sum_{x_j \in \mathbb{D}}  \phi(x_j)\right\|^2_{\mathcal{H}},
\end{equation}
where $b$ represents the batch size, and $m$ denotes the number of batches. We define a binary assignment matrix $M \in \{0,1\}^{n \times m}$ with:
\[
M_{ij} = \begin{cases} 
1 & \text{if sample $i$ belongs to batch $\mathbb{B}_j$} \\ 
0 & \text{otherwise}
\end{cases}
\]
subject to $\sum_{j=1}^m M_{ij} = 1 ,~ \forall i=1,2,\dots,n$ and $\sum_{i=1}^nM_{ij} = b ,~ \forall j=1,2,\dots,m$ (each sample in exactly one batch, and each batch has $b$ samples).
Each row of $M$ corresponds to a unique sample in $\mathbb{D}$, preserving the original data ordering. This matrix representation transforms problem \cref{eq:frame1} into:
\begin{equation*}
\begin{aligned}
    \min_{M \in \{0,1\}^{n \times m}} \quad & \frac{1}{m} \left\| \frac{1}{b} M^\top \Phi - \frac{1}{n} \mathbf{1}_{n \times m}^\top \Phi \right\|_{\mathcal{H}}^2 \\
    \text{s.t.} \quad 
    & \mathbf{1}_n^\top M = b~ \mathbf{1}_m^\top, \\
    & M \mathbf{1}_m = \mathbf{1}_n,
\end{aligned}
\end{equation*}
where $\Phi = [\phi(x_1), \phi(x_2), \dots, \phi(x_n)]^\top$, and $\mathbf{1}_{n \times m}$ denotes the $n \times m$ all-ones matrix. Applying the kernel trick, we expand the objective to:
\begin{equation}
\label{eq: trace}
\begin{aligned}
    \min_{M \in \{0,1\}^{n \times m}} \quad & \frac{1}{m} \operatorname{tr} \left( \frac{1}{b^2} M^\top \Psi M - \frac{2}{n b} M^\top \Psi \mathbf{1}_{n \times m} + \frac{1}{n^2} \mathbf{1}_{n \times m}^\top \Psi \mathbf{1}_{n \times m} \right) \\
    \text{s.t.} \quad 
    & \mathbf{1}_n^\top M = b ~\mathbf{1}_m^\top, \\
    & M \mathbf{1}_m = \mathbf{1}_n, \\
\end{aligned}
\end{equation}
where $\Psi = \Phi \Phi^\top$ is the kernel matrix, which encodes pairwise sample similarities. As suggested in ~\cite{zhu2019aligning,zhu2019multi,zhu2020deep}, we adopt the Gaussian kernel. Leveraging $\Psi$'s symmetry, problem \cref{eq: trace} is a special case for problem \cref{eq:ori-pro} with $A = \frac{1}{b^2} \Psi$ and $G = -\frac{2}{nb} \Psi \mathbf{1}_{n \times m}$. 

We conduct comprehensive experiments evaluating the performance of our proposed algorithm for solving problem \cref{eq: trace}. Further, we compare the resulting mini-batch selection results against two baselines:
\begin{itemize}
    \item \textbf{Random}: Batches are sampled uniformly at random without replacement from the training set.
    \item \textbf{Vector}: The method is proposed in \cite{2021-AAAI-MMD}. Noting that their LP relaxation may not yield binary solutions, we implement a special case of our framework where $X$ reduces to a single column, which is structurally equivalent to their formulation.
    \item \textbf{Matrix}: Our proposed framework is applied on problem \cref{eq: trace}.
\end{itemize}

We evaluate our approach with several popular optimization methods: SGD, weighted SGD (WSGD), weighted SGD with momentum (WMSGD), and Adam on three datasets: one synthetic dataset, MNIST and CIFAR-10 Dataset, as detailed in next two subsections.
All results are averaged over three independent runs.

\subsection{Performance on Synthetic Dataset}
We construct a benchmark problem involving the sum of $K$ groups of quadratic functions. Each group $i$ contains $N_i$ quadratic subproblems, with total dimension $n = \sum_{i=1}^K N_i$:
\begin{equation*}
\min_{x} \sum_{i=1}^K \sum_{j=1}^{N_i} \left( \frac{1}{2} x^\top Q_{ij} x + g_{ij}^\top x \right),
\end{equation*}
where each $Q_{ij}$ is a positive definite matrix. The Gaussian kernel matrix is computed using the vectors $-Q_{ij}^{-1}g_{ij}$. All experiments were conducted for 400 epochs with a mini-batch size of 4. The theoretical optimum is given by:
\begin{equation*}
\begin{aligned}
    x^* &= -\left(\sum_{i=1}^K \sum_{j=1}^{N_i} Q_{ij}\right)^{-1} \left(\sum_{i=1}^K \sum_{j=1}^{N_i} g_{ij}\right), \\
    f^* &= \frac{1}{2} (x^*)^\top \left(\sum_{i=1}^K \sum_{j=1}^{N_i} Q_{ij}\right) x^* + \left(\sum_{i=1}^K \sum_{j=1}^{N_i} g_{ij}\right)^\top x^*.
\end{aligned}
\end{equation*}
\begin{table}[h!]
\caption{Performance comparison of error of $x$ and $f$ for different methods over various epochs and dataset sizes.}
\label{sy overall}
\begin{tabular}{c|c|cccccc}
\hline
\multirow{2}{*}{n}    & \multirow{2}{*}{Epoch} & \multicolumn{3}{c}{$\log_{10}(\text{Error of } x)$} & \multicolumn{3}{c}{$\log_{10}(\text{Error of } f)$} \\ \cline{3-8} 
                      &                        & \textbf{Random} & \textbf{Vector} & \textbf{Matrix} & \textbf{Random} & \textbf{Vector} & \textbf{Matrix} \\ \hline
\multirow{4}{*}{80}   & 100                    & -5.13           & -6.59           & \textbf{-7.10}  & -7.38           & -10.82          & \textbf{-11.99} \\
                      & 200                    & -5.81           & -7.00           & \textbf{-7.90}  & -8.41           & -10.90          & \textbf{-12.67} \\
                      & 300                    & -5.84           & -7.56           & \textbf{-7.95}  & -8.45           & -11.89          & \textbf{-12.73} \\
                      & 400                    & -5.45           & -7.34           & \textbf{-7.84}  & -7.51           & -11.67          & \textbf{-12.66} \\ \hline
\multirow{4}{*}{400}  & 100                    & -5.80           & -7.18           & \textbf{-7.90}  & -6.62           & -9.78           & \textbf{-11.15} \\
                      & 200                    & -5.90           & -7.71           & \textbf{-8.05}  & -6.85           & -10.71          & \textbf{-11.32} \\
                      & 300                    & -6.12           & -7.80           & \textbf{-8.38}  & -7.31           & -10.80          & \textbf{-11.93} \\
                      & 400                    & -6.36           & -7.79           & \textbf{-8.17}  & -7.74           & -10.78          & \textbf{-11.49} \\ \hline
\multirow{4}{*}{1200} & 100                    & -6.73           & -8.70           & \textbf{-8.80}  & -7.42           & -11.42          & \textbf{-11.78} \\
                      & 200                    & -6.70           & -8.49           & \textbf{-8.68}  & -7.49           & -11.12          & \textbf{-11.57} \\
                      & 300                    & -6.94           & \textbf{-8.98}  & -8.96           & -7.92           & \textbf{-12.21} & -11.98          \\
                      & 400                    & -7.44           & -8.12           & \textbf{-8.54}  & -7.44           & -10.39          & \textbf{-11.16} \\ \hline
\end{tabular}
\end{table}

\begin{figure}[h!]
\centering
\subfloat[\centering $\beta =20, ~\eta = 0.01$]{\label{fig:no01}\includegraphics[width=0.5\textwidth]{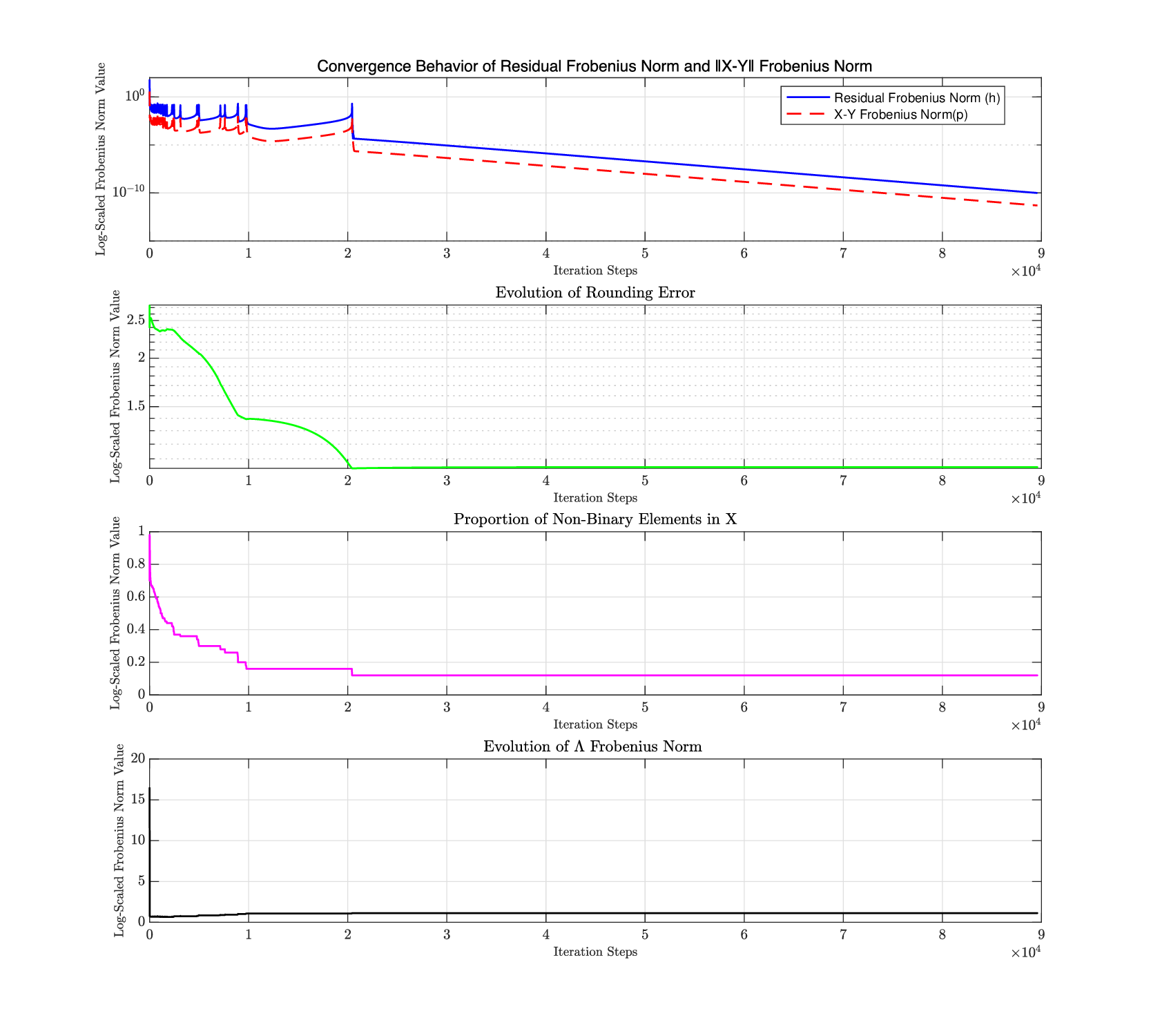}}
\subfloat[\centering $\beta =20, ~\eta = 1$]{\label{fig:to01}\includegraphics[width=0.5\textwidth]{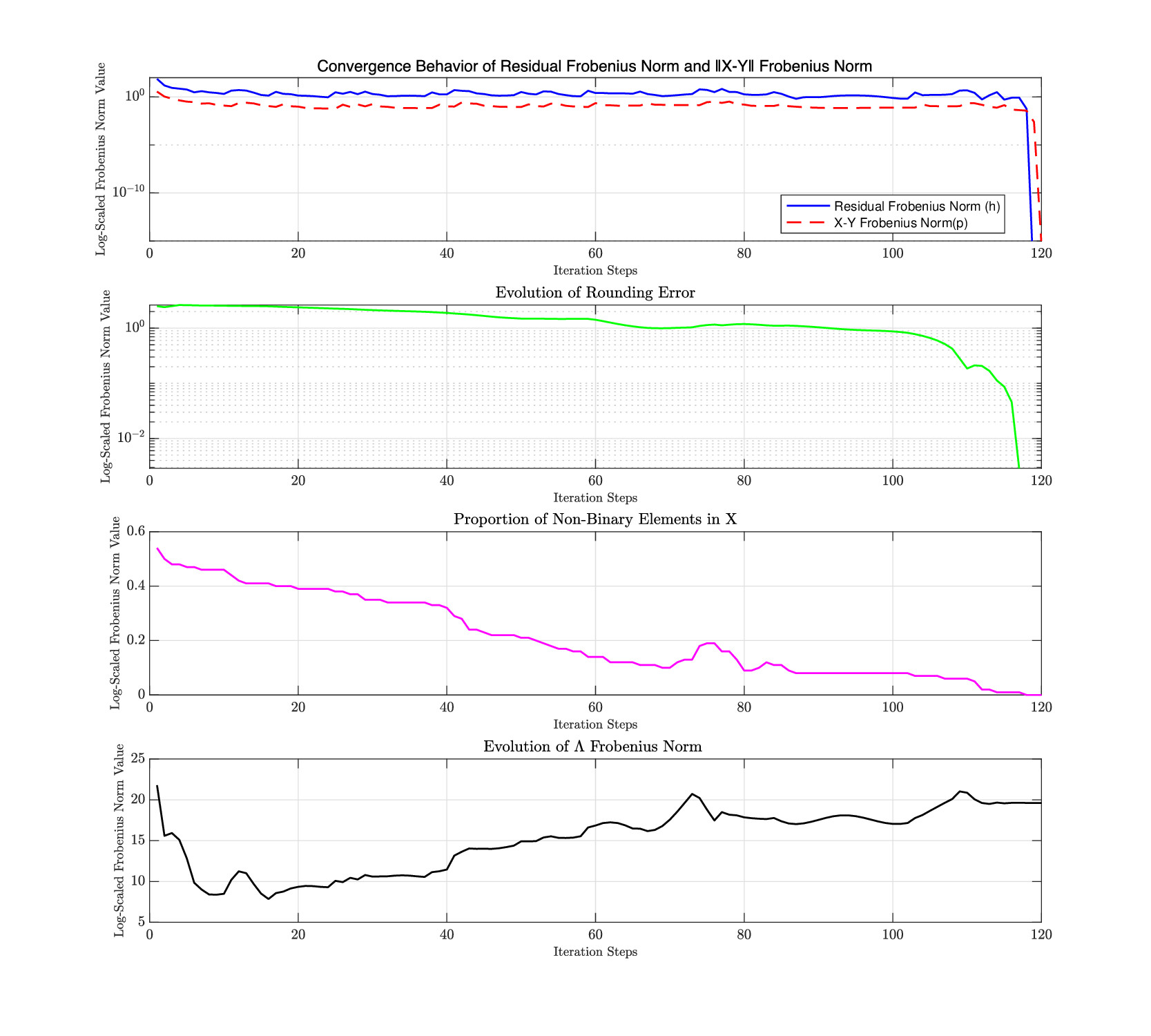}}

\subfloat[\centering $\beta =2000, ~\eta = 1$]{\label{fig:beta2000to01}\includegraphics[width=0.5\textwidth]{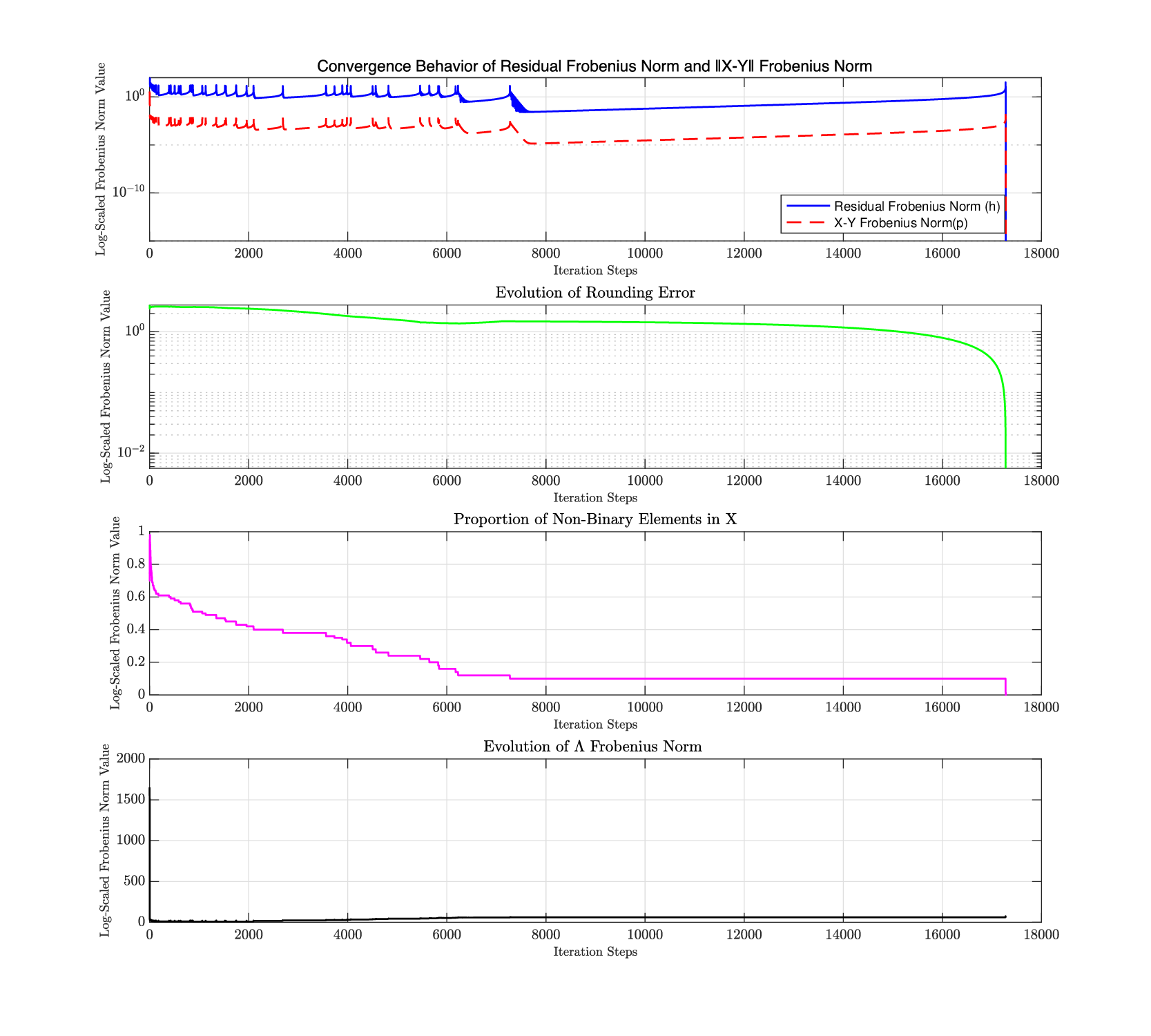}}
\subfloat[\centering $\beta =200000,~ \eta = 1$]{\label{fig:beta200000to01}\includegraphics[width=0.5\textwidth]{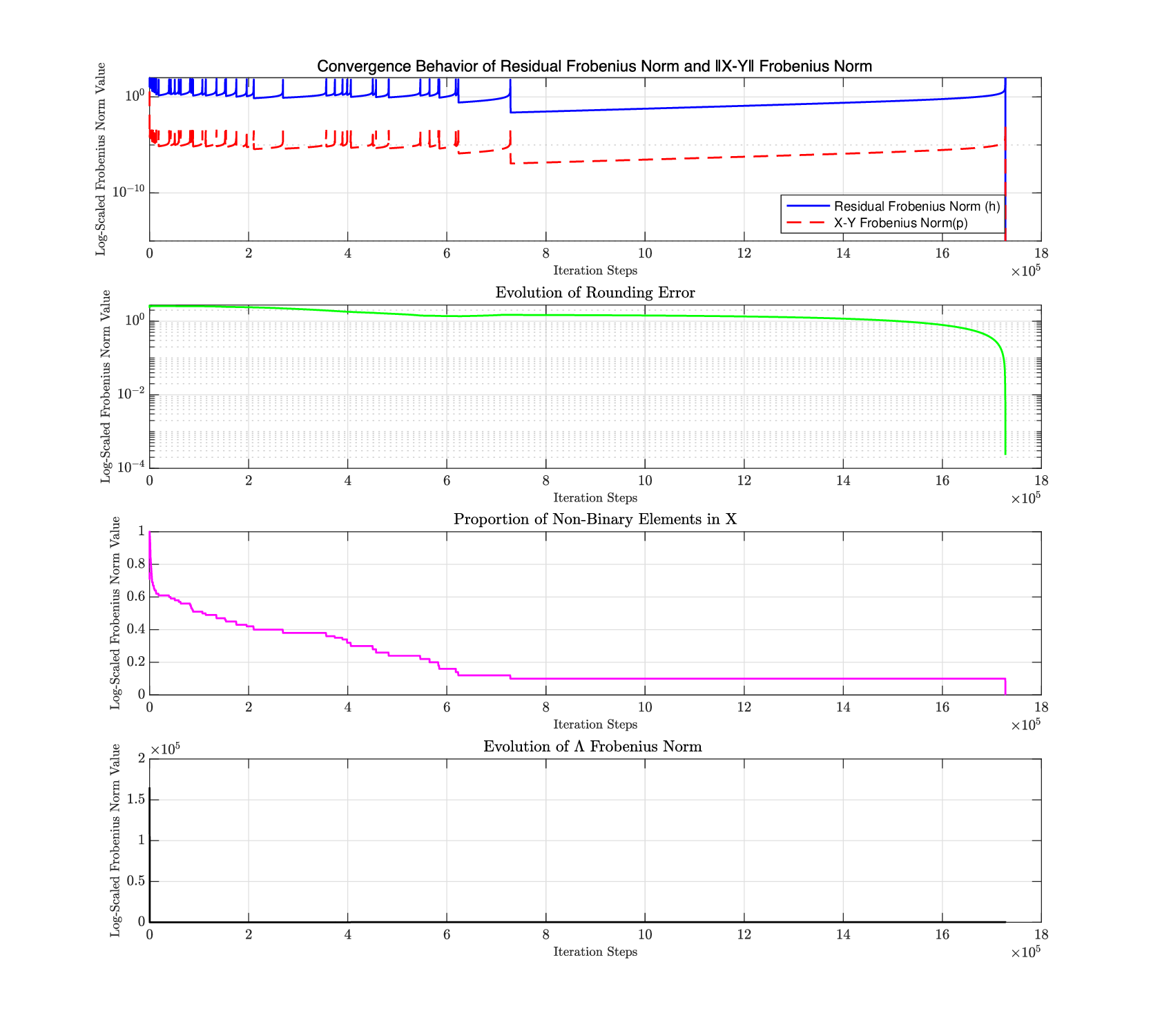}}
\caption{Convergence behavior under different $\beta,~\eta$ values.}
\label{fig:no01-to01}
\end{figure}

Firstly, we test the finite-step property of our algorithm and the boundedness of dual variable sequence $\{\Lambda^k \}$ in \cref{fig:no01-to01}. It demonstrates that \cref{alg:algorithm framework} achieves finite-step termination when $\eta$ is sufficiently large, as evidenced by \cref{fig:no01,fig:to01}. The first panel in each subfigure displays the termination conditions $h^k$ and $p^k$ defined in \cref{eq:kkt_res}. The second panel shows $\|X^k - \text{round}(X^k)\|_F$, where $\text{round}(X^k)$ denotes the rounded (binary) version of $X^k$. The third panel plots the fraction of non-binary elements in $X^k$, calculated as $\frac{1}{nm}\sum_{i=1}^n\sum_{j=1}^m \mathbb{I}(0 < X_{ij}^k < 1)$. The last panel plots the $\|\Lambda\|_F$ during iterate process. These results demonstrate that when $\eta=0.01$, the sequence $X^k$ converges to a non-integer solution $X^*$, whereas when $\eta=1$, it converges to a binary solution $X^*$ in finite steps. Increasing the value of $\beta$ reduces the convergence speed of \cref{alg:algorithm framework}, but the algorithm still terminates in a finite number of steps, as shown in \cref{fig:beta2000to01,fig:beta200000to01}. For all cases, $\|\Lambda\|_F$ is bounded,  which sustains the validity of \cref{assu 1}.

\cref{fig:synthesis re n400 f} presents the training results on the synthetic dataset. With access to the theoretical optimum \(f^*\), we evaluate convergence by computing \(\|f^k - f^*\|\) at each training step \(k\), where \(f^k\) is the objective value at iteration \(k\). Our \textbf{Matrix} method consistently achieves the lowest \(\|f^k - f^*\|\) values across SGD, WSGD, and WMSGD optimizers, with the \textbf{Vector} method performing between \textbf{Matrix} and \textbf{Random} across all iterations. Under Adam, \textbf{Matrix} and \textbf{Vector} achieve similar final error values, but \textbf{Matrix} demonstrates more stable convergence behavior throughout training compared to \textbf{Vector}. These results demonstrate that batch selection improves optimization, with our matrix-based method consistently outperforming vector-based batch selection. Comprehensive quantitative comparisons are summarized in \cref{sy overall}. We evaluate the method on three datasets with $n=80,400,1200$ respectively. The base-10 logarithm of the errors in $x$ and $f$ is computed every 100 epochs. Our \textbf{Matrix} method achieve lowest loss for almost all cases.

\begin{figure}[h!]
\centering
\subfloat[\centering $(f-\overset{*}{f})^2$ using SGD]{\label{fig:sgd}\includegraphics[width=0.40\textwidth]{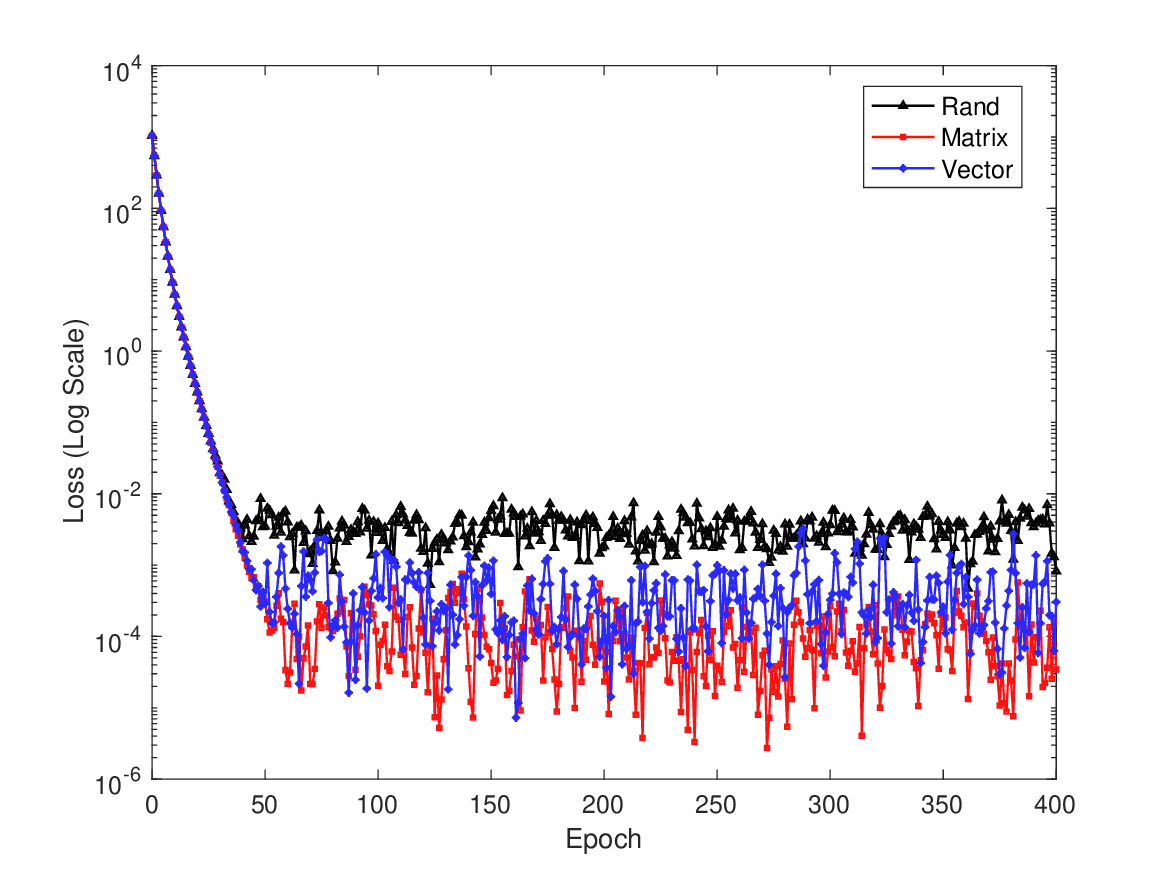}}
\subfloat[\centering $(f-\overset{*}{f})^2$ using WSGD]{\label{fig:wsgd}\includegraphics[width=0.40\textwidth]{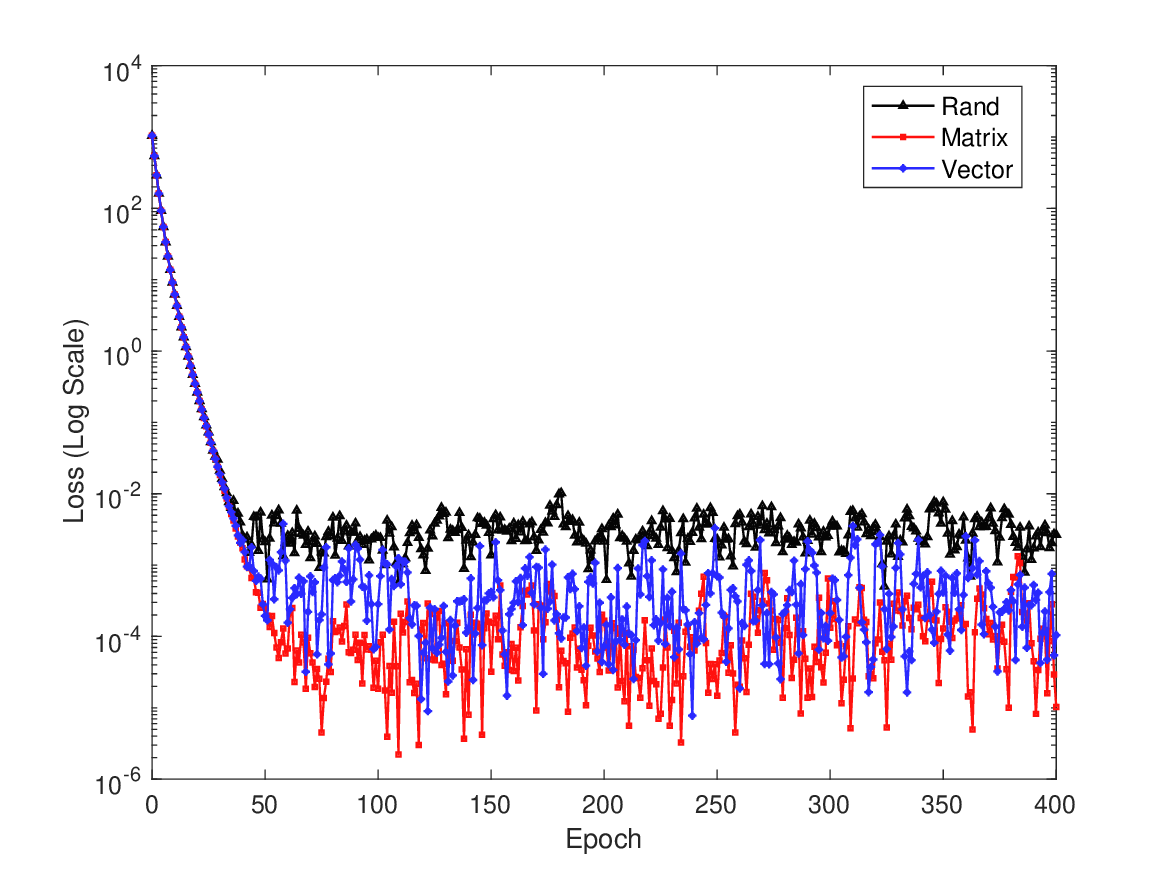}}

\subfloat[\centering $(f-\overset{*}{f})^2$ using WMSGD]{\label{fig:wmsgd}\includegraphics[width=0.40\textwidth]{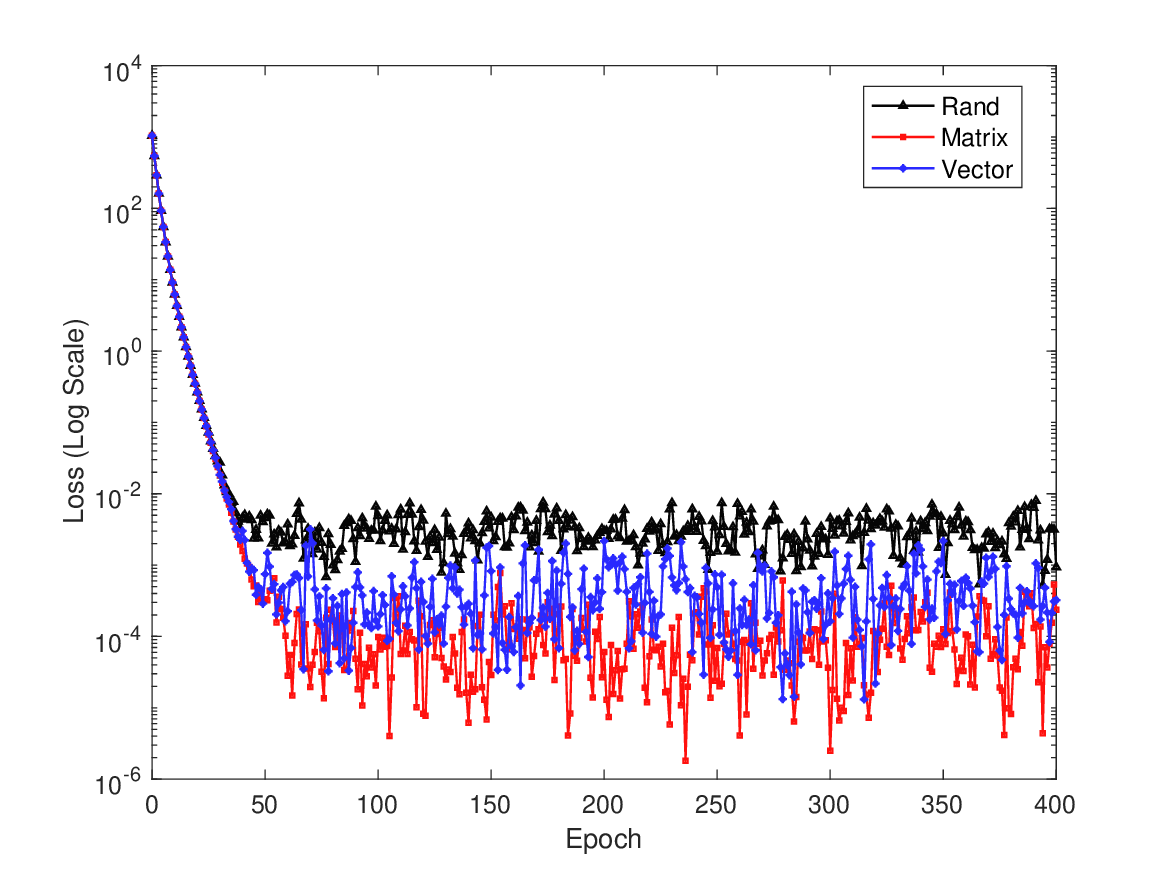}}
\subfloat[\centering $(f-\overset{*}{f})^2$ using ADAM]{\label{fig:adam}\includegraphics[width=0.40\textwidth]{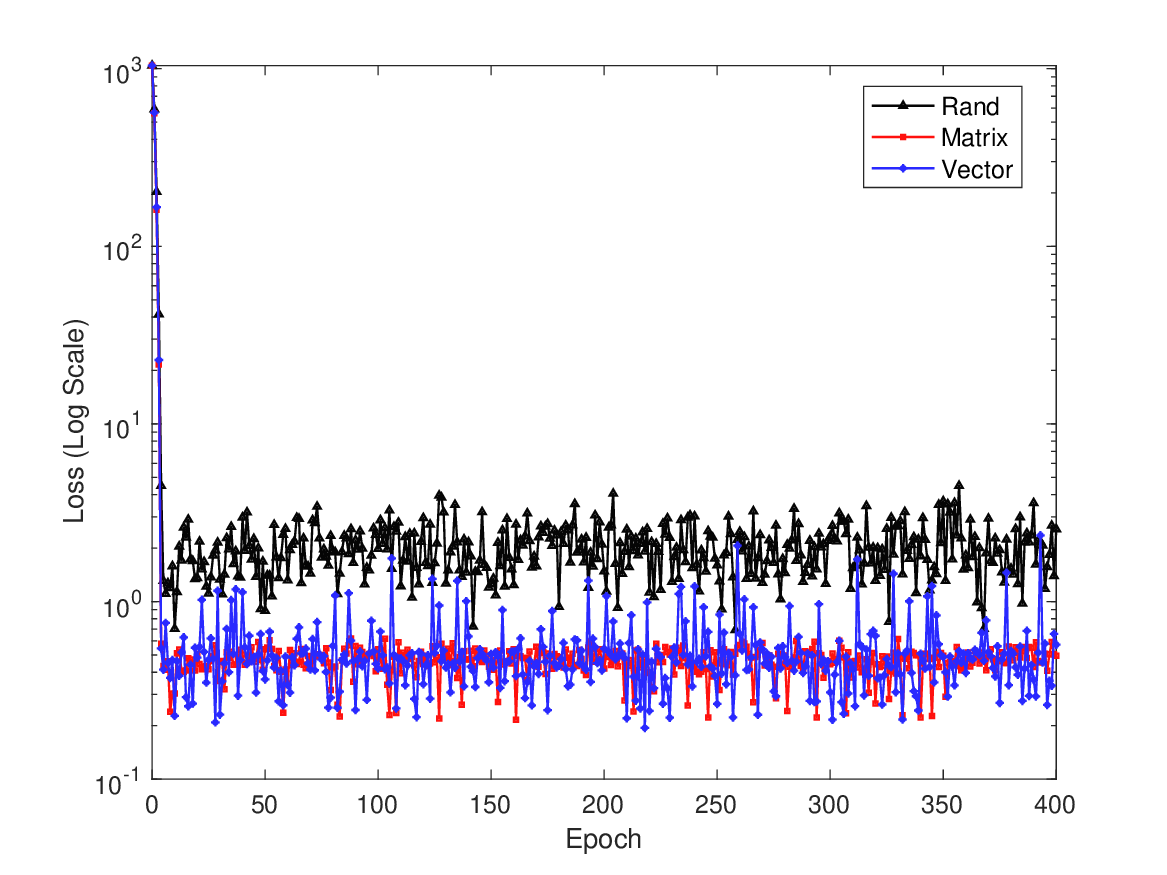}}

\subfloat[\centering $\|x-\overset{*}{x}\|^2$ using SGD]{\label{fig:sgd2}\includegraphics[width=0.40\textwidth]{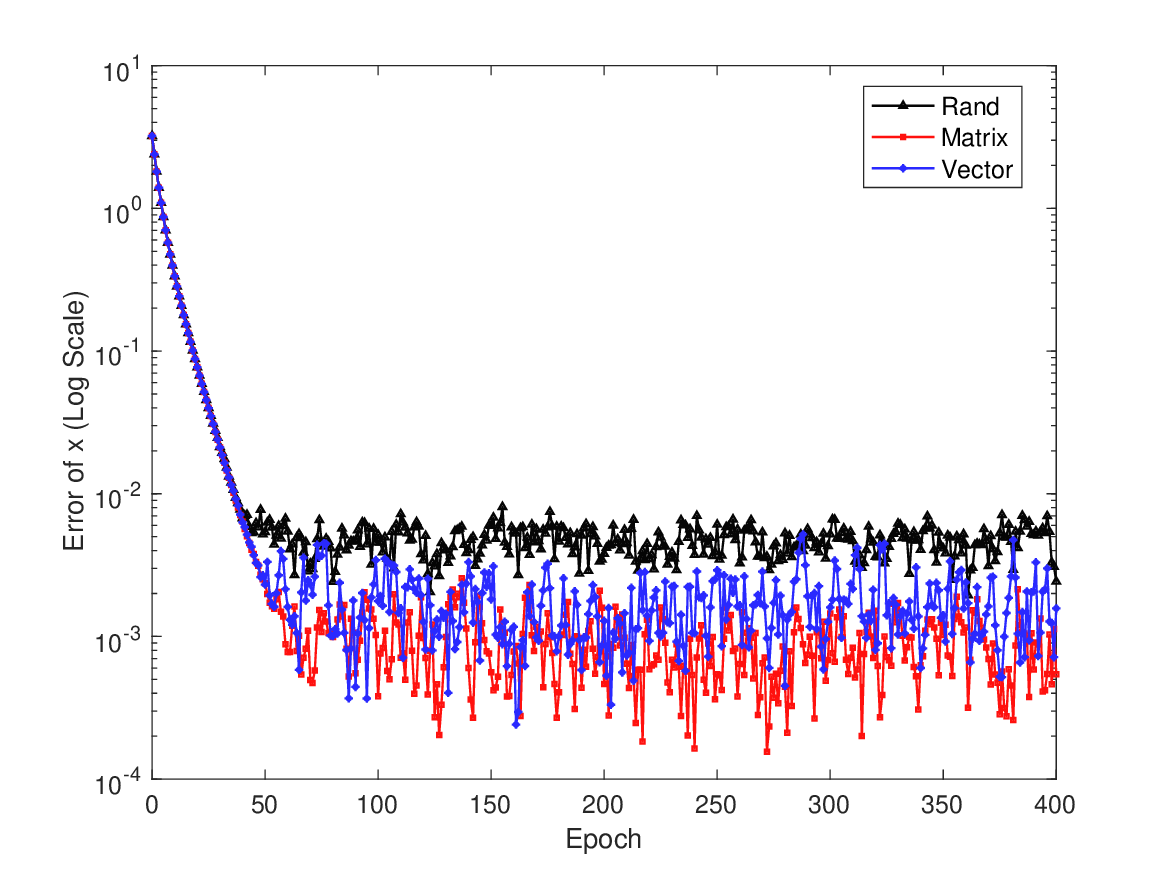}}
\subfloat[\centering $\|x-\overset{*}{x}\|^2$ using WSGD]{\label{fig:wsgd2}\includegraphics[width=0.40\textwidth]{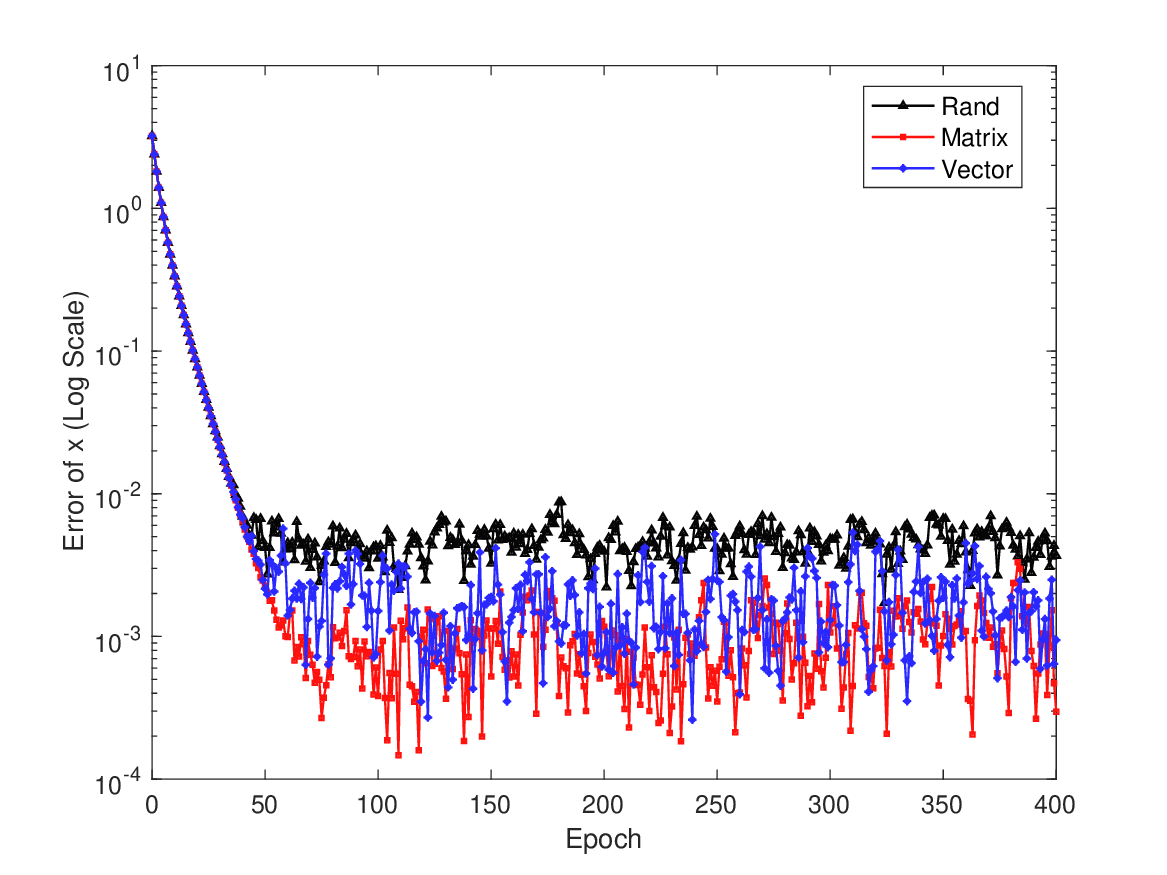}}

\subfloat[\centering $\|x-\overset{*}{x}\|^2$ using WMSGD]{\label{fig:wmsgd2}\includegraphics[width=0.40\textwidth]{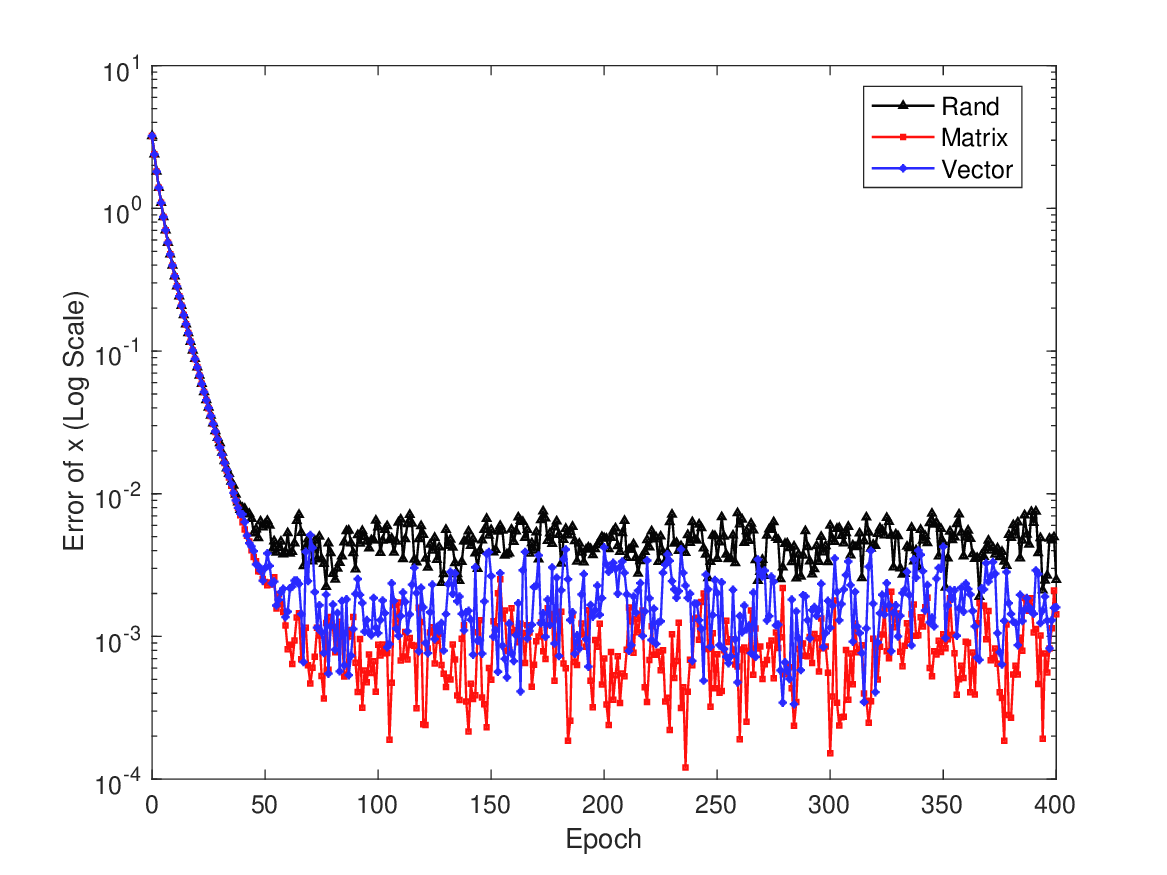}}
\subfloat[\centering $\|x-\overset{*}{x}\|^2$ using ADAM]{\label{fig:adam2}\includegraphics[width=0.40\textwidth]{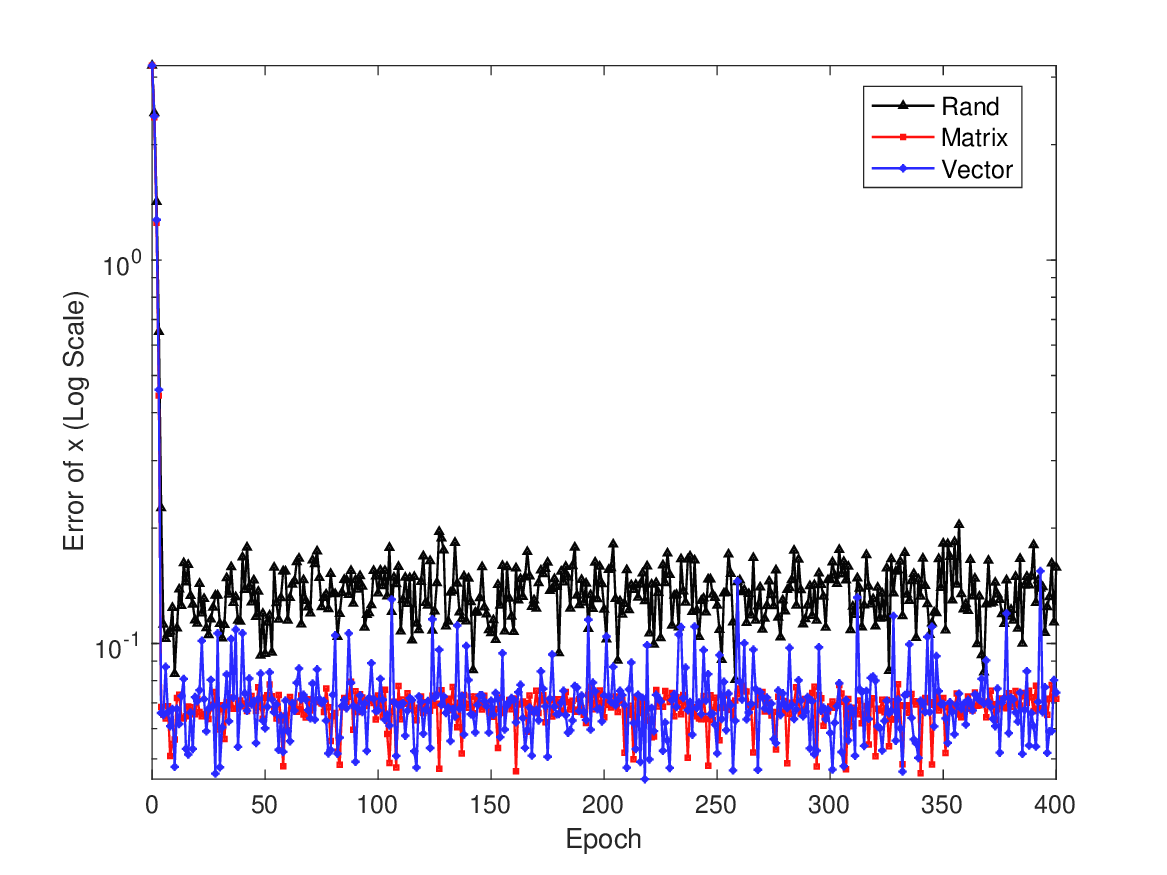}}
\caption{Convergence speed on synthetic dataset using different random optimization methods, n=400 case }
\label{fig:synthesis re n400 f}
\end{figure}
\subsection{Results on MNIST and CIFAR-10 Dataset}
For MNIST and CIFAR-10, we use the standard train/test splits provided with these datasets. We employ ResNet-18 as the backbone architecture for all experiments. All models are trained for 100 epochs, with mini-batch sizes of 256 for MNIST and 250 for CIFAR-10.

For MNIST, we vectorize images into 1D arrays and compute the kernel matrix directly from raw pixel values. For CIFAR-10, we first extract Histogram of Oriented Gradients (HOG) features \cite{albiol2008face} and then compute the kernel matrix from these representations. This approach is motivated by the fact that raw pixel values in CIFAR-10's natural color images are less effective for capturing semantic structures, whereas HOG features provide more discriminative edge and shape information.

\cref{fig:MNIST re} and \cref{fig:CIFAR-10 re} present experimental results on MNIST and CIFAR-10 datasets, respectively, comparing training loss and test error throughout optimization. Our \textbf{Matrix} approach consistently achieves the lowest training loss across all four optimizers, demonstrating accelerated convergence and strong compatibility with diverse optimization methods. In contrast, the \textbf{Vector} method exhibits inferior convergence with Adam, as evidenced by significantly higher loss and test error values.

\begin{figure}[h!]
\centering
\subfloat[\centering SGD Train Loss]{\label{fig:SGD Train Loss}\includegraphics[width=0.45\textwidth,,height=3.2cm]{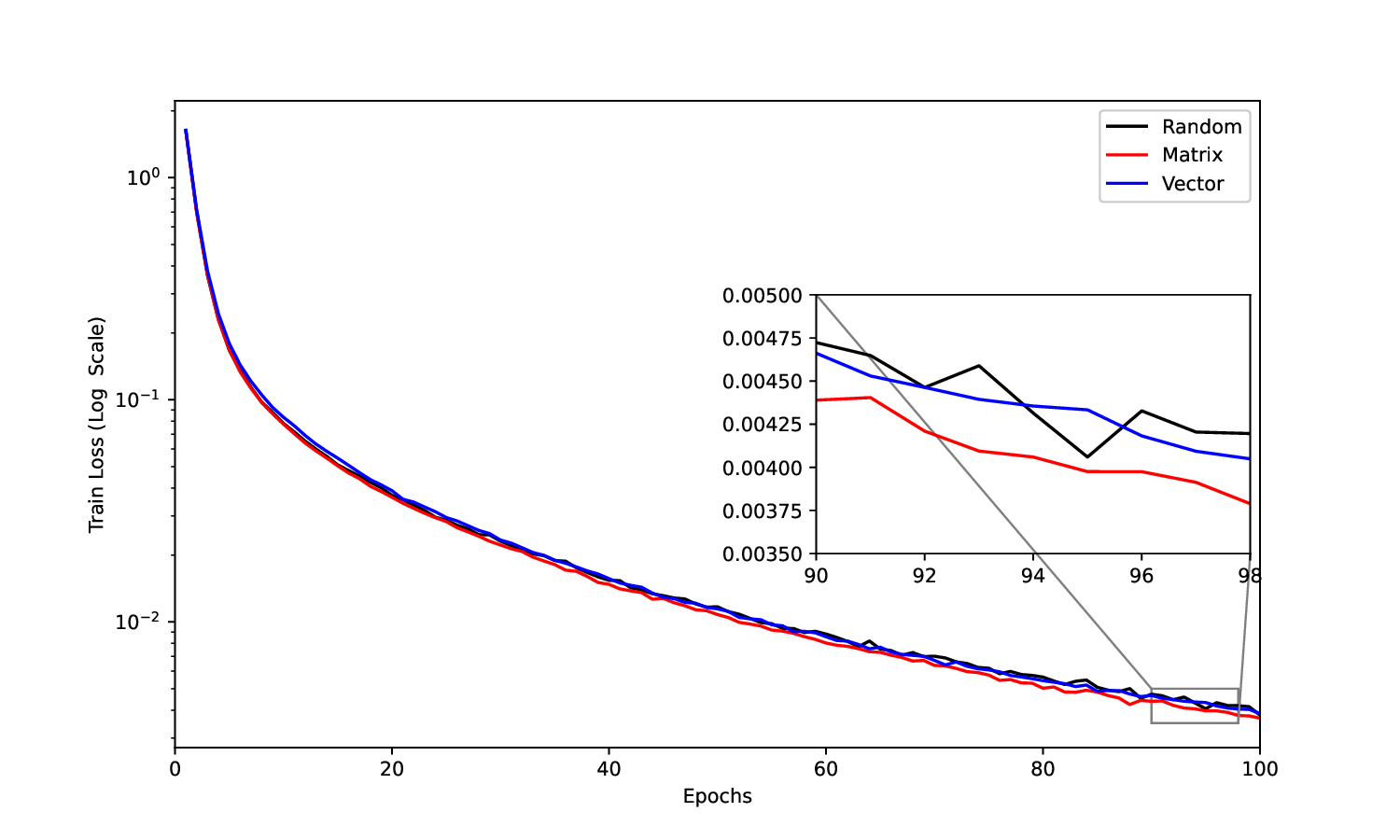}}
\subfloat[\centering SGD Test Error]{\label{fig:SGD Test Error}\includegraphics[width=0.45\textwidth,,height=3.2cm]{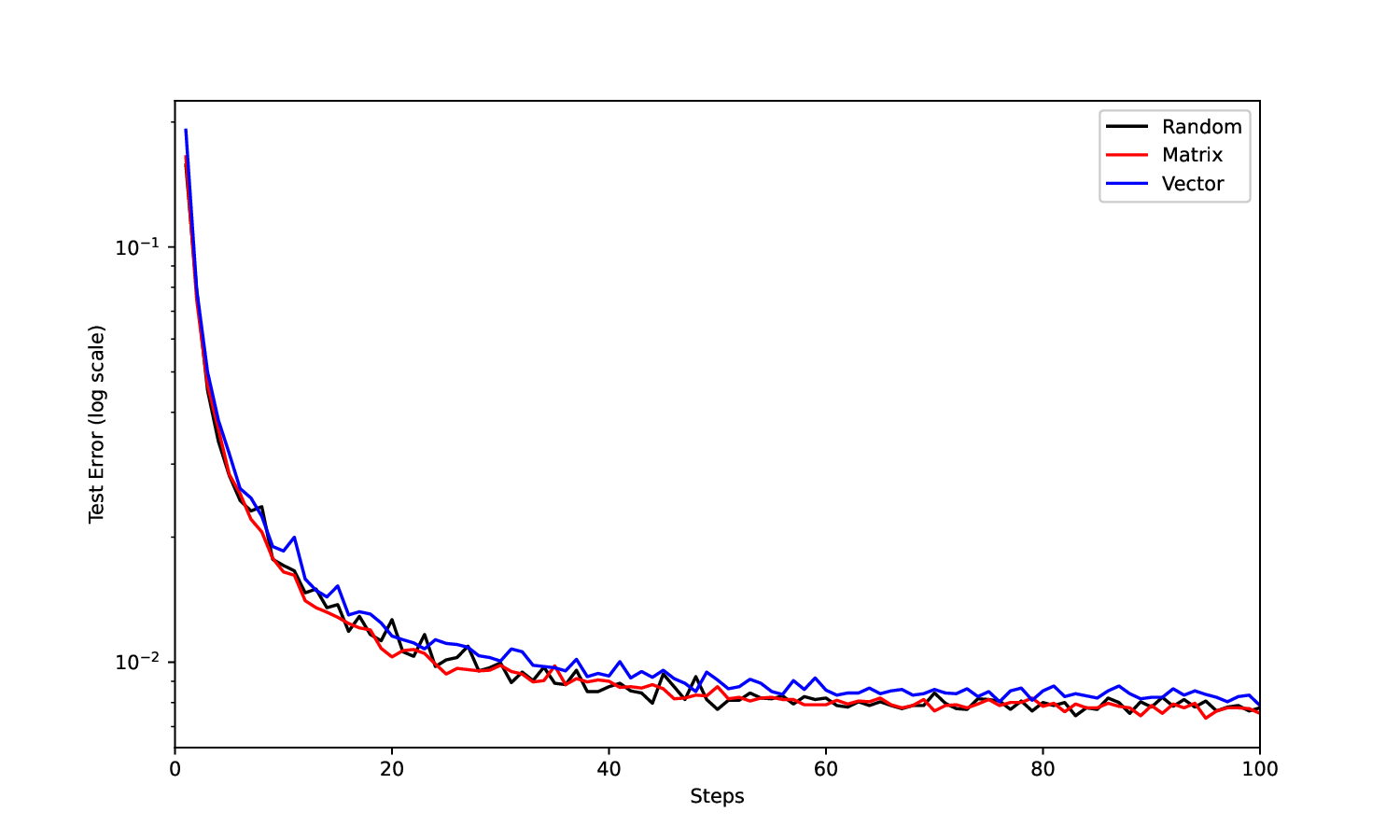}}

\subfloat[\centering SGDW Train Loss]{\label{fig:SGDW Train Loss}\includegraphics[width=0.45\textwidth,height=3.2cm]{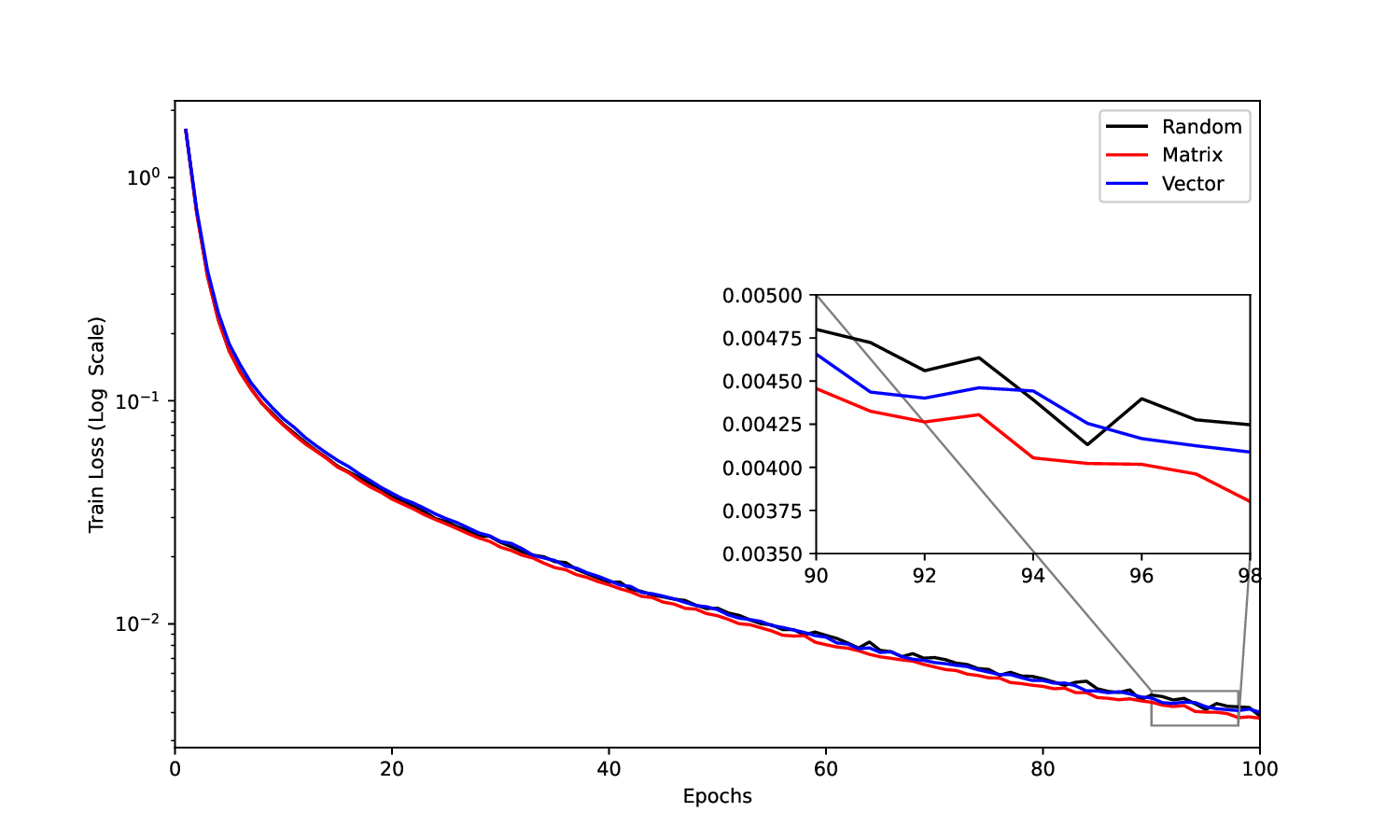}}
\subfloat[\centering SGDW Test Error]{\label{fig:SGDW Test Error}\includegraphics[width=0.45\textwidth,,height=3.2cm]{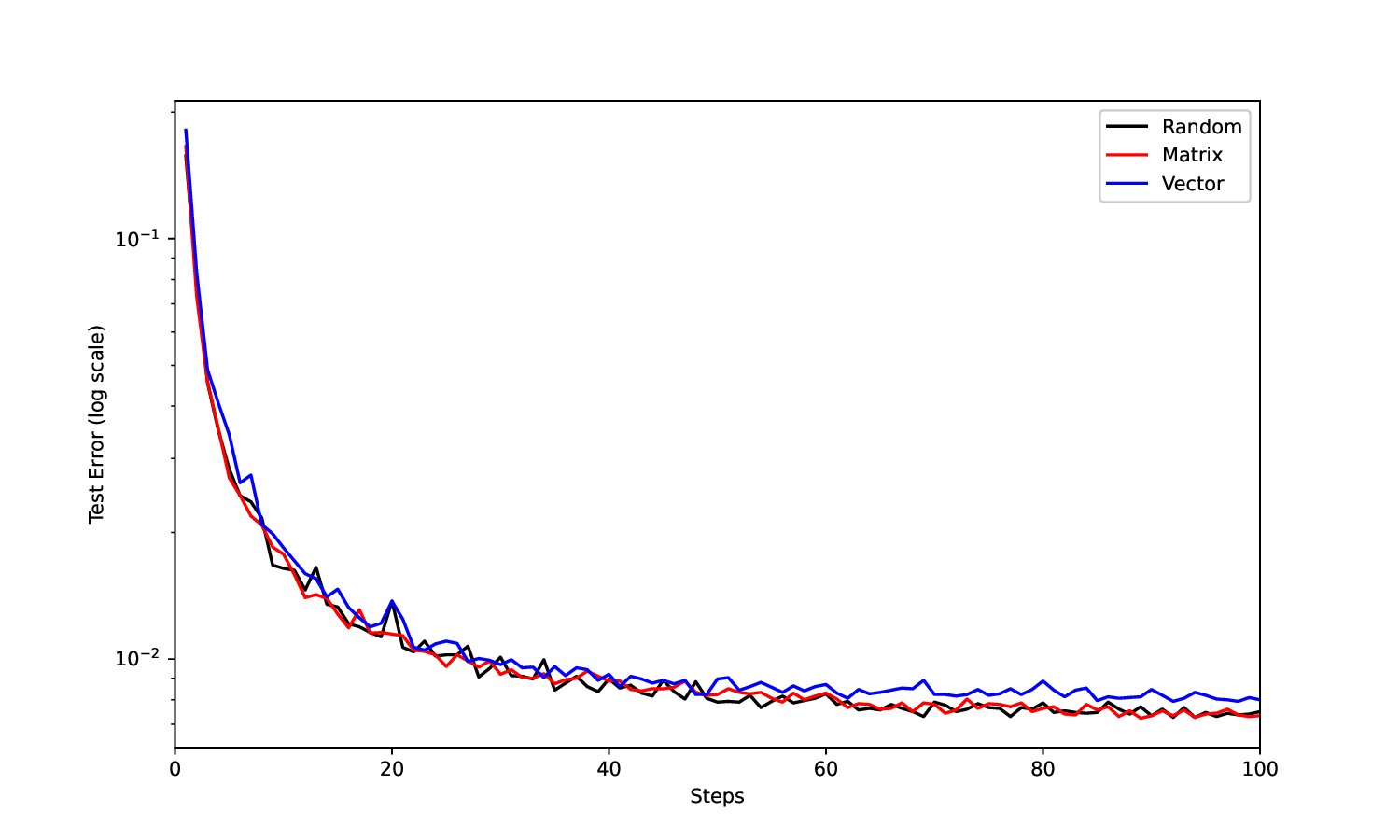}}

\subfloat[\centering SGDWM Train Loss]{\label{fig:SGDWM Train Loss}\includegraphics[width=0.45\textwidth,,height=3.2cm]{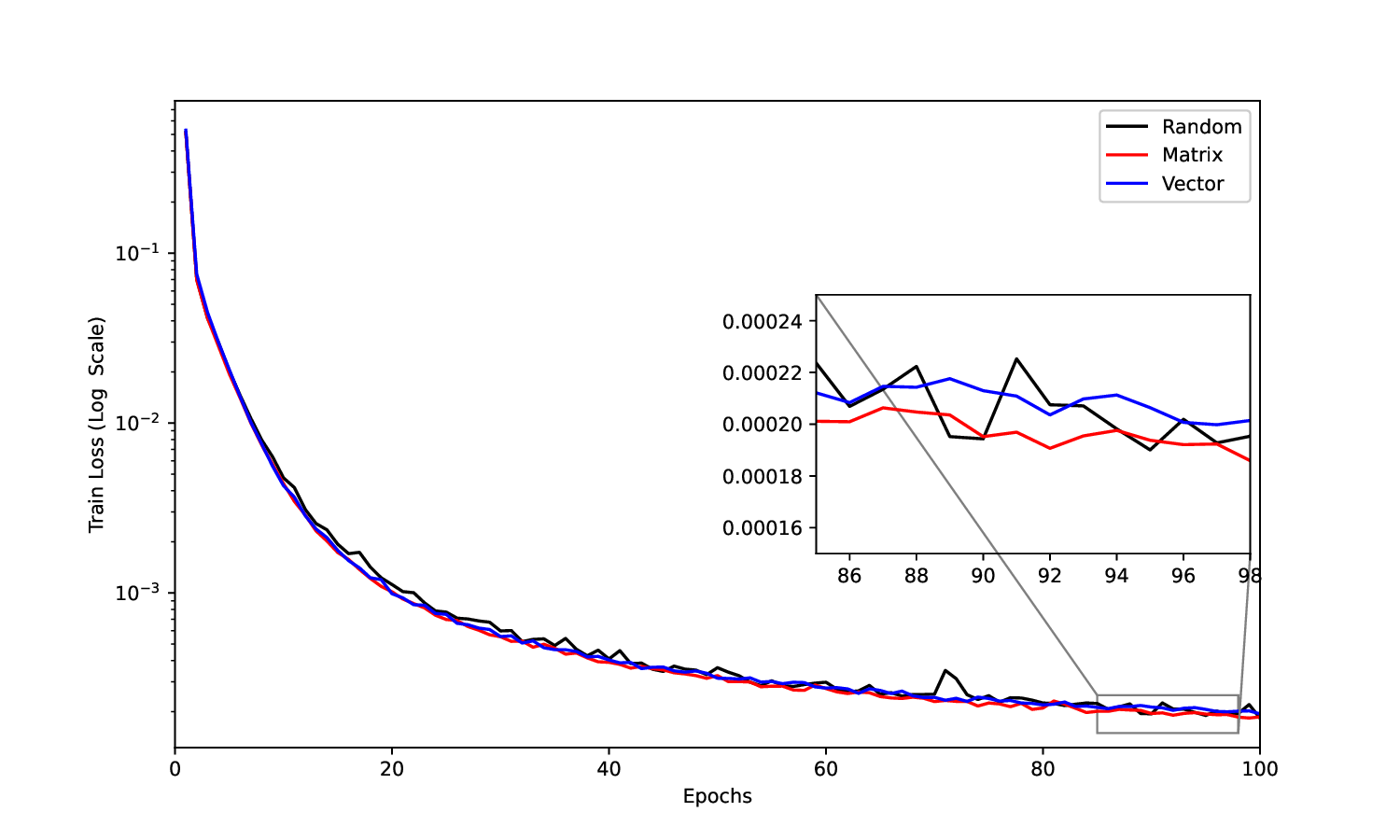}}
\subfloat[\centering SGDWM Test Error]{\label{fig:SGDWM Test Error}\includegraphics[width=0.45\textwidth,,height=3.2cm]{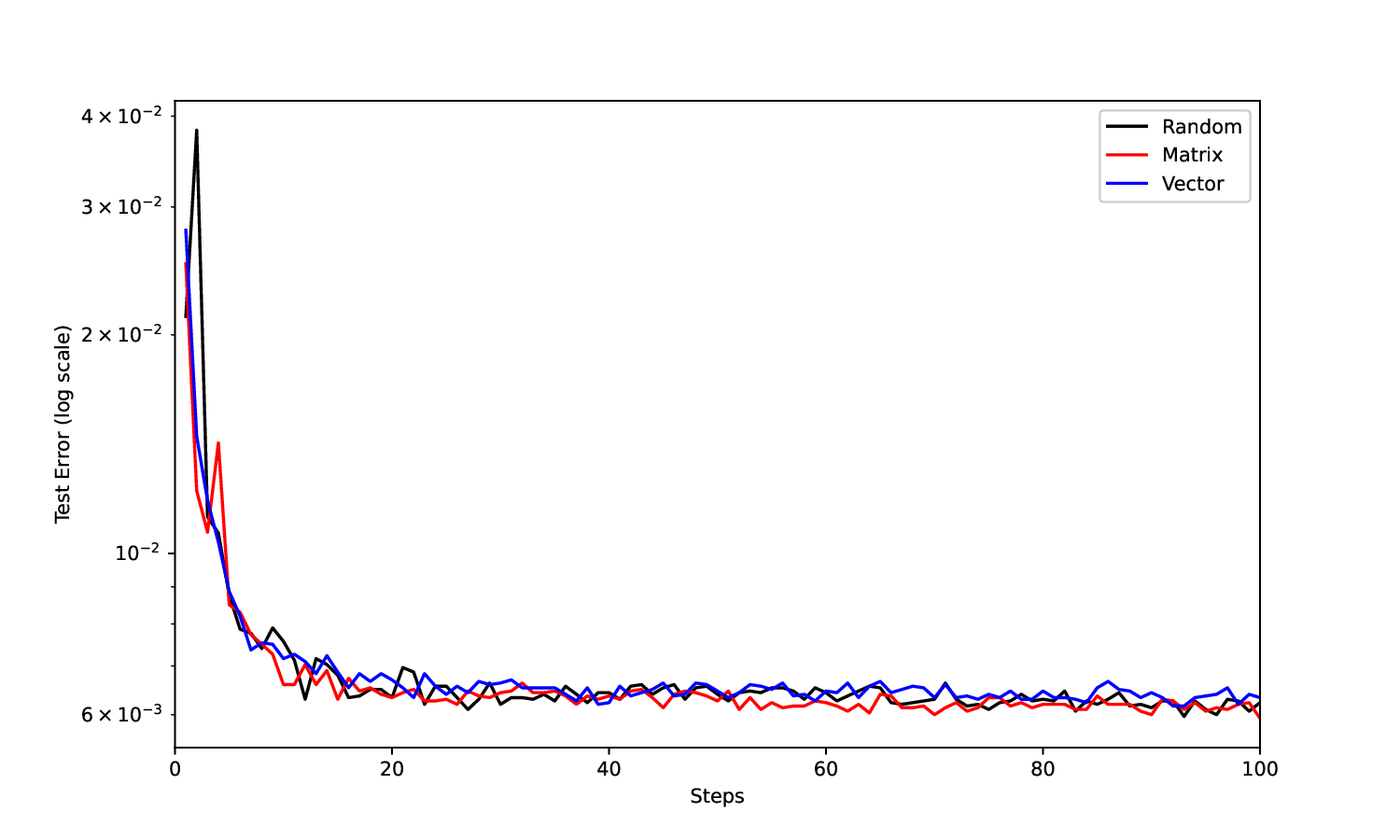}}

\subfloat[\centering ADAM Train Loss]{\label{fig:ADAM Train Loss}\includegraphics[width=0.45\textwidth,,height=3.2cm]{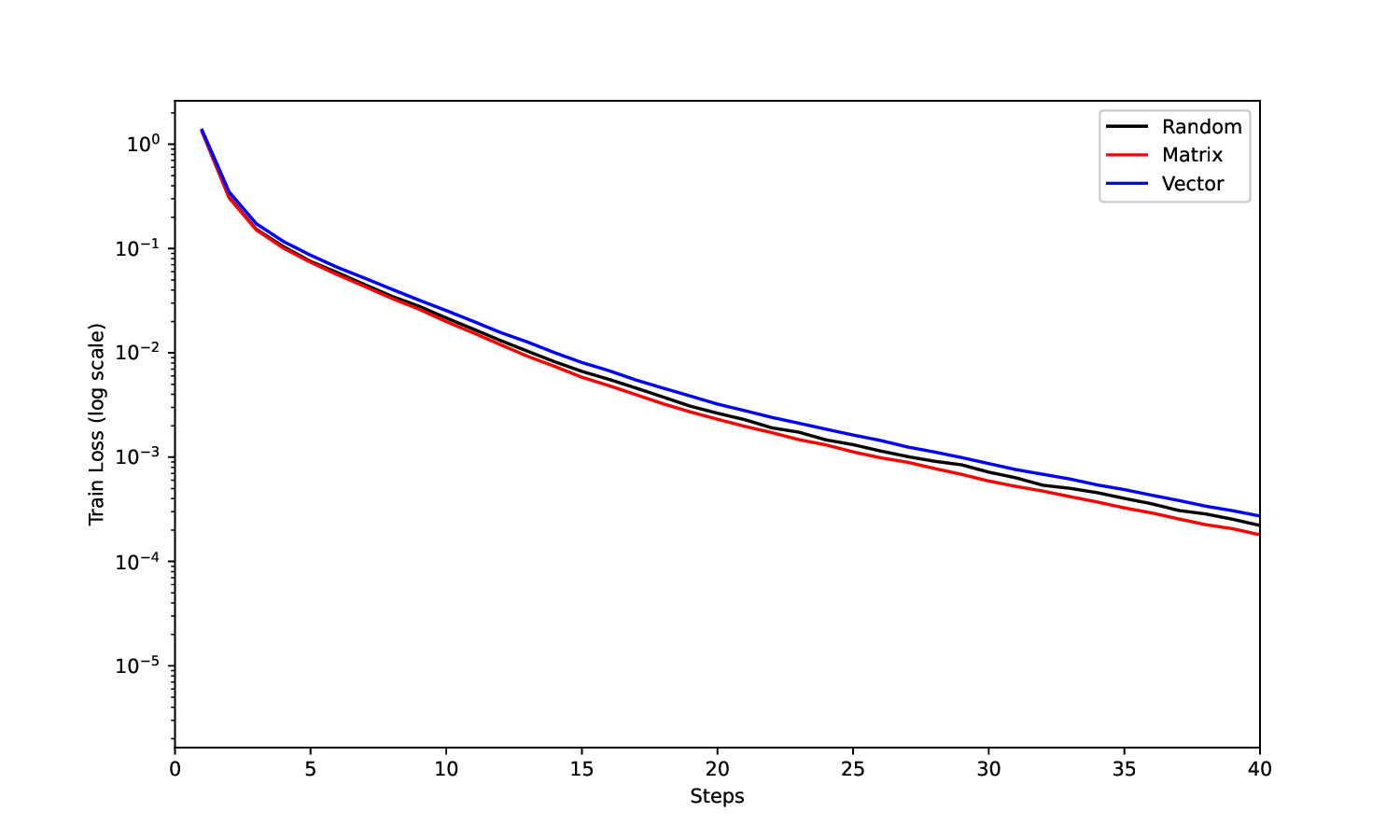}}
\subfloat[\centering ADAM Test Error]{\label{fig:ADAM Test Error}\includegraphics[width=0.45\textwidth,,height=3.2cm]{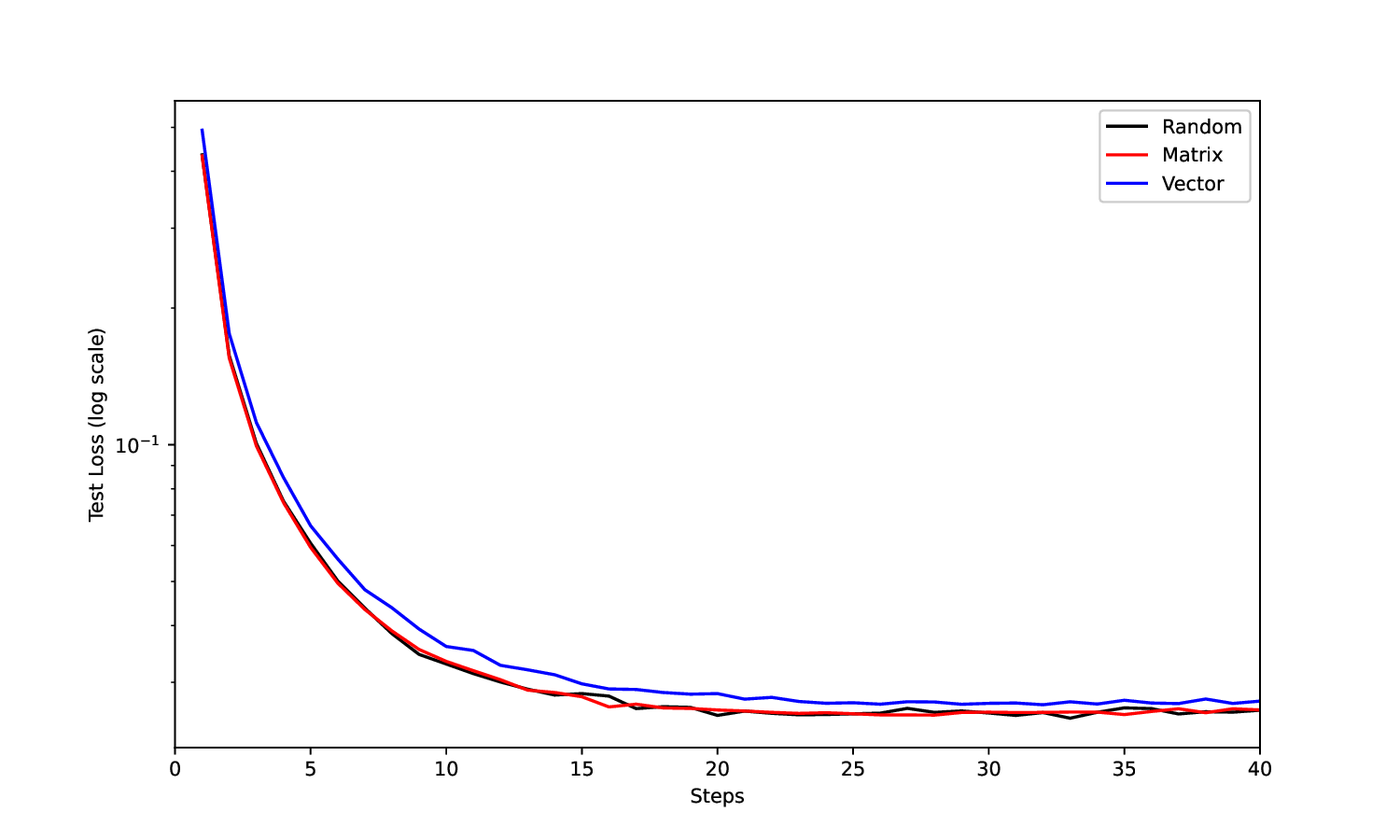}}
\caption{Performance comparison on MNIST dataset using different random optimization methods }
\label{fig:MNIST re}
\end{figure}

\begin{figure}[h!]
\centering
\subfloat[\centering SGD Train Loss]{\label{fig:cifar10 SGD Train Loss}\includegraphics[width=0.45\textwidth]{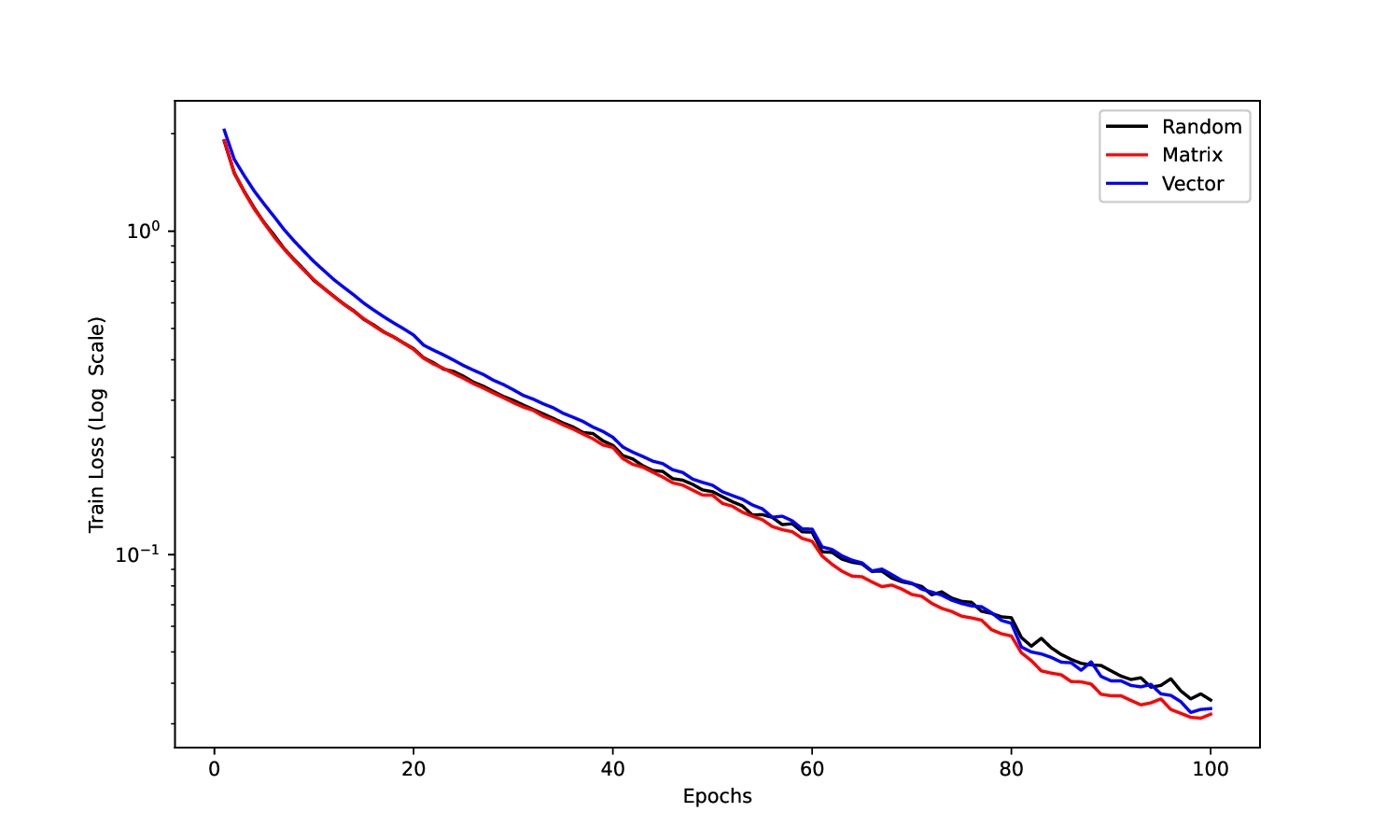}}
\subfloat[\centering SGD Test Error]{\label{fig:cifar10 SGD Test Error}\includegraphics[width=0.45\textwidth]{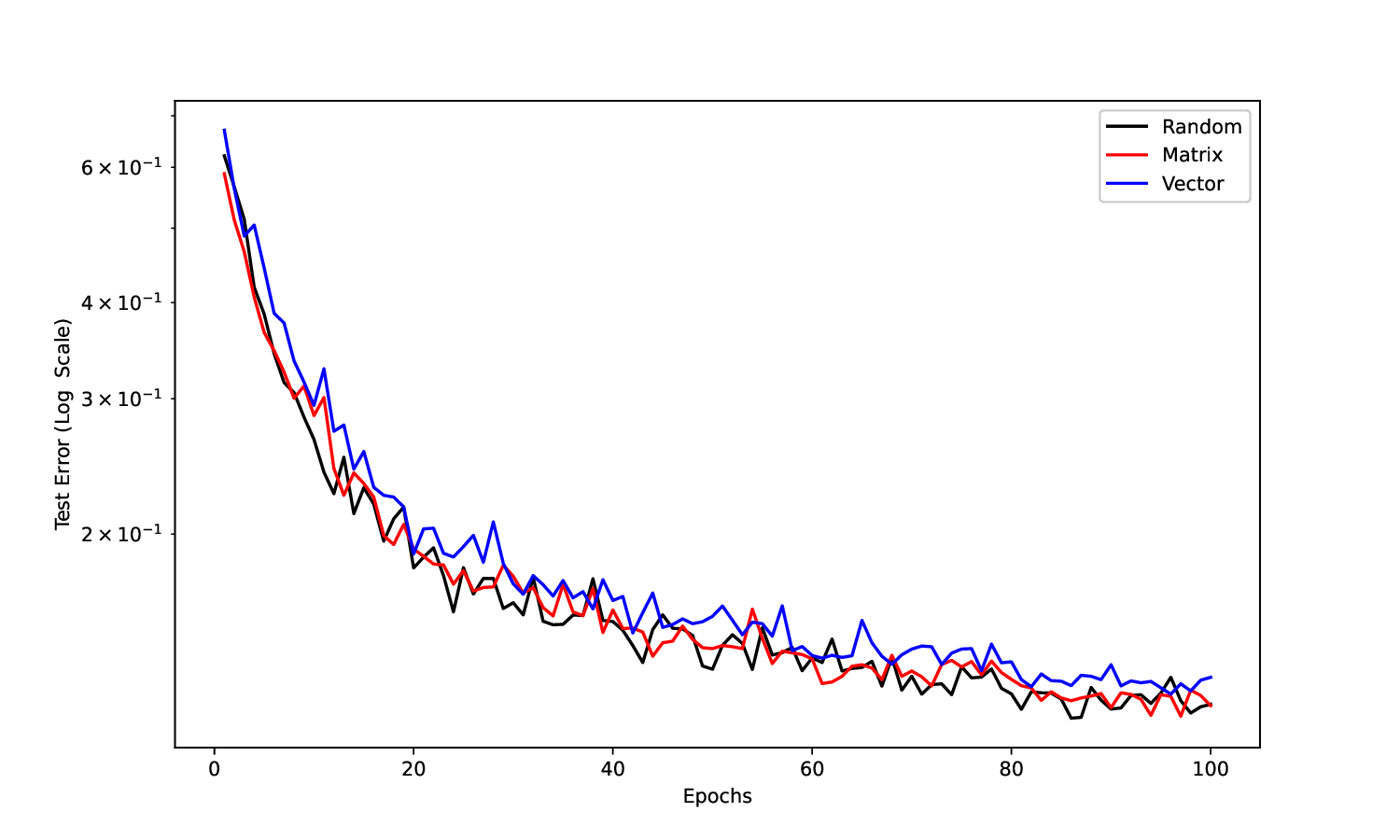}}

\subfloat[\centering SGDW Train Loss]{\label{fig:cifar10 SGDW Train Loss}\includegraphics[width=0.45\textwidth]{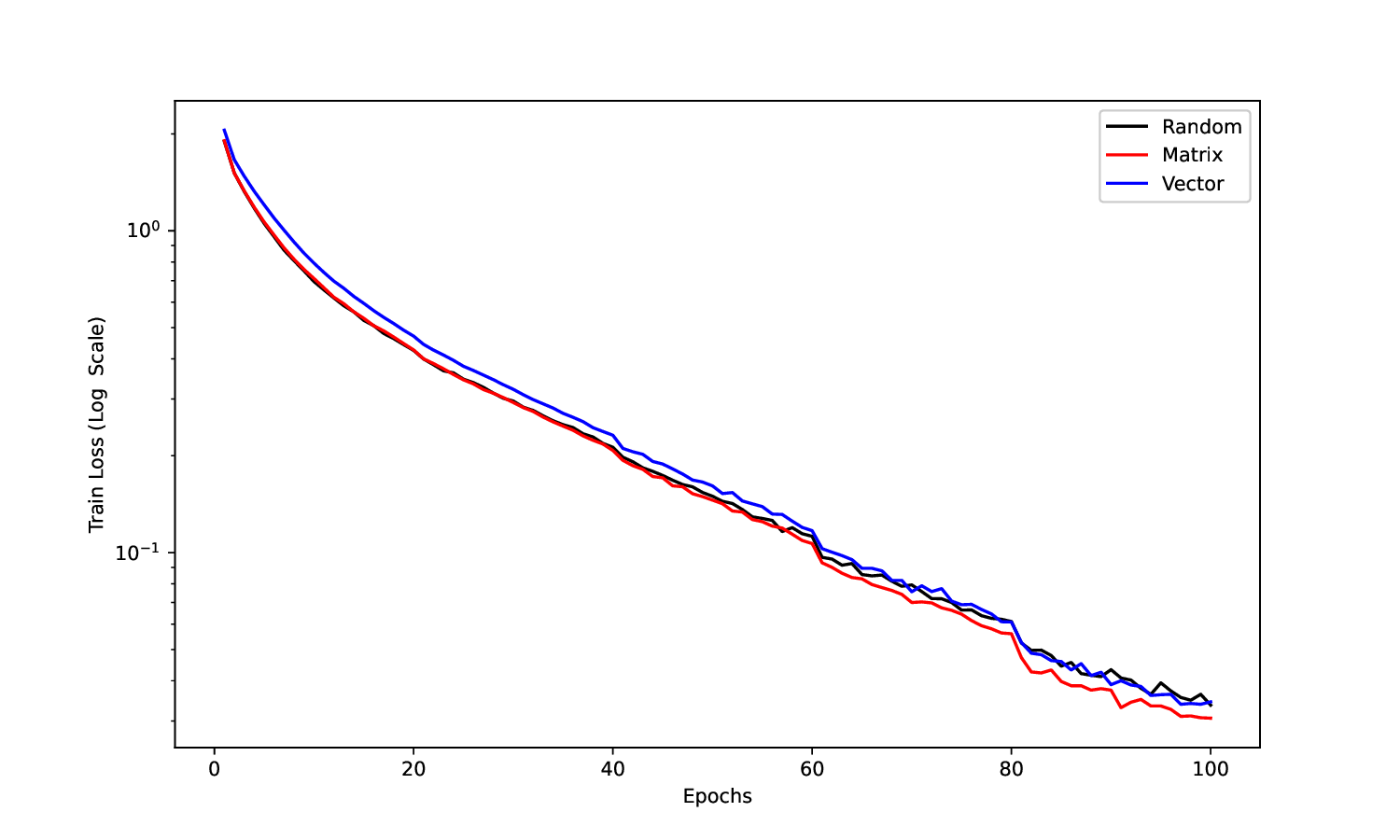}}
\subfloat[\centering SGDW Test Error]{\label{fig:cifar10 SGDW Test Error}\includegraphics[width=0.45\textwidth]{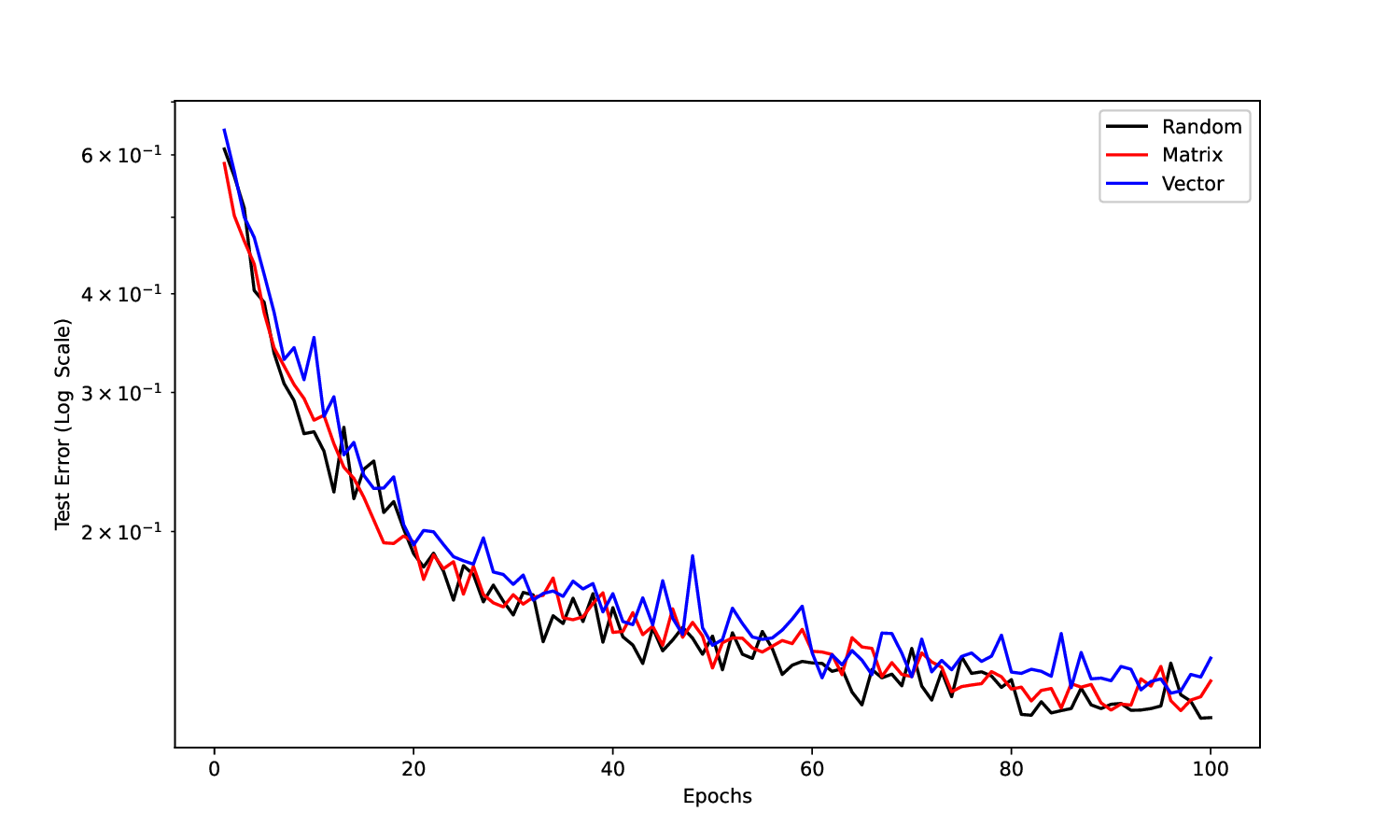}}

\subfloat[\centering SGDWM Train Loss]{\label{fig:cifar10 SGDWM Train Loss}\includegraphics[width=0.45\textwidth]{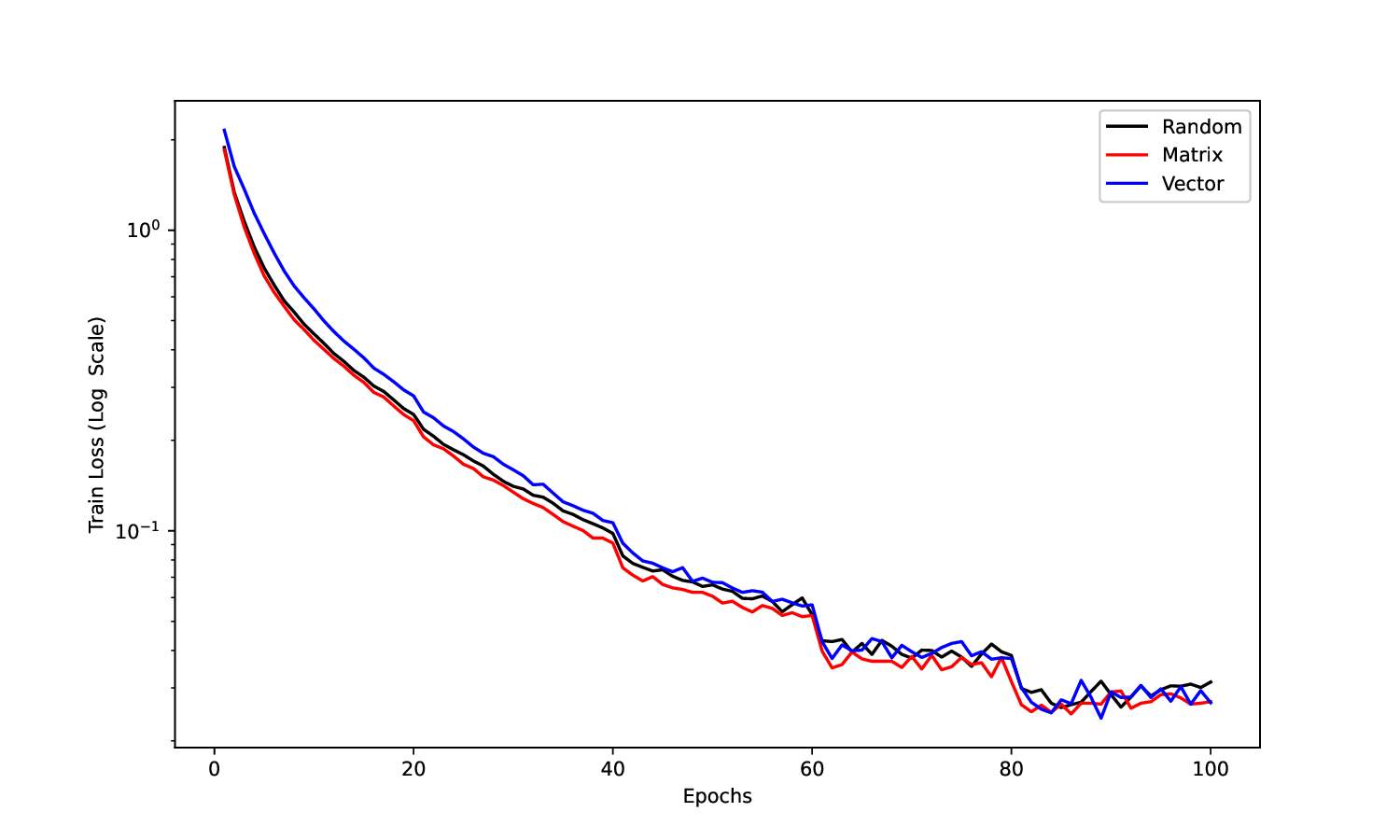}}
\subfloat[\centering SGDWM Test Error]{\label{fig:cifar10 SGDWM Test Error}\includegraphics[width=0.45\textwidth]{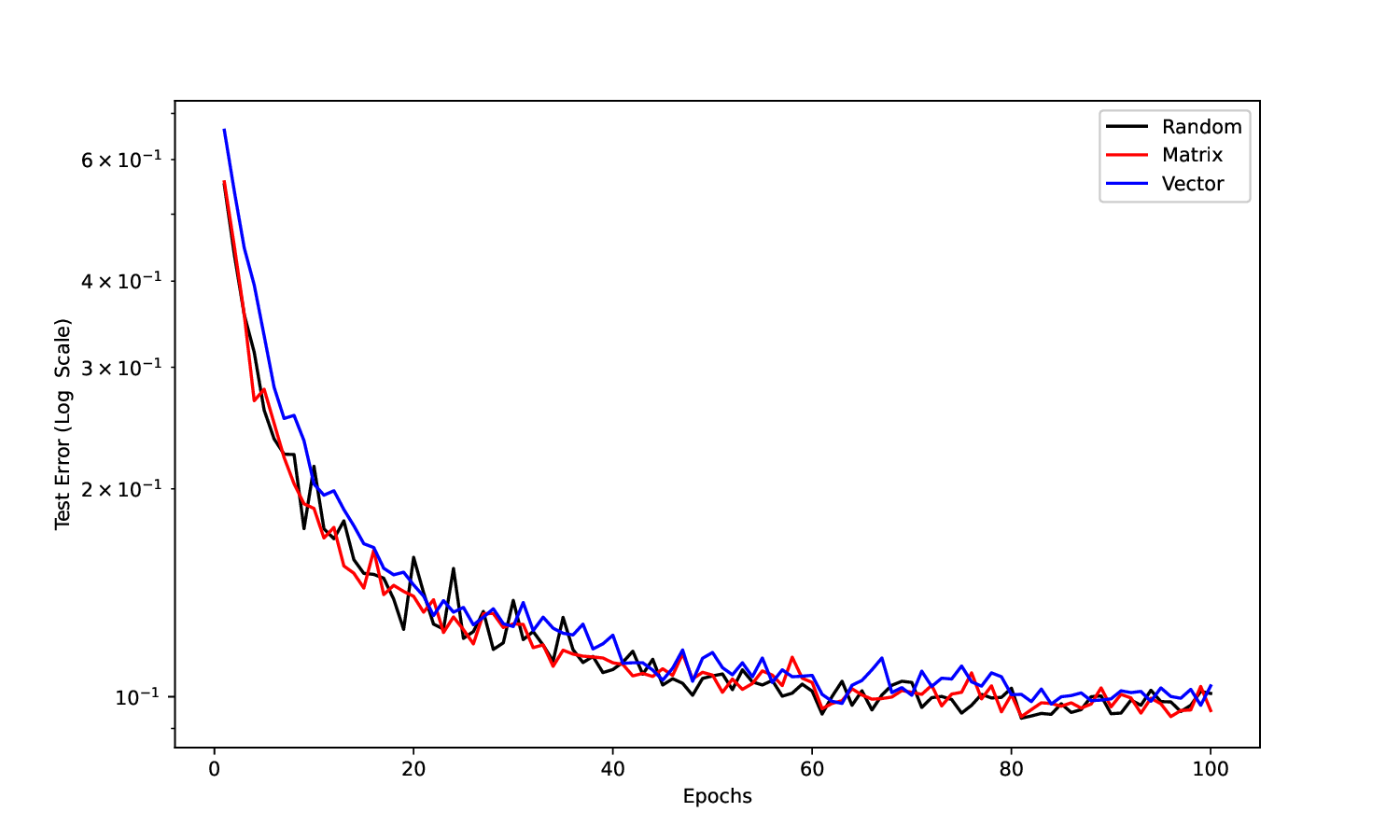}}

\subfloat[\centering ADAM Train Loss]{\label{fig:cifar10 ADAM Train Loss}\includegraphics[width=0.45\textwidth]{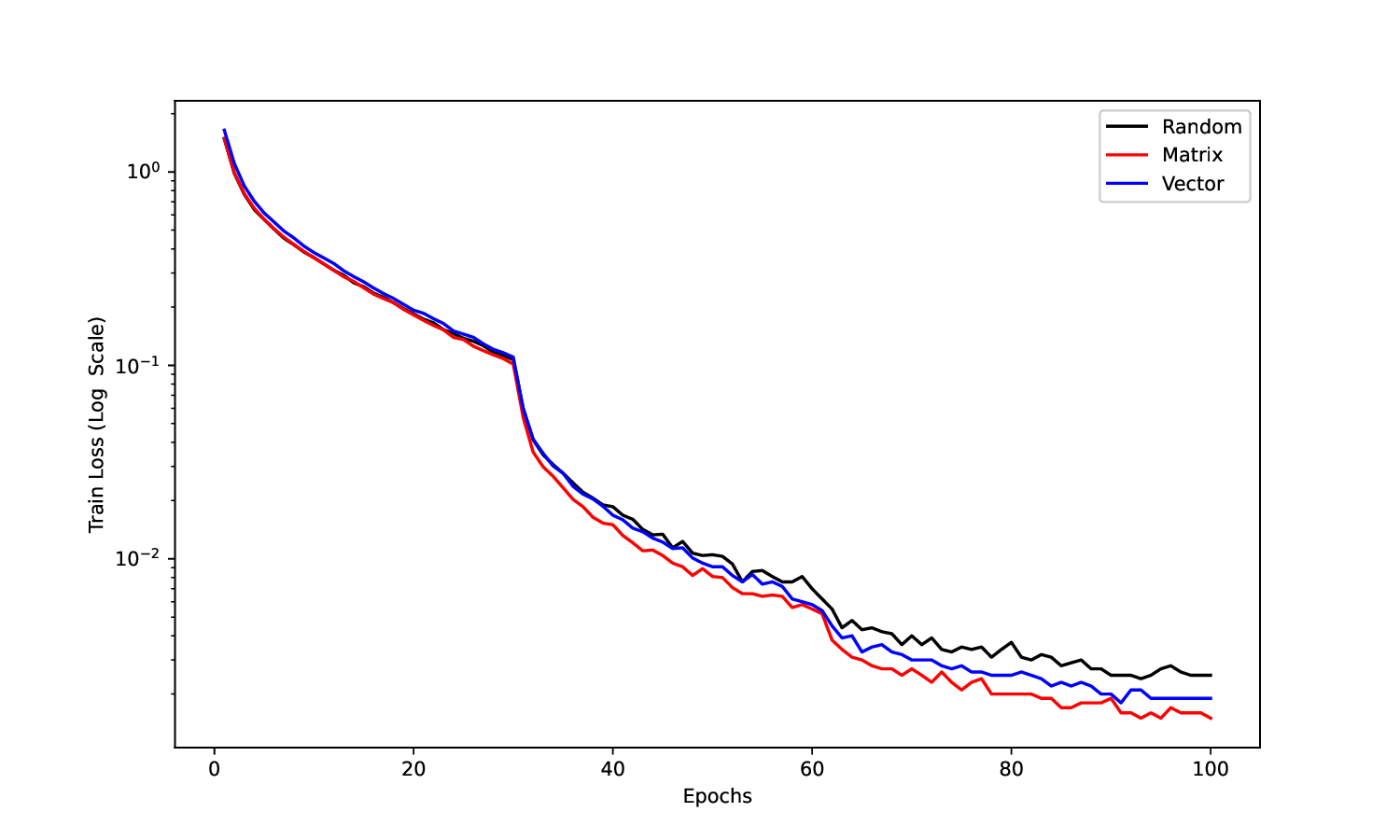}}
\subfloat[\centering ADAM Test Error]{\label{fig:cifar10 ADAM Test Error}\includegraphics[width=0.45\textwidth]{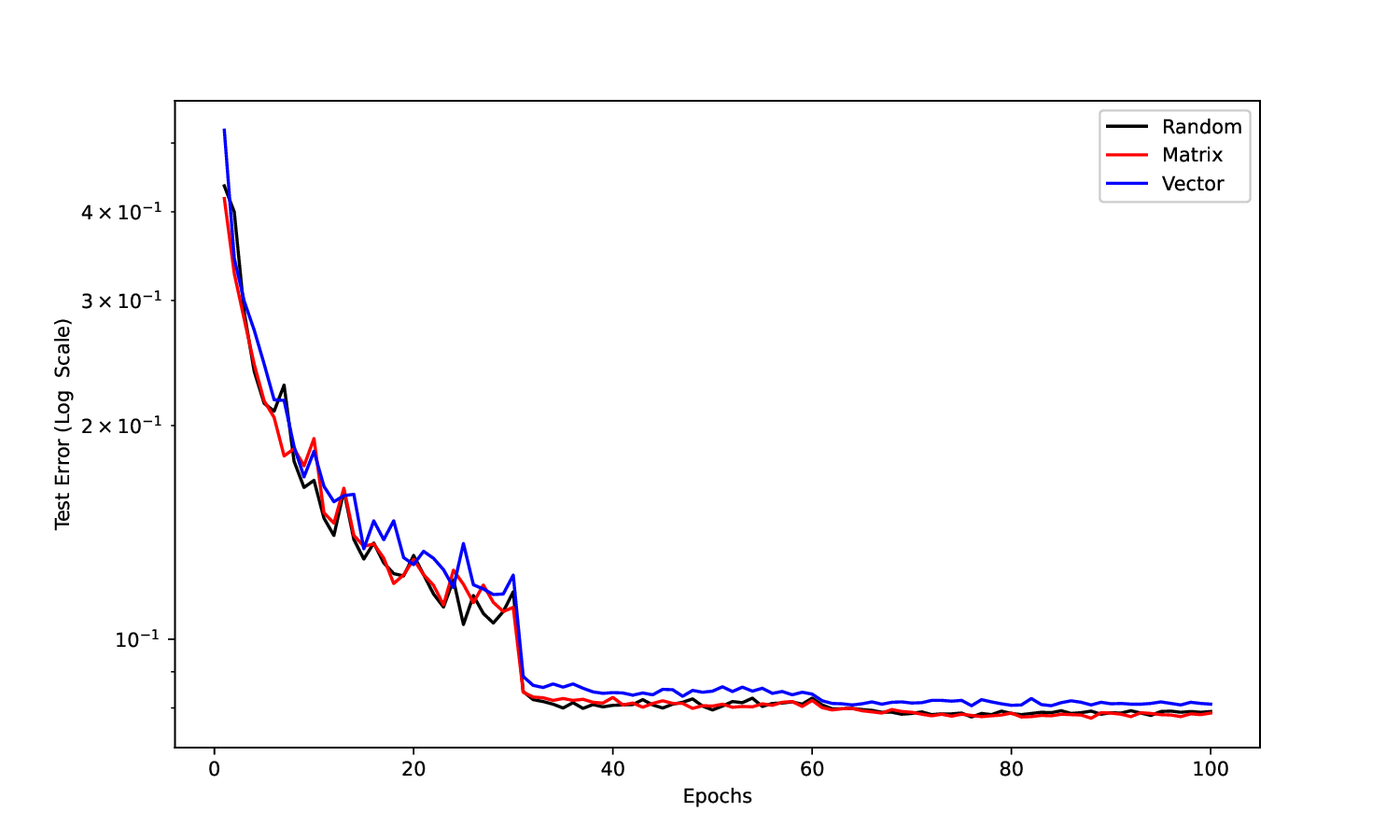}}
\caption{Performance comparison on CIFAR-10 dataset using different random optimization methods }
\label{fig:CIFAR-10 re}
\end{figure}
\section{Conclusions}
\label{sec:conclusions}
This paper considers a quadratic problem  with assignment constraints \cref{eq:ori-pro}, a combinatorial optimization problem widely applicable in engineering and machine learning. We relax the binary constraints into box constraints and introduce an \(\ell_{1/2}\)-regularization strategy for this problem, resulting in problem \cref{eq:l1/2-pro}. Critically, we rigorously prove the exact equivalence between the regularized problem \cref{eq:l1/2-pro} and the original binary problem \cref{eq:ori-pro} and any KKT point of problem \cref{eq:l1/2-pro} is guaranteed to be a feasible binary assignment matrix for problem \cref{eq:ori-pro} when the sparsity penalty parameter is larger than a threshold. We employ a variable-splitting technique and obtain problem \cref{eq:l1/2-proxy}. Problem \cref{eq:l1/2-pro} and \cref{eq:l1/2-proxy} share same global solutions and KKT points. Then we adopt the ADMM algorithm to compute the KKT points of problem \cref{eq:l1/2-proxy}. A comprehensive theoretical analysis is provided to validate the algorithm's performance. Under the assumption that dual variable sequence is bounded, we establish the theorem that ADMM algorithm converges to a KKT point of problem \cref{eq:l1/2-proxy}. Furthermore, beyond a certain threshold of the regularization parameter, ADMM algorithm can converges a local minimizer of problem \cref{eq:l1/2-pro}, which is also a local minimizer of problem \cref{eq:ori-pro}. Under additional assumption of strict complementary, we conclude that the algorithm achieves finite-step termination. In contrast to most existing convergence analyses of nonconvex ADMM, 
our framework covers the challenging case that simultaneously involve nonconvexity, nonsmoothness, coupling terms, and full constraints on all variables. In terms of applications, we demonstrate that problem~\cref{eq:ori-pro} can effectively model MMD-based mini-batch selection. Numerical experiments on synthetic and real datasets verify the feasibility and convergence behavior of the proposed ADMM algorithm. 


\newpage
\clearpage 
\bibliographystyle{siamplain}
\bibliography{references}
\end{document}